\documentclass[a4paper]{article}
\usepackage{my_package}

\title{Fermionic Gaussian free field structure in the Abelian sandpile model and uniform spanning tree}

\date{\today}
\author[1]{Leandro Chiarini}
\author[2]{Alessandra Cipriani}
\author[1]{Alan Rapoport}
\author[1]{Wioletta M. Ruszel}

\affil[1]{Utrecht University, Budapestlaan 6, 3584 CD Utrecht, The Netherlands\\
\url{l.chiarinimedeiros@uu.nl, a.rapoport@uu.nl, w.m.ruszel@uu.nl}}
\affil[2]{UCL, Department of Statistical Sciences, 1-19 Torrington Place, London, WC1E 7HB, UK\\
\url{a.cipriani@ucl.ac.uk}}

\begin{document}

\maketitle

\begin{abstract}
In this paper we rigorously construct a finite volume representation for the height-one field of the  Abelian sandpile model and the degree field of the uniform spanning tree in terms of the \emph{fermionic Gaussian free field}. This representation can be seen as the lattice representation of a free symplectic fermion field.
It allows us to compute cumulants of those fields, both in finite volume and in the scaling limit, including determining the explicit normalizing constants for fields in the corresponding logarithmic field theory. 
Furthermore, our results point towards universality of the height-one and degree fields, as we prove that the scaling limits of the cumulants agree (up to constants) in the square and triangular lattice. We also recover the equivalent scaling limits for the hypercubic lattice in higher dimensions, and discuss how to adapt the proofs of our results to general graphs.
\end{abstract}

\section{Introduction}

\paragraph{Lattice models and (log)-conformal field theories.}
Lattice models from statistical mechanics have been successfully used to describe macroscopic properties of interacting systems and model critical phenomena by specifying their random microscopic interaction. One of the major breakthroughs in theoretical physics was the development of conformal field theory (CFT) (see~\cite{BelPol} and~\cite{DiFrancesco} for general references). This theory is based on the postulate that many critical lattice models of two-dimensional statistical mechanics have conformally invariant scaling limits.  

To be able to understand the CFT structure emerging from scaling limits of such lattice models, one often resorts to (possibly non-commutative) algebras instead of probability measures. Those algebras allow us to describe quantities of interest, referred to as  ``observables'' rather than random variables. We then measure these observables via suitable operators, referred to as ``states'', which are analogous to expectations in probability theory. The construction of such algebras and states is highly non-trivial.
The construction, characterization, and understanding of these algebras and operators are common challenges in the study of CFTs.

Solvable CFTs can be studied in terms of representations of the Virasoro algebra, which is a complex Lie algebra. It allows us to identify continuum theories and universality classes corresponding to particular lattice models. It can lead to exact formulae for scaling limits of correlations, partition functions, and critical exponents.

Prominent examples of predictions obtained by the CFT approach are, among others, conformal invariant scaling limit and critical exponents of the Ising model~\citep{BelPol,BelPol2}, crossing probabilities in percolation~\citep{Cardy}, and cluster growth in diffusion limited aggregation~\citep{DLACFT}. 

Some of the drawbacks of using CFT methods to understand critical lattice models are that often they are non-rigorous and that they use functions of local operators, making them less appropriate for analyzing global quantities. In fact, it is a known long-standing open question whether there is a direct link connecting CFTs and lattice models (see~\cite{Ito}). There has been progress in recent years to rigorously establish the predictions obtained from the CFT approach regarding scaling limits of lattice models, leading to important mathematical contributions. Yet, the full picture is still far from well-understood. 

Let us mention a few examples of important mathematical contributions.
In \cite{HoSm, ChelHon}, the authors related correlation functions of the Ising model and the relevant correlations of the associated CFT. 
In~\cite{Camia1, Camia2}, the authors identified the scaling limit of the magnetization field at the critical/near-critical point. 
A Virasoro representation of the Gaussian free field as the simplest example of a Euclidean field theory was proved in~\citet{kang2013gaussian}.
Furthermore, the concept of \emph{fermionic observables} in the context of discrete complex analysis, put forward by Kenyon and Smirnov, led to proving conformal invariance of the height function of the dimer model~\citep{Ken}, critical percolation~\citep{Smir1}, Ising model~\citep{Smir2}, or very recently the construction of conformal field theory at lattice level~\citep{HonKyVik}.
In this last article, the authors give a rigorous link between CFT local fields and lattice local fields for the Gaussian free field and the Ising model.

Several years after the introduction of CFTs,~\cite{Gurarie} observed logarithmic singularities in correlation functions of certain CFTs. 
Typically, those logarithmic CFTs (logCFTs) describe the critical behavior of lattice models at second order phase transitions. The Virasoro representation in this case involves pairs of fields, a primary field and its logarithmic partner.
LogCFTs are much less classified and understood, contrary to ordinary CFTs, both from the theoretical physics point of view and even less from the mathematical point of view. 
The reasons behind this fact are that computations are significantly harder due to the non-local features of the theory, and that it is not known what a generic logCFT looks like. The simplest, but still highly non-trivial, logCFT which is understood from the theoretical physics point of view is the {\it free symplectic fermion theory} with central charge $c=-2$ (see also~\cite{Gaberdiel,Kausch,Gurarie}).

Models which display logarithmic singularities in the correlation functions and are conjectured to belong to this class include percolation~\citep{Cardy:1999zp}, self-avoiding random walks~\citep{DuSa}, random spanning forests \citep{Iva,Liu}, the (double) dimer model~\citep{DimRue,AdaCar}, and the Abelian sandpile model~\citep{PiRue}. We refer to the overview article~\cite{HoPaVi} for further references.

\paragraph{Abelian sandpile model and logCFTs.}
The Abelian sandpile model (ASM) was introduced by~\cite{Bak1987Jul} as a prototype of a model displaying self-organized criticality. This refers to the property of a model to drive itself into a critical state, characterized by power-law behavior of certain observables such as the avalanche size, without fine-tuning any external parameter.

The ASM on a finite graph is defined as follows. Assign to each vertex of the graph an integer number (height configuration), modeling the amount of ``grains of sand'' associated to it.
This random dynamical system runs in two steps: in the first step, a vertex of the graph is picked uniformly at random and a grain of sand is added to it, increasing its height by one. In the second step, the vertices that are unstable (that is, those which bear strictly more grains than the graph degree) topple, sending out one grain to each neighbor.
This procedure is repeated until all vertices are stable again. Grains sent outside the graph to so-called ``sinks'' are lost. 

The Abelian property yields that the order of topplings does not matter, and the system will eventually reach a stable configuration thanks to dissipation at the boundary. 
This Markov chain has a known unique stationary measure: the uniform measure on the set of recurrent configurations.
Thanks to the famous \emph{burning bijection} introduced by~\cite{DharMaj}, which relates recurrent configurations for the ASM with spanning trees, one can determine the stationary measure and several quantities of the model explicitly. 

Let us highlight that many large-scale aspects of the ASM are heavily graph-dependent.
As a consequence of the burning bijection, we have that the invariant measure of the ASM depends on the type of graph where the ASM is defined, its size, and dissipative sites. 
Furthermore, the scaling limit of the stable configuration starting from a deterministic single source is driven by the \emph{sandpile PDE}, whose solutions are lattice-dependent \citep{levine2016apollonian, levine}.
On the contrary, as we will discuss later, we prove that cumulants of the height-one field of the ASM coincide (up to a multiplicative constant) in both the square and triangular lattice, which is a strong indicator of universality.

What makes the model critical is the occurrence of long-range correlations at all scales, resulting from possible avalanches of topplings. Although the ASM is very simple, it is challenging to treat mathematically due to its non-locality and many questions still remain open. 

In a series of papers from the last two decades (see~\citet{MogRaj, RueLog,PiRue,MahieuRuelle,JengPirouxRuelle,Jeng}) and the review~\cite{RuelleLarge} with references therein, physicists have made significant contributions from the theoretical physics viewpoint to understand how and which logCFT emerges from a stochastic lattice model like the ASM, and attempted to identify it. In particular, they have computed the height probabilities of the ASM on different lattices (e.g. Euclidean lattices, triangular and hexagonal lattices), 2-, 3- and 4-point height correlation functions, studied different bulk and boundary observables, and identified logarithmic pairs of several fields.

Typically, an ansatz from the theoretical physics point of view to validate (or discard) the continuum theory is to take an educated guess for a field theory $\Phi$ and test it, in the sense that, if a lattice observable in one point $x$ converges in the scaling limit towards $\Phi(x)$, then the corresponding correlation functions must also converge to the equivalent field theoretic correlators. The more identities are tested positively, the higher the convinction that the proposed theory is the correct one.

More precisely, consider the ASM on a rescaled lattice with mesh $\eps$ such that the points $x_i^\eps\to x_i$ as $\eps\to 0$, $i=1,\,\ldots, \,n$ for general $n \in \mathbb{N}$. Formally, we expect that in the continuum scaling limit, the field $\Phi$ satisfies
\begin{equation}\label{eq:introFT}
\lim_{\eps\to 0 }\eps^{-\sum_i D_i}\langle O^\eps(x_1^\eps)\cdots O^\eps(x_n^{\eps})\rangle_{\text{lattice}}=\langle \Phi_1(x_1)\cdots \Phi_n(x_n) \rangle_{\text{field theory}},
\end{equation}
where the $D_i$'s are related to the scale dimension of the field, and $O^\eps$ is a local observable of the ASM on the lattice. 

The first educated guess for the ASM was to consider the free symplectic fermion theory mentioned above. Its Lagrangian is given by~\cite[Equation 27]{RuelleLarge}
\begin{equation}\label{eq:contfGFF}
S=\frac{1}{\pi}\int \d z\,\d{\overline z} \,\partial_z\theta\,\partial_{\overline z}\tilde\theta ,
\end{equation}
where $\theta,\,\tilde \theta$ are a pair of free, massless, Grassmannian scalar fields. In fact, the authors in \cite{JengPirouxRuelle} (see also~\cite{brankov,MahieuRuelle}) showed that the bulk dissipation field and the height-one field of the ASM can be realized as a logarithmic pair of a symplectic fermion theory. The height-one field can be identified as \cite[Equation 108]{RuelleLarge}
\begin{equation}\label{eq:Phi_CFT}
    \Phi_\theta = - C \left(\partial_z \theta \,\partial_{\overline z}\, \tilde{\theta} + \partial_{\overline z} \theta \,\partial_{z} \tilde{\theta}\right) ,
\end{equation}
where the constant $C$ can be computed explicitly and is equal to the probability of the ASM to have height 1 on a generic site on $\mathbb{Z}^2$. In~\citet[Section 8]{RuelleLarge}, it is argued that higher height fields are not described by the symplectic fermion theory, discarding the ansatz that this theory is the right logCFT to describe the ASM in the scaling limit. As we will see later, the height-one lattice field can be associated to a local observable of the uniform spanning tree. 
This ceases to be the case for higher heights, as the observables become non-local, which strongly suggests a qualitative different field theory.

To the best of authors' knowledge, the only mathematically rigorous work supporting this picture can be found in~\cite{durre, durrethesis}. There the author determines the scaling limit of the cumulants of the height-one field on a domain in $\mathbb{Z}^2$ and proves conformal covariance with scale dimension $2$.

\paragraph{Aim, results of the paper and novelty.}
The aim of this paper is to make a first step towards rigorously constructing logarithmic conformal theory on the lattice level and work towards establishing a link between logCFTs and certain types of critical lattice models displaying logarithmic singularities in the correlation functions.
Our starting point is to establish rigorously scaling limits of the sort \eqref{eq:introFT} for a class of observables.

Let $V$ be a finite set of vertices on some graph $\mathcal{G}$.
We will consider two lattice fields, namely the height-one field of the ASM $(h(v))_{v\in V}$, which is the indicator that at a site $v$ there is only one grain of sand, and the field $(\mathcal{X}_v)_{v\in V}$, given by the (normalized) degree of a site in the uniform spanning tree (UST).

For a set of generators $\{\psi_v,\, \bpsi_v\}_v$ of a suitable (real) Grassmann algebra, we will define the lattice fields
$\textbf X$ and $\textbf Y$ which are products over discrete gradients of the Grassmannian variables defined by

\begin{equation}\label{eq:X_vintro}
X_v = \frac{1}{\deg_\mathcal{G}(v)} \sum_{i=1}^{\deg_\mathcal{G}(v)} \nabla_{e_i}\psi(v)\nabla_{e_i}\bpsi(v)
\end{equation}
and
\begin{equation}\label{eq:Y_vintro}
Y_v =  \prod_{i=1}^{\deg_\mathcal{G}(v)}(1- \nabla_{e_i}\psi(v)\nabla_{e_i}\bpsi(v)).
\end{equation}
We will evaluate those fields according to an operator
$\langle \cdot \rangle$, referred to as the {\it fermionic Gaussian free field} (fGFF) state on the lattice. 
This operator can be seen as the natural counterpart to the expectation of the usual Gaussian free field on the lattice, but whose ``spins'' take values in a Grassmannian algebra rather than on $\mathbb{R}$.
The fGFF is a type of Gaussian integral over Grassmannian variables, and it is a known tool to treat Matrix-Tree-type theorems (see for instance~\cite{abdesselam,BCHS,Caracciolohyperforest}).

In the following, we will highlight the most relevant results in a qualitative way and discuss their implications. We defer the precise formulations of the results to Section \ref{sec:results}. 

The first result (Theorems~\ref{thm:height-fgff} and~\ref{thm:trig-discrete}) is a representation of the $n$-point function of the height-one field of the ASM in terms of functions of Grassmannian variables. We will prove it for points on the Euclidean lattice $\mathbb{Z}^d$ and triangular lattice $\textbf{T}$, but the statement holds on general lattices. Therefore call $L$ such a lattice.

\begin{theorem}
The height-one field can be represented in terms of fermionic variables as
\[\E\left(\prod_{v\in V}h(v)\right)= \left \langle\prod_{v\in V} X_v Y_v \right\rangle
\]
where $\langle\cdot\rangle$ is the fGFF state.
\end{theorem}

The second relevant result is the scaling limit of the cumulants and determination of the constant $C$ appearing in \eqref{eq:Phi_CFT} in terms of permutations of double gradients of the continuum harmonic Green's function. This is stated in Corollary \ref{cor:const_cum}, which follows from the more general scaling limit for the field $\textbf X \textbf Y$ stated in Theorem~\ref{thm:main_cum3_cont} for $\mathbb{Z}^d$ and in Theorem~\ref{thm:trig-cont} item~\ref{item2_limit_triang} for the triangular lattice.

\begin{theorem}\label{thm:intro_sl}
Let $U \subset \R^d$, $V\subseteq U$ be a set of points, and $U$, $V$  satisfy ``nice'' properties. Consider the renormalized graph $\mathcal G_\eps\coloneqq U/\eps\cap L.$ There exists an explicit constant $C_L$ such that the joint cumulants $\kappa$ of the height-one field scale as
\[
    \lim_{\eps\to 0}\eps^{-\frac{1}{2}\sum_v \deg_{\mathcal{G}_\eps}( v)}\kappa\left(h^\eps(v):\,v\in V\right) = - C_L^{|V|} \sum_{\sigma } \sum_{\eta } \prod_{v} \partial_{\eta(v)}^{(1)}\partial_{\eta(\sigma(v))}^{(2)} g_U\left(v,\, \sigma(v)\right),
\]
where $h^\eps$ is a suitable embedding of $h$ in $\mathcal G_\eps$ and $g_U$ is the harmonic Green's function on $U$ with Dirichlet boundary conditions, $\sigma$'s are certain permutations of $V$ and $\eta$'s are directions of derivations.
\end{theorem}

The third result concerns the scaling limit of the cumulants of the renormalized degree field of the uniform spanning tree $\mathsf T$. The precise statement can be found in Corollary~\ref{cor:degree_field} and Theorem~\ref{thm:trig-cont}. 

We define
\[
\mathcal X_v\coloneqq \frac{\deg_{\mathsf T}(v)}{\deg_{\mathcal G}(v)}.
\]

\begin{theorem}\label{thm:intro_degField}
Let $V$, $U$ and $\mathcal G_\eps$ as in Theorem~\ref{thm:intro_sl}. There exists an explicit constant $c_L <0$ such that the joint cumulants $\kappa$ of the normalized degree field scale to
\begin{align*}
 \lim_{\eps\to 0} \eps^{- \frac{1}{2}\sum_v \deg_{\mathcal G_\eps}(v) } \kappa\left(\mathcal{X}^\eps_v:\,v\in V\right) &= \lim_{\eps\to 0} \eps^{- \frac{1}{2}\sum_v \deg_{\mathcal G_\eps}(v) } \kappa\left( X_v:\,v \in V\right) \\
& =
	- c_L^{|V|}  \sum_{\sigma} \sum_{\eta } \prod_{v} \partial_{\eta(v)}^{(1)}\partial_{\eta(\sigma(v))}^{(2)} g_U\left(v,\, \sigma(v)\right),
\end{align*}
for a suitable embedding $\mathcal{X}^\eps$ of the field $\mathcal{X}$ in $\mathcal G_\eps$.
\end{theorem}

In the following let us discuss the implications of our results.
We prove in Theorem~\ref{thm:main_cum2} and Corollary~\ref{cor:lim_cum_X} that the scaling limit of the cumulants of the height-one field in $\mathbb{Z}^2$ (see  Theorem~\ref{thm:intro_sl}) match those of the field $-C_2\, \textbf X$, where $C_2>0$ is equal to $C_L$ on $\mathbb{Z}^2$. The field $\textbf X$ can be interpreted as the {\it lattice realization of a free symplectic fermion} and is responsible for the structure of the field. Furthermore, we deduce that the auxiliary field $\textbf Y$ will act as a multiplicative constant and can be thought of as a \emph{lattice correction term}. 
The field $-\textbf X$ is an ideal candidate to validate the claim that the height-one field is represented as a free symplectic fermion theory. Note that the constant from \eqref{eq:Phi_CFT} and our $C_L$ match as well. Remark that we do not determine the limiting field which bears those specific cumulants. See also the next section on open problems and further discussions.

Theorem~\ref{thm:intro_degField} suggests that the degree field of the UST is described by the free symplectic fermion theory $\textbf X$ as well, which hints towards a positive answer to the question posed by \cite{Liu} that the UST can be described by a logCFT. Furthermore, the same theorem implies that the symplectic theory of the height-one field in the ASM and degree field are not the same, yet very similar.

Another interesting consequence of Theorems~\ref{thm:intro_sl} and \ref{thm:intro_degField} is the aspect of universality for both fields. Although we computed the scaling limits of the cumulants for two different lattices, the methods are general and applicable for other types of lattices. Universality was conjectured already in \cite{HuLin, RuelleTriang}, where the authors proved that the critical exponents of avalanche size probabilities of the ASM are the same for several lattices as well.

Finally, we will prove in Theorem~\ref{thm:scaling} that the fields $\textbf X \textbf Y$ and $-\textbf X$ viewed as random distributions converge to a non-trivial limit (not white noise) on test functions with disjoint supports, using the same scaling as in Theorems~\ref{thm:intro_sl} and~\ref{thm:intro_degField}. Note however that one would obtain white noise if we relax the assumption that the test functions have disjoint support and adjust the scaling to $\eps^{\frac14\sum_v \deg(v) }$ instead of $\eps^{\frac12\sum_v \deg(v) }$, analogously to \citet[Theorem 2]{cipriani2022properties} (see also \citet[Theorem 5]{kassel2015transfer}).

To the best of the authors' knowledge, there are several novel aspects in the present paper. Firstly, we give a rigorous representation of the height-one field in the ASM and degree field of the UST in terms of Grassmannian variables at the lattice level in $\mathbb{Z}^d$ and $\textbf T$. This suggests a certain lattice realization of a free symplectic fermion theory.
A similar concept appears in the article~\cite{Moghimi-Araghi2005Jul}, where the authors use Grassmannian Gaussians with a different ``covariance'' to derive a formal treatment of the limiting theory for the height-one field.

Secondly, we prove scaling limits and certain universality of the  mentioned fields. This is the first rigorous proof showing convergence of correlations to the analogous continuum correlators for a fermionic system in $d$ dimensions and on the triangular lattice. The proof relies on a careful analysis of the structure of the cumulants, identifying which terms survive the scaling limit and which cancel out. Our analysis is substantially different from the one used in \cite{durre}, which is written out only in $\mathbb{Z}^2$ and not generalizable to different graphs in an obvious manner. Thirdly, our proof technique to analyze cumulants of fermionic fields is very general and robust, and permits to determine the lattice constants $C_L$ and $c_L$ explicitly. Those can be used as multiplicative constants in the definition of the continuum field. We find that the constant $C_L$ for the ASM is explicitly related to the height-one probability on the underlying lattice.

\paragraph{Open problems and further research.}
As already mentioned, the connection between critical lattice models and CFTs is not fully understood even in the simpler setting (without log divergences), in which there is a range of natural candidates for the limiting CFTs. 
In the logCFT setting much less is understood, as fewer explicit candidates are identified.
This paper opens up the avenue for a deeper mathematical investigation of the possible connections between ASM, UST, fGFF and the associated logCFTs. 

As a warm up, one could prove universality of the height-one and degree fields on general lattices and determine a formula for the corresponding constants. Another question is to prove conformal covariance and determine the scale dimension on a general lattice. The scale dimension is $2$ on the square lattice, as a consequence of \citet[Theorem 1]{durre} and \citet[Proposition 1]{cipriani2022properties}.

One mathematical challenge is to define the limiting field of the height-one field or degree field rigorously. An object with Lagrangian~\eqref{eq:contfGFF} seems to require new fundamental tools in order to be rigorously defined, although some Grassmannian stochastic processes do already exist~\citep{Rogers,Gubinelli}. 

Another challenge is to determine the logarithmic fields describing higher heights and general observables of the ASM rigorously.
There is strong evidence~\citep{RuelleLarge} that the logarithmic field describing higher heights in the ASM does not belong to the free symplectic fermion theory~\eqref{eq:contfGFF}. It would be interesting to push our methods in this direction to see whether they can provide some light on the logCFTs describing higher heights. 

The relation to the UST is also a source of another open problem. As stated in~\cite{Liu}, a logarithmic CFT is conjectured to describe the behavior of the UST in the limit. Our result on the degree field seems to support this statement. However, it has to be noted that we are only analyzing the degree for the moment, and therefore other microscopic variables of the model should be considered in order to fully prove or disprove this statement.

\paragraph{Structure of the paper.} We begin our paper setting up notation and defining the main objects of interest in Section~\ref{sec:notation-and-definitions}. In particular, in order to keep the paper self-contained, we will give a primer on fermionic Gaussian free fields and their properties. The main results are given in Section~\ref{sec:results}. Section~\ref{sec:proofs} is devoted to the proofs of the main theorems on the hypercubic lattice, while Section~\ref{sec:towards} presents the proofs of our main results on the triangular lattice and a discussion on how to generalize the results to other lattices.

\paragraph{Acknowledgements.} We thank Roland Bauerschmidt, Rajat Subhra Hazra, Tyler Helmuth and Dirk Schurricht for many valuable and inspiring discussions.

\section{Notation and Preliminaries}\label{sec:notation-and-definitions}
\paragraph{Functions and (Euclidean) sets}
For the rest of the paper, we will work in dimension $d\ge 2$.
We will write $|A|$ for the cardinality of a set $A$. For $n\in \N$, let $[n]$ denote the set $[n] \coloneqq\{1,\,\ldots,\, n\}$.
For $d\geq 2$, we will use $\langle\cdot,\cdot\rangle$ and $\|\cdot\|$ to denote the $\ell^2(\Z^d)$ inner product and norm, respectively. By an abuse of notation we will use the same symbols for the inner product and norm in $L^2(\R^d)$.

We will denote an oriented edge $f$ on the standard $\Z^d$ lattice as the ordered pair $(f^-,\,f^+)$, being $f^-$ the tail and $f^+$ the tip of the edge. Furthermore
\[
    e_i=(0,\,\ldots,\,0,\,\underbrace{1}_{i\text{-th position}},\,0,\,\ldots,\,0),\quad i=1,\,\ldots,\,d
\]
denotes the $d$ standard coordinate vectors of $\R^d$. The $e_i$'s define a natural orientation of edges which we will tacitly choose whenever we need oriented edges (for example when defining the matrix $M$ in~\eqref{eq:M}). The opposite vectors will be written as $e_{d+i} \coloneqq -e_i,\,i=1,\,\ldots,\,d$. The collection of all $e_i $ for $i\in \{1,\,\ldots,\,2d\}$ will be called $E_o$, where $o$ is the origin. By abuse of notation, but convenient for the paper, if $f=(f^-,\,f^-+e_i)$ for some $i\in[2d]$, we denote by $-f$ the edge $(f^-,\,f^- - e_i)$; that is, the reflection of $f$ over $f^-$. 

Call $A \subseteq \R^d$ a countable set. For every $v \in A$, denote by $E_v$ the set $E_o+v$, and let $E(A) = \bigcup_{v\in A} E_v$. 

Let $U\subseteq \R^d$ and $e \in E_o$. For a function $f:\,U\to\R^d$ differentiable at $x$ we define $\partial_e f(x)$ as the directional derivative of $f$ at $x$ in the direction corresponding to $e$, that is 
\begin{equation*}
	\partial_e f(x) 
=
	\lim_{t \to 0^{+}} \frac{f(x+t  e)-f(x)}{t}.
\end{equation*}
%
%
Likewise, when we consider a function in two variables $f: \mathbb{R}^d \times \mathbb{R}^d \longrightarrow \mathbb{R}$, we write then $\partial^{(j)}_ef(\cdot,\cdot)$ to denote the directional derivative in the $j$-th entry, $j=1,\,2$.
For $A,\,B\subseteq \R^d$ let $\dist (A,\,B) \coloneqq \inf_{(x,y)\in A\times B} \|x-y\|$.

\paragraph{Graphs and Green's function} 
As we use the notation $(u,\,v)$ for a directed edge we will use $\{ u,\, v\}$ for the corresponding undirected edge.

For a finite (unless stated otherwise) graph $\mathcal{G}=(\Lambda,\, E)$ we denote the degree of a vertex $v$ as $\deg_\mathcal{G}(v) \coloneqq \left|\{u \in \Lambda:\, u \sim v\}\right|$, where $u\sim v$ means that $u$ and $v$ are nearest neighbors.

\begin{definition}[Discrete derivatives]
For a function $f:\,\Z^d\to \R$ its discrete derivative $\nabla_{e_i} f$ in the direction $i=1,\,\ldots,\,2d$ is defined as
\[
\nabla_{e_i} f(u)\coloneqq f(u+e_i)-f(u),\quad u\in\Z^d.
\]
Analogously, for a function $f:\,\Z^d\times \Z^d\to\R$ we use the notation $\nabla^{(1)}_{e_i}\nabla^{(2)}_{e_j}f$ to denote the double discrete derivative defined as
\[
\nabla^{(1)}_{e_i}\nabla^{(2)}_{e_j}f(u,\,v)\coloneqq f(u+e_i,\,v+e_j)- f(u+e_i,\,v)-f(u,\,v+e_j)+f(u,\,v),
\]
for $u,\,v\in \Z^d$, $i,\,j=1,\,\ldots,\,2d$.
\end{definition}

\begin{remark}\label{rem:direct-vs-not-directed}
Throughout this article, we will work with directed edges to encode discrete derivatives in observables of interest.
However, whenever we are referring to graphs, the Laplacian operator, and probabilistic models on graphs (for example sandpiles or spanning trees), we will always think of $\mathbb{Z}^d$ as an undirected graph.
In fact, one can show that all of the fields ${\bf X}$ and ${\bf Y}$ (defined at the end of Subsection~\ref{subsec:fGFF}) remain the same if one changes the direction of any/all edges.
\end{remark}
\begin{definition}[Discrete Laplacian on a graph] 
We define the (unnormalized) {\it discrete Laplacian} with respect to a vertex set $\Lambda\subseteq \Z^d$ as
\begin{equation}\label{eq:laplacian}
    \Delta(u,\,v) \coloneqq
    \begin{dcases}
    \hfil -2d & \text{if } u=v,\\
    \hfil 1 & \text{if } u \sim v,\\
    \hfil 0 & \text{otherwise.}
    \end{dcases}
\end{equation}
where $u,\,v \in \Lambda$ and $u \sim v$ denotes that $u$ and $v$ are nearest neighbors. For any function $f:\Lambda \to \mathcal A$, where $\mathcal A$
is an algebra over $\R$, we define 
\begin{equation}\label{eq:laplacian_on_function}
	\Delta f(u) \coloneqq 
	\sum_{v \in \Lambda}\Delta(u,\,v) f(v) =
	\sum_{v \sim u} (f(v)-f(u)).
\end{equation}
\end{definition}
Note that we define the function taking values in an algebra because we will apply the Laplacian both on real-valued functions and functions defined on Grassmannian algebras, which will be introduced in Subsection~\ref{subsec:primer}.

We also introduce $\Delta_\Lambda\coloneqq (\Delta(u,\,v))_{u,\,v\in \Lambda}$, the restriction of $\Delta$ to $ \Lambda$.
Notice that for any lattice function $f$ we have that for all $u \in \Lambda$, 
\begin{equation}\label{eq:Delta_Lam}
\Delta_\Lambda f(u) = \sum_{v \in \Lambda} \Delta(u,\,v) f(v) = \Delta f_\Lambda (u)
\end{equation}
where $f_\Lambda$ is the lattice function given by $ f_\Lambda(u) \coloneqq f(u) \1_{u\in\Lambda}$.

The interior boundary of a set $\Lambda$ will be defined by\label{boundary}
\[
\partial^{\mathrm{in}}\Lambda \coloneqq \{u\in\Lambda:\,\exists \,v\in\Z^d\setminus\Lambda:\,u\sim v\},
\]
and the outer boundary by
\[
\partial^{\mathrm{ex}}\Lambda \coloneqq \{u\in\Z^d\setminus\Lambda:\,\exists \,v\in\Lambda:\,u\sim v\}.
\]
We also consider the interior of $\Lambda$, given by $\Lambda^{\mathrm{in}} \coloneqq \Lambda \setminus \partial^{\mathrm{in}}\Lambda$. The notation $\partial U$ will also be used to denote the boundary of a set $U \subseteq \R^d$.

\begin{definition}[Discrete Green's function]
Let $x \in \Lambda$ be fixed. The Green's function $G_\Lambda(x,\cdot)$ with Dirichlet boundary conditions is defined as the solution of
\[
\begin{dcases}
-\Delta_{\Lambda} G_{\Lambda}(x,\,y) = \delta_x(y) &\text{ if } y\in \Lambda, \\
\hfil G_{\Lambda}(x,\,y) = 0 &\text{ if } y\in \partial^{\mathrm{ex}} \Lambda,
\end{dcases}
\]
where $\Delta_\Lambda$ is defined in \eqref{eq:Delta_Lam}. 
\end{definition}

\begin{definition}[Infinite volume Green's function, {\cite[Sec. 1.5-1.6]{lawler2013intersections}}]\label{def:Green-function-discrete}
    With a slight abuse of notation we denote by $G_0(\cdot,\cdot)$ two objects in different dimensions:
    \begin{itemize}
    \item $d\ge 3$: $G_0$ is the solution of
    \[
    \begin{dcases}
        -\Delta G_0(u,\,\cdot)=\delta_u(\cdot)\\
        \lim_{\|x\|\to\infty}G_0(u,\,x) = 0
    \end{dcases}, \quad u\in \Z^d.\\
    \]
	
    \item $d=2$: $G_0$ is given by
    \[
    G_0(v,\,w)=\frac14 a(v-w),\quad v,\,w\in \Z^2,
    \]
    where $a(\cdot)$ is the potential kernel defined as
    \[
    a(x) = \sum_{n=0}^\infty \left[\P_o(S_n=o)-\P_o(S_n=x)\right], \quad x\in \Z^2,
    \]
    and $\{S_n\}_{n\geq0}$ is a random walk on the plane staring at the origin and $\P_o$ its probability measure.
    \end{itemize}
\end{definition}

\label{pg:ghost}Points in $\partial^{\text{ex}}\Lambda$ will be later identified with $g$, which we call \emph{the ghost vertex}, to define a graph with wired boundary conditions.
Define $\Lambda^{g} \coloneqq \Lambda \cup \{g\}$, and consider another Laplacian given by
\[
\Delta^g(u,\,v) \coloneqq
\begin{dcases}
    \hfil \Delta_\Lambda(u,\,v) & \, u,\,v \in \Lambda, \\
    \hfil \left|\{w \in \partial^\text{ex}\Lambda:\, u\sim w\}\right| & u \in \partial^\text{in} \Lambda,\, v=g,  \\
    \hfil \hfil \left|\{w \in \partial^\text{ex}\Lambda:\, v\sim w\}\right| &   u=g ,\, v \in \partial^\text{in} \Lambda, \\
    \hfil -\left|\partial^\text{in}\Lambda\right| & u=v=g, \\
    \hfil 0 & \text{otherwise}.
\end{dcases}
\]
As said, this is equivalent to looking at the graph given by $\Lambda \cup \partial^\text{ex}\Lambda$ and identifying all elements of $\partial^\text{ex}\Lambda$ as the ghost.

\paragraph{Cumulants} We now give a brief recap of the definition of cumulants and joint cumulants for random variables. 
\begin{definition}[Cumulants of a random variable]\label{def:cum_single}
Let $X$ be a random variable in $\mathbb{R}^d$ with all moments finite. The cumulant generating function $K(t)$ for $t\in \R^d$ is defined as 
\[
K(t) \coloneqq \log \left ( \mathbb{E}\left[e^{ \inpr{t}{X}} \right]\right ) = \sum_{m=0}^{+\infty} \sum_{\substack{m_1,\,\dots,\,m_d:\\m_1+\cdots +m_d=m}} \frac{(-1)^m}{m_1!\cdots m_d!} \kappa_m(X) t^{m_1+\cdots + m_d}.
\]
The cumulants $(\kappa_m(X))_{m\in\N}$ are obtained from the power series expansion of the cumulant generating function $K(t)$. 
\end{definition}

Let $V=\{v_1,\,\ldots,\,v_n\} $ be a set, $n\in\N$, and $(X_{v_i})_{i=1}^n$ be a family of random variables in $\mathbb{R}^d$ with all finite moments.

\begin{definition}[Joint cumulants of a family of random variables]\label{def:cum_mult}
The joint cumulant generating function $K:\,\R^n\to \R$ of the family $(X_{v_i})_{i=1}^n$ is defined as
\[
K(t_1,\,\ldots,\,t_n) \coloneqq \log \left ( \mathbb{E} \left[e^{\sum_{j=1}^n t_j \, X_{v_j}} \right ] \right ).
\]

The cumulants can be defined as Taylor coefficients of $K(t_1,\,\ldots,\,t_n)$. 
The joint cumulant $\kappa(X_v:\, v\in V)$ of $(X_v)_{v\in V}$ can be computed as 
\[
\kappa(X_v:\, v\in V) \coloneqq \sum_{\pi \in \Pi(V)} (|\pi|-1)! \,(-1)^{|\pi|-1} \prod_{B\in \pi} \mathbb{E} \left[\prod_{v\in B} X_v \right],
\]
where $\Pi(V)$ is the set of partitions of $V$.
\end{definition}
Let us remark that, by some straightforward combinatorics, it follows from the previous definition that
\[
\mathbb{E} \left[\prod_{v\in V} X_v \right] = \sum_{\pi \in \Pi(V)} \prod_{B\in \pi} \kappa(X_v:\, v\in B) .
\]

\paragraph{Abelian sandpile model and uniform spanning trees}
Let $\Lambda\subseteq \Z^d$ be finite.
We identify all vertices of $\Z^d\setminus \Lambda $ into the ghost vertex $g$.
Call a height configuration a map $\rho:\,\Lambda \to \N$.
The \emph{Abelian sandpile model} (ASM) with toppling matrix $(\Delta(u,\,v))_{u,\,v\in \Lambda^g}$ is a discrete-time Markov chain on $S_{\Lambda} = \prod_{v\in \Lambda} \{1,\,\dots,\,2d\}$.
Given $\rho \in S_{\Lambda}$ the Markov chain evolves as follows: choose uniformly at random a site $w\in \Lambda$ and increase the height by one.
For a site $w$ which is unstable, that is, such that $\rho(w)> 2d$, we decrease the height at $w$ by $2d$ and increase the height at each nearest neighbor of $w$ by one.
At the ghost vertex $g$, particles leave the system.
All unstable sites are toppled until we obtain a stable configuration.
It is known that the unique stationary measure of the ASM is the uniform measure on all recurrent configurations $\mathcal{R}_{\Lambda}$ and furthermore that $|\mathcal{R}_{\Lambda}|=\det(-\Delta_{\Lambda})$.
In fact, there is a bijection (the ``burning algorithm'' or ``burning bijection'' by~\citet{DharMaj}) between spanning trees of $\Lambda^g$ and recurrent configurations of $\mathcal{R}_{\Lambda}$.
See also \cite{redig,jarai} for more details.

For $z\in \Lambda$, let $h_{\Lambda}(z)=1_{\{\rho(z)=1\}}$ be the indicator function of having height 1 at site $z\in \Lambda$. 

\begin{definition}[Good set, {\citet[Lemma 24]{durrethesis}}]\label{def:goodset}
We call $A \subseteq \Lambda$ a {\it good set} if  the set $A$ does not contain any nearest neighbors and, for every site $v\in\Lambda\setminus A$, there exists a path $P$ of nearest-neighbor sites in $\Lambda^g$ so that $P$ and $A$ are disjoint.
\end{definition}

\begin{lemma}[{\citet[Lemma 24]{durrethesis}}]\label{lem:char_heightone} 
Let $V\subseteq\Lambda$. The expected value $\E\left[\prod_{v\in V}h_\Lambda(v)\right]$, which is the probability of having height one on $V$ under the stationary measure for the ASM on $\Lambda$, is non-zero if and only if $V$ is a good set.
\end{lemma}

For any finite connected graph $\mathcal G=(\Lambda,\, E)$ we define a {\it spanning tree} $T$ of $\mathcal G$ as a connected subgraph containing every vertex of $\Lambda$ and having no loops. Let us denote by $\mathbb{T}$ the set of all possible spanning trees of $\mathcal G$.
We use ${\bf P}$ to denote the uniform measure on all possible such trees. That is, for $T\in \mathbb{T}$ we have

\[
{\bf P} (T) = \frac{1}{\det(-\Delta_{\Lambda})}.
\]
See e.g.~\citet[Chapter 4]{lyonsperes}.
 
\subsection{A primer on Grassmann variables}\label{subsec:primer}

In this subsection,  we will introduce notions and results about Grassmannian variables and integration. 
We refer to~\cite{abdesselam,Meyer} for further reading.

\begin{definition}[{\citet[Definition 1]{abdesselam}}]\label{def:GrasAlg}
Let $M\in \N$ and $\xi_1,\,\ldots,\,\xi_M$ be a collection of letters.  Let $\R\left[\xi_{1},\,\ldots,\,\xi_{M}\right]$ be the quotient of the free non-commutative algebra $\R\langle\xi_1,\,\ldots,\,\xi_M\rangle$ by the two-sided ideal generated by the anticommutation relations
\begin{equation}\label{eq:anticommute}
\xi_j \xi_j = - \xi_i \xi_j,
\end{equation}
where $i,\,j \in [M]$. We will denote it by $\Omega^{M}$ and call it the Grassmann algebra in $M$ variables. The $\xi$'s will be referred to as {\it Grassmannian variables} or {\em generators}. Due to anticommutation these variables are called ``fermionic'' (as opposed to commutative or ``bosonic'' variables).
\end{definition}

Notice that, due to the anticommutative property, we have that for any variable Pauli's exclusion principle holds~\citep[Proposition 2]{abdesselam}:
\begin{equation}\label{eq:squares-are-zero}
\xi_i^2 = 0, \quad i\in[M].
\end{equation}
An important property for elements of $\Omega^{M}$ is the following (see e.g.~\citet[Proposition A.6]{sportiello}).
\begin{proposition}\label{prop:GrassMon}
    The Grassmann algebra $\Omega^{M}$ is a free $\R$-module with basis given by the $2^{M}$ monomials $\xi_I=\xi_{i_1}\cdots\xi_{i_p}$ where $I=\{i_1,\,\ldots,\,i_p\} \subseteq [M]$ with $i_1<\dots<i_p$. Each element $F \in \Omega^{M}$ can be written uniquely in the form
    \[
    F= \sum_{I \subseteq [M]} a_I \xi_I, \quad a_I \in \R .
    \]
\end{proposition}

Next we will define Grassmannian derivation and integration. 
\begin{definition}[Grassmannian derivation, {\citet[Equation 8]{abdesselam}}]
Let $j \in [M]$. The derivative $\partial_{\xi_j}:\, \Omega^{M} \to \Omega^{M}$ is a map defined by the following action on the monomials $\xi_{i_1} \cdots \xi_{i_p}$. For $I=\{i_1,\,\ldots,\,i_p\}$ one has
\[
    \partial_{\xi_j} \xi_I = \begin{cases}
    (-1)^{\alpha-1} \xi_{i_1} \cdots \xi_{i_{\alpha-1}} \xi_{i_{\alpha+1}} \cdots \xi_{i_p} & \text{if there is }1\le \alpha\le p \text{ such that } i_{\alpha} = j, \\
    \hfil 0 & \text{if } j\notin I.
    \end{cases}
\]
\end{definition}
The following characterization of Grassmanian integration can be found in e.g.~\citet[Equation 2.2.7]{swan}. Grassmann--Berezin integrals on fermionic spaces are completely determined by their values on Grassmann monomials as these form a basis for the space.
\begin{definition}[Grassmannian--Berezin integration]
The Grassmann--Berezin integral on fermionic spaces is defined as
\[
\int F d \xi \coloneqq \partial_{\xi_{M}}\partial_{\xi_{M-1}} \cdots \partial_{\xi_2}\partial_{\xi_1}F, \quad F\in \Omega^{M}.
\]
\end{definition}

On the grounds of this definition, for the rest of the paper Grassmannian--Berezin integrals will be denoted by $\left(\prod_{i=1}^M\partial_{\xi_i}\right) F$.
\begin{definition}\label{def:func-of-grasmmann}
	Let $f:\, \R \to \R$ be an analytic function given by $f(x)= \sum_{k=0}^\infty a_k x^k $, and let $F \in \Omega^{M}$ be an element of the Grassmann algebra.
	We define the composition of an analytic function with an element of the Grassmann algebra, $f(F)$, by
	\begin{equation*}
		f(F)\coloneqq \sum_{k=0}^{\infty}
		a_k F^{k}.
	\end{equation*}
\end{definition}
Notice that the series $f(F)$ is, in fact, a finite sum since the $\xi$'s are nilpotent. An important special case arises when $M$ is even, $M=2m$, and the generators $\xi_1,\,\ldots,\,\xi_M$ are divided into two groups $\psi_1,\,\ldots,\,\psi_m$ and $\bpsi_1,\,\ldots,\,\bpsi_m$, where we think of
each $\psi_i$ as paired with its corresponding $\bpsi_i$. Since we have an even number of generators of the Grassmann algebra, we have that
\[
\prod_{i=1}^m \partial_{\bpsi_i} \partial_{\psi_i} = (-1)^{m(m-1)/2} \left ( \prod_{i=1}^m \partial_{\bpsi_i} \right ) \left ( \prod_{i=1}^m \partial_{\psi_i} \right ),
\]
and this can be identified as being a collection of "complex" fermionic variables in the language of
\citet[Equation (A.60)]{sportiello}. We stress that the notation $\overline\psi$ is only suggestive of complex conjugation and does not have anything to do with complex numbers. We will use bold to denote the collection of Grassmannian variables, for instance $\bm \psi=(\psi_i)_{i=1}^m$.
In particular in the following $\bm \psi $ and $\bm\bpsi$ will be treated as $m\times 1$ vectors.

The next result~\citep[Proposition A.14]{sportiello} computes the integral of so-called Grassmannian Gaussians.
\begin{proposition}[Gaussian integral for “complex” fermions]\label{prop:det}
Let $A$ be an $m\times m$ matrix with coefficients in $\R$.
Then
\[
\left(\prod_{i=1}^m\partial_{\bpsi_i}\partial_{\psi_i}\right)\exp\left((\bm{\psi},\,A\bm{\bpsi})\right)=\det(A).
\]
\end{proposition}

Another result~\citep[Theorem A.16]{sportiello} is the analog of Wick's formula for Grassmannian Gaussians and also will be important to study properties of ``transformed'' normal variables in the fermionic context.

For a given matrix $A = (A_{i,\,j})_{i \in I_0,\, j \in J_0}$, and $I \subseteq I_0,\, J\subseteq J_0$, such that $|I|=|J|$, we write $\det (A)_{IJ}$ to denote the determinant of the submatrix $(A_{i,\,j})_{i \in I,\, j \in J}$. When $I=J$, we simply write $\det(A)_I$.

\begin{theorem}[Wick’s theorem for “complex” fermions]\label{thm:sportiello}
Let $A $ be an ${m\times m}$, $B$ an $ {r\times m}$ and $C$ an $ {m\times r}$ matrix respectively with coefficients in $\R$.
For any sequences of indices $I = 
\{i_1,\,\dots,\,i_r\}$ and $J = \{j_1,\,\dots,\,j_r\}$ in $[m]$ of the same length $r$, if the matrix $A$ is invertible we have
\begin{enumerate}[label=\arabic*.,ref=\arabic*.]
    \item\label{thm_Wick_one} $\displaystyle \left(\prod_{i=1}^m\partial_{\bpsi_i}\partial_{\psi_i}\right) \prod_{\alpha=1}^r  \psi_{i_\alpha}\bpsi_{j_\alpha}\exp\left((\bm{\psi},\,A\bm{\bpsi})\right) = \det(A) \det\left(A^{-\intercal}\right)_{IJ}$,

    \item\label{thm_Wick_two} $\displaystyle \left(\prod_{i=1}^m\partial_{\bpsi_i}\partial_{\psi_i}\right) \prod_{\alpha=1}^r (\psi^T C)_\alpha (B\bpsi)_\alpha\exp\left((\bm{\psi},\,A\bm{\bpsi})\right) = \det(A) \det\left(BA^{-1}C\right)$.
\end{enumerate}
If $|I| \neq |J|$, the integral is $0$.
\end{theorem}

\subsection{The fermionic Gaussian Free Field}\label{subsec:fGFF}

Unless stated otherwise, let $\Lambda\subseteq \Z^d$ be finite and connected in the usual graph sense. We will also consider its wired version $\Lambda^g$ as described on page~\pageref{pg:ghost}. 

\paragraph{Grassmannian algebra on $\Lambda$.}

We construct the (real) Grassmannian algebra $\Omega^{2\Lambda}$ resp. $\Omega^{2\Lambda^g}$ with generators $\{\psi_v,\,\bpsi_v:\,v\in \Lambda\}$ resp. $\{\psi_v,\,\bpsi_v:\,v\in \Lambda^g\}$ as described in Subsection~\ref{subsec:primer}. Note that $\Omega^{2\Lambda}$ is a subset of the algebra $\Omega^{2\Lambda^g}$.\newline

In the following we will now define the fermionic Gaussian free field with pinned and Dirichlet boundary conditions. 
We use here for clarity the notation $\langle f,\,g\rangle_{\Lambda^g}\coloneqq \sum_{v\in\Lambda^g}f_v g_v$.

\begin{definition}[Fermionic Gaussian free field]\hfill
\begin{description}[font=\normalfont\itshape\space]
\item[Pinned boundary conditions.]
The unnormalized fermionic Gaussian free field state on $\Lambda^g$ pinned at $g$ is the linear map $\efp{\cdot}_\Lambda:\, \Omega^{2\Lambda^g} \to \mathbb{R}$ defined as 
\[
	\efp{F}_{\Lambda} \coloneqq \left(\prod_{v\in\Lambda^g}\partial_{\bpsi_v}\partial_{\psi_v}\right) \psi_g \bpsi_g \exp\left(\inpr{\bm{\psi}}{-\Delta^g \bm{\bpsi}}_{\Lambda^g} +\inpr{\delta_{g}}{\bm{\psi}\bm{\bpsi}}_{\Lambda^g}\right)F,\quad F\in \Omega^{2\Lambda^g} .
 \]
\item[Dirichlet boundary conditions.]
The unnormalized fermionic Gaussian free field state with Dirichlet boundary conditions is the linear map $\efp{\cdot}_\Lambda:\, \Omega^{2\Lambda} \to \mathbb{R}$ defined as
\[
\efd{F}_{\Lambda} \coloneqq \left(\prod_{v\in\Lambda}\partial_{\bpsi_v}\partial_{\psi_v}\right)\exp\left(\inpr{\bm{\psi}}{-\Delta_\Lambda \bm{\bpsi}}\right)F,\quad F\in \Omega^{2\Lambda}.
\] 
\end{description}
\end{definition}

We are borrowing here the terminology of ``state'' from statistical mechanics (compare \citet[Definition 3.17]{friedlivelenik} since we do not associate a probability measure to the fermionic fGFFs. Note that we can define the normalized counterpart of the fGFF with pinned boundary conditions as 
\[
	\Efp{F}_{\Lambda}\coloneqq\frac{1}{\zfp_\Lambda}\efp{F}_{\Lambda},
\]
where $\zfp_\Lambda\coloneqq \efp{1}_\Lambda$. The normalization constant can be determined from Theorem~\ref{thm:sportiello}:
\[
 \zfp_\Lambda = \det(-\Delta_\Lambda).
\]
We can also define the normalized expectation $\Efd{\cdot}_\Lambda$ in a similar fashion using Proposition~\ref{thm:sportiello} (noting that the normalization constant $\zfd_\Lambda$ also equals $\det(-\Delta_\Lambda)$).
To avoid cluttering notation, in our definitions we write $\Lambda$ as subindex, even though $\efd{\cdot}_{\Lambda}$, $\efp{\cdot}_\Lambda$ and their normalized counterparts live on $\Lambda^g$.

We will also consider gradients of the generators in the following sense. 

\begin{definition}[Gradient of the generators]
The gradient of the generators in the $i$-th direction of $\R^d$ is given by
 \[
 \nabla_{e_i}\psi(v)=\psi_{v+e_i}-\psi_v,\quad\nabla_{e_i}\bpsi(v)=\bpsi_{v+e_i}-\bpsi_v,\quad v\in \Lambda,\,i=1,\,\ldots,\,2d.
 \]
 \end{definition}
Remember that $e_{d+i} \coloneqq -e_{i}$ for $1 \le i \le d$, as introduced at the beginning of Section~\ref{sec:notation-and-definitions}. 

Although the fermionic setting does not carry the notion of realization of random variables, we interpret the evaluation of the states over observables as expectations, so we can extend the notion of cumulants to the fermionic setting via the analogous expression for usual probability measures. Therefore, we define cumulants as follows.

\begin{definition}[Cumulants of Grassmannian observables]\label{def:cum_mult_Gras}
Let $V\subseteq \Lambda^{g}$. The joint cumulants $\kappa^{\bullet}_{\Lambda}(W_v:\, v\in V)$ of the Grassmannian observables $(W_v)_{v\in V}$ are defined as
\begin{equation}\label{def:cum_Gras}
	\kappa^{\bullet}_\Lambda\left(W_v:\,v\in V\right) = \sum_{\pi\in\Pi(V)} \left(|\pi|-1\right)\!!\, (-1)^{|\pi|-1} \prod_{B\in\pi} \left\langle{\prod_{v\in B} W_v}\right\rangle_\Lambda^\bullet .
\end{equation}
\end{definition}
The $\bullet$ indicates that we are considering states both under the pinned and the Dirichlet conditions. As before one has
\[
 \left\langle{\prod_{v\in V} W_v }\right\rangle_\Lambda^\bullet  = \sum_{\pi \in \Pi(V)} \prod_{B\in \pi} \kappa^\bullet_\Lambda(W_v:\, v\in B).
\]

\begin{example}
Let $v \in \Lambda$ and $F = \psi_v \bpsi_v$. By Theorem~\ref{thm:sportiello} item~\ref{thm_Wick_one} we have that
\[
\begin{split}
\Efd{\psi_v \bpsi_v}_\Lambda &= \frac{1}{\det(-\Delta_{\Lambda})} \det(-\Delta_{\Lambda}) G_{\Lambda}(v,\,v)= G_{\Lambda}(v,\,v).
\end{split}
\]
The two-point function for $v,\,w \in \Lambda$ such that $v\neq w$ is equal to
\[
\Efd{\psi_v \bpsi_v \psi_w \bpsi_w}_\Lambda =  \det \begin{pmatrix} G_{\Lambda}(v,\,v) & G_{\Lambda}(v,\,w) \\
G_{\Lambda}(w,\,v) & G_{\Lambda}(w,\,w) 
\end{pmatrix}=G_\Lambda(v,\,v)G_\Lambda(w,\,w)-G_\Lambda(v,\,w)^2.
\]
In particular, for $v\neq w$, $\kappa_\Lambda^{\mathbf 0}(\psi_v \bpsi_v,\,\psi_w \bpsi_w)=-G_\Lambda(v,\,w)^2 < 0$; that is, we have negative ``correlations", or more precisely, negative joint cumulants of second order.
\end{example}

Let us now introduce the main object of study in this paper.

\begin{definition}[Gradient squared of the generators]\label{def:XandY}
The ``gradient squared'' ${\bf X}=(X_v)_{v\in V}$ of the generators  is defined as

\begin{equation}\label{eq:defX}
    X_v\coloneqq 
    \frac{1}{2d}\sum_{i=1}^{2d}\nabla_{e_i}\psi(v)\nabla_{e_i}\bpsi(v),\quad v\in \Lambda.
\end{equation}
\end{definition}
This ``gradient squared'' will be evaluated under the fGFF state.
We will also need  auxiliary Grassmannian observables ${\bf Y}=(Y_v)_{v\in V}$ defined as
\begin{equation}\label{eq:defY}
Y_v\coloneqq \prod_{i=1}^{2d}\left(1-\nabla_{e_i}\psi(v)\nabla_{e_i}\bpsi(v)\right),\quad v\in \Lambda.
\end{equation}

The reader may have noticed that $X_v,Y_v$ may not be compatible with the state $\Efd{\cdot}_\Lambda$ for points close to the boundary.
However, our setup will allow us to work only ``well inside'' $\Lambda$ so that we can safely work with $X_v$ and $Y_v$ under the Dirichlet state.
In fact, Lemma~\ref{lem:equiv-fggs} is key in this regard, showing that the fermionic Gaussian free fields defined with pinned and Dirichlet boundary conditions respectively agree on observables that only depend on $\Lambda$.


\section{Results}\label{sec:results}

In this section, we will state our main results. The first theorem provides a representation of the expectation of the height-one field of the Abelian sandpile model in terms of the fields ${\bf X}$ and ${\bf Y}$. 

\begin{theorem}\label{thm:height-fgff}
For $V\subseteq \Lambda^{\mathrm{in}}$ a good set as in Definition \ref{def:goodset}, and ${\bf X}$, ${\bf Y}$ defined in Equations \eqref{eq:defX} and \eqref{eq:defY}, we have 
\begin{equation}\label{eq:th1}
\E\left(\prod_{v\in V}h_\Lambda(v)\right)=\Efp{\prod_{v\in V} X_v Y_v}_\Lambda.
\end{equation}
\end{theorem}

\noindent
In the next theorem, we derive a closed-form expression for the joint cumulants of the field ${\bf X}$, together with their continuum scaling limit.

Let $U$ be a smooth, connected, bounded subset of $\R^d$ and define $U_\eps \coloneqq U/\eps \cap \Z^d$.
For any $v\in U$, let $v_\eps$ be the discrete approximation of $v$ in $U_\eps$; that is, $v_\eps \coloneqq \floor{v/\eps}$.
Define also $G_\eps$ as the discrete harmonic Green's function on $U_\eps$ with $0$-boundary conditions outside $U_\eps$, and $g_U$ the continuum harmonic Green's function on $U$ with $0$-boundary conditions outside $U$ (recall Section~\ref{sec:notation-and-definitions}).
We write $X_v^\eps$, $\bf{X}^\eps$ and $Y_v^\eps$, $\bf{Y}^\eps$ to emphasize the dependence of $v$ on $\eps$ whenever $v$ belongs to $U_\eps$. Cyclic permutations without fixed points of a finite set $A$ are denoted as $S_\text{cycl}(A)$.

\begin{theorem}\label{thm:main_cum2}$ $

\begin{enumerate}[label=\arabic*., ref=\arabic*]
\item\label{item1_cum2} Let $n \geq 1$ and let the set of points $V\coloneqq \left\{v_1,\,\dots,\,v_n\right\}\subseteq \Lambda^{\mathrm{in}}$ be a good set. Let $\eta$ denote a map from $ V$ to $E(V)$ such that $\eta(v)\in E_v$ for all $v\in V$. The  joint cumulants of the field $(-X_v)_{v\in V}$ are given by 

\begin{equation}\label{eq:thm:maincum2}
	 \kappa_\Lambda^{\mathbf 0}\left(-X_v:\,v\in V\right) = - \frac{1}{(2d)^n}\sum_{\sigma\in S_{\cycl}(V)} \sum_{\eta:\, V\to\{e_1,\,\dots,\,e_{2d}\}} \prod_{v\in V} \nabla_{\eta(v)}^{(1)}\nabla_{\eta(\sigma(v))}^{(2)} G_\Lambda\left(v,\, \sigma(v)\right) .
\end{equation}

\item\label{item2_cum2} Let $n\ge 2$ and $V\coloneqq \left\{v_1,\,\dots,\,v_n\right\}\subseteq U$ be such that $\dist(V,\,\partial U)>0$. If $v_i \neq v_j$ for all $i\neq j$, then
\begin{align}\label{eq:cum_limit}
    \tilde\kappa_1(v_1,\,\dots,\,v_n)&\coloneqq
	\lim_{\eps\to 0} \eps^{-dn} \kappa_\Lambda^{\mathbf 0}\left(-X_v^\eps:\,v\in V\right) \nonumber
	\\ &=  
	- \frac{1}{d^n}  \sum_{\sigma\in S_{\cycl}(V)} \sum_{\eta:\,V\to \{e_1,\,\ldots,\,e_d\}} \prod_{v\in V} \partial_{\eta(v)}^{(1)}\partial_{\eta(\sigma(v))}^{(2)} g_U\left(v,\, \sigma(v)\right).
\end{align}
\end{enumerate}
\end{theorem}

\begin{corollary}\label{cor:lim_cum_X}
    Let $C_2\coloneqq 2/\pi-4/\pi^2$. Under the assumptions of Theorem~\ref{thm:main_cum2} item  \ref{item2_cum2} we have that
    \[
    \lim_{\eps\to 0}\eps^{-2n}\kappa(-C_2 \, X_v^\eps:\,v\in V)=\lim_{\eps\to 0} \eps^{-2n} \kappa\left(h_{U_\eps}(v_\eps):\,v\in V\right).
    \]
\end{corollary}

\noindent
The expression for the limiting cumulants of the height-one field can be found in~\citet[Theorem 2]{durre}. The proof of the statement follows from comparing this expression to~\eqref{eq:cum_limit} and recalling that cumulants are homogeneous of degree $n$.

Another corollary of Theorem \ref{thm:main_cum2} is that the cumulants of the degree field of a uniform spanning tree are identical to the cumulants of the field ${\bf X}$ with respect to  the fermionic GFF state.

\begin{corollary}\label{cor:degree_field}
Let $V \subseteq \Lambda^{\mathrm{in}}$. Let the average degree field $(\mathcal{X}_v)_{v\in \Lambda}$ in a uniform spanning tree be defined as
\[
\mathcal{X}_v\coloneqq \frac{1}{2d}\sum_{i=1}^{2d}\1_{\{(v,\,v+e_i)\in T\}}= \frac{1}{2d} \deg_T(v),
\]
where $T$ has the law $\mathbf P$ of the uniform spanning tree on $\Lambda$ (with wired boundary conditions).
\begin{enumerate}
\item\label{item1_corDF} Under the assumptions of Theorem~\ref{thm:main_cum2} item \ref{item1_cum2}, we have that
\begin{equation}\label{eq:ferm-degree}
\kappa\left({\mathcal{X}}_v:\,v\in V\right) = \kappa_\Lambda^{\mathbf 0}\left(X_v:\,v\in V\right)
\end{equation}
and therefore
\[
\begin{split}
 \kappa\left({\mathcal{X}}_v:\,v\in V\right) 
=  - \left (\frac{-1}{2d} \right )^n\sum_{\sigma\in S_{\cycl}(V)} \sum_{\eta:\, V\to\{e_1,\,\dots,\,e_{2d}\}} \prod_{v\in V} \nabla_{\eta(v)}^{(1)}\nabla_{\eta(\sigma(v))}^{(2)} G_\Lambda\left(v,\, \sigma(v)\right).
\end{split}
\]
\item\label{item2_cor_DF} Under the assumptions of Theorem~\ref{thm:main_cum2} item~\ref{item2_cum2}, we have that
\[
\lim_{\eps\to 0} \eps^{-dn} \kappa\left(\mathcal{X}_v^\eps:\,v\in V\right) =
	- \left (\frac{-1}{d} \right )^n  \sum_{\sigma\in S_{\cycl}(V)} \sum_{\eta:\,V\to \{e_1,\,\ldots,\,e_d\}} \prod_{v\in V} \partial_{\eta(v)}^{(1)}\partial_{\eta(\sigma(v))}^{(2)} g_U\left(v,\, \sigma(v)\right).
\]
\end{enumerate}
\end{corollary}

Let us now turn our attention to the cumulants of another field, namely ${\bf X \, Y}$. We define the matrix $M_\Lambda$ with entries
\begin{equation}\label{eq:M}
    M_\Lambda(f,\,g) \coloneqq 
	\nabla_{\eta^*(f)}^{(1)}\nabla_{\eta^*(g)}^{(2)}G_{\Lambda}(f^{-},\,g^-) ,\quad f,\,g \in E(\Lambda),
\end{equation}
where $\eta^*(f)\in\{e_1,\,\ldots,\,e_{2d}\}$ is the coordinate direction induced by $f\in E(\Lambda)$ on $f^-$ (in the sense that $\eta^*(f)=e_i$ if $f=(f^-,\, f^-+e_i)$). Note that $M$ is the transfer matrix~\cite[Chapter 4]{lyonsperes}. Hereafter, to simplify notation we will omit the dependence of $M_\Lambda$ on $\Lambda$ and simply write $M$.

\begin{theorem}[Cumulants of ${\bf X \, Y}$ on a graph]\label{thm:main_cum3_discrete}
For $n\geq 1$ let $V\coloneqq \left\{v_1,\,\dots,\,v_n\right\}\subseteq \Lambda^{\mathrm{in}}$ be a good set. For a set of edges $\mcE\subseteq E(V)$ and $v\in V$ denote $\mcE_v\coloneqq \{f\in\mcE:\,f^- = v\}\subseteq E_v$. The $n$-th joint cumulants of the field $(X_vY_v)_{v\in V}$ are given by
\begin{align}\label{eq:thm:maincum3}
    \kappa_\Lambda^{\mathbf 0}\left(X_v Y_v:\,v\in V\right) 
	&= 
	\left(\frac{-1}{2d}\right)^n
	\sum_{\mcE\subseteq E(V):\, |\mcE_v|\ge 1 \ \forall v} K(\mcE) 
	\sum_{\tau\in S_{\co}(\mcE)} \sign(\tau) \prod_{f\in \mcE} M\left(f,\,\tau(f)\right)
\end{align}
where $K(\mcE)\coloneqq\prod_{v\in V}K(\mcE_v)$ and $K(\mcE_v)\coloneqq (-1)^{|\mcE_v|}|\mcE_v|$.
\end{theorem}

\noindent
We refer the reader to Subsection~\ref{subsec:graphs-and-permutations} for a definition of the connected permutations $S_{\co}(A)$ of a finite set $A$.

Before we proceed to the next theorem, we remind the reader of the notation $\det (A)_I$, introduced before Theorem~\ref{thm:sportiello}, to denote the determinant of the matrix $A$ with its rows and columns restricted to the indexes in $I$.

\begin{theorem}[Scaling limit of the cumulants of ${\bf X \, Y}$]\label{thm:main_cum3_cont}
For $n\ge 2$ let $V\coloneqq \left\{v_1,\,\dots,\,v_n\right\}\subseteq U$ be such that $\dist(V,\,\partial U)>0$. If $v_i \neq v_j$, for all $i\neq j$, then
\begin{align}\label{eq:cum_limit3}
    \nonumber 
    \tilde\kappa_2(v_1,\,\dots,\,v_n) &\coloneqq \lim_{\eps\to 0} \eps^{-dn} \kappa_\Lambda^{\mathbf 0}\left(X_v^\eps \, Y_v^\eps:\,v\in V\right) 
	\\ &= - (C_d)^n \sum_{\sigma\in S_{\cycl}(V)} \sum_{\eta:\,V\to \{e_1,\,\ldots,\,e_d\}} \prod_{v\in V} \partial_{\eta(v)}^{(1)}\partial_{\eta(\sigma(v))}^{(2)} g_U\left(v,\, \sigma(v)\right),
\end{align}
where the constant $C_d$ is given by
\begin{align}\label{eq:const_C}
  C_d = \frac{1}{d} \sum_{\mcE_o \subseteq E_o:\, \mcE_o\ni e_1} 
(-1)^{|\mcE_o|}|\mcE_o| \left[ \det\left(\overline M \right)_{\mcE_o\setminus{\{e_1\}}} + \1_{\{\mcE_o\ni-e_1\}} \det\big(\overline M'\big)_{\mcE_o\setminus{\{e_1\}}}\right],
\end{align}
where for any $f,\,g\in E_o$
\[
\overline M(f,\,g) =
\nabla_{\eta^*(f)}^{(1)}\nabla_{\eta^*(g)}^{(2)}G_{0}(f^{-},\,g^-)
\]
and
\begin{equation}\label{eq:Mprime}
\overline M'(f,\,g) =
\begin{cases}
\overline{M}(e_1,\,g) & \text{ if } f=  -e_1, \\
\overline{M}(f,\,g) & \text{ if } f \neq -e_1.
\end{cases} 
\end{equation}

\end{theorem}

\begin{remark}
    As expected in virtue of Theorem~\ref{thm:height-fgff} and Corollary~\ref{cor:lim_cum_X}, for $d=2$ one obtains $$C_2 = \frac{2}{\pi}- \frac{4}{\pi^2}=\pi\,\mathbb{P}(\rho(o)=1)$$ by using the known values of the potential kernel (see for example \citet[p. 149]{spitzer}).
\end{remark}

\begin{remark}\label{rmk:confcov}
In $d=2$ the joint moments of ${\bf X Y}$ are also conformally covariant with scale dimension $2$,
just like the height-one field in \citet{durre} and the squared of the gradient (bosonic) GFF in \citet{cipriani2022properties}.
\end{remark}
The following result establishes the limiting cumulants of the height-one field of the ASM on the hypercubic lattice in any dimension.
\begin{corollary}[Height-one field limiting cumulants in $d\ge 2$]\label{cor:const_cum}
Under the assumptions of Theorem~\ref{thm:main_cum2} item \ref{item2_cum2}, for any $d\ge 2$ we have that
\[
    \lim_{\eps\to 0}\eps^{-dn}\kappa\left(h_{U_\eps}(v):\,v\in V\right) = - (C_d)^n \sum_{\sigma\in S_{\cycl}(V)} \sum_{\eta:\,V\to \{e_1,\,\ldots,\,e_d\}} \prod_{v\in V} \partial_{\eta(v)}^{(1)}\partial_{\eta(\sigma(v))}^{(2)} g_U\left(v,\, \sigma(v)\right),
\]
with $C_d$ as in~\eqref{eq:const_C}.
\end{corollary}

Finally, we can view our field as a distribution acting on smooth test functions in the following sense. For any functions $f_1,\,\dots,\,f_n \in\mathcal C_c^\infty(U)$, we define
\begin{equation}\label{eq:def:kappahat1}
    \widehat{\kappa}_1\left(f_1,\,\dots,\,f_n\right) \coloneqq \int_{U^n} \kappa_\Lambda^{\bzero}\left(-X^\eps_{x_1},\,\ldots,\,-X^{\eps}_{x_n}\right) \prod_{i\in[n]} f_i(x_i) \, \d x_i ,
\end{equation}
respectively
\begin{equation}\label{eq:def:kappahat2}
    \widehat{\kappa}_2\left(f_1,\,\dots,\,f_n\right) \coloneqq \int_{U^n} \kappa_\Lambda^{\bzero}\left(X^\eps_{x_1}Y^\eps_{x_1},\,\ldots,\,X^{\eps}_{x_n}Y^\eps_{x_n}\right) \prod_{i\in[n]} f_i(x_i) \, \d x_i .
\end{equation}

\begin{theorem}\label{thm:scaling}
For any functions $f_1,\,\dots,\,f_n \in \mathcal C^\infty_c(U)$ with pairwise disjoint supports one has
\begin{equation}\label{eq:thm:scaling}
    \lim_{\eps\to0} \eps^{-dn}\, \widehat{\kappa}_1\left(f_1,\,\dots,\,f_n\right) = \int_{U^n} \tilde\kappa_1(x_1,\,\dots,\,x_n) \prod_{i\in[n]} f_i(x_i) \, \d x_i ,
\end{equation}
with $\tilde\kappa_1$ defined as the limit in~\eqref{eq:cum_limit}. An analogous result is obtained exchanging $\widehat \kappa_1$ and $\tilde\kappa_1$ by $\widehat\kappa_2$ and $\tilde\kappa_2$ respectively in~\eqref{eq:thm:scaling}.
\end{theorem}

\begin{remark}\label{rmk:disjoint}
We highlight the importance of the non-overlapping supports of the test functions, which allows the joint cumulants to have a non-trivial limit when scaled by $\eps^{-d}$ for each function.
If instead we remove the assumption of disjoint supports, and we consider $\mathbf{X}^\eps$ respectively $\mathbf{X}^\eps\mathbf{Y}^\eps$ as distributions acting on test functions, we need to scale the field by $\eps^{-d/2}$, in which case we can show that
\[
\lim_{\eps\to 0}\eps^{-\frac{d}{2}}\widehat{\kappa}_1\left(f_1,\,f_2\right)=\inpr{f_1}{f_2}_{L^2(U)},
\]
and that
\[
\lim_{\eps\to 0}\eps^{-\frac{dn}{2}}\widehat{\kappa}_1\left(f_1,\,\dots,\,f_n\right) =0,\quad n\ge 3,
\]
for any $f_1,\,\ldots,\,f_n\in C^\infty_c(U)$, and similarly for $\widehat\kappa_2$.

A first result in this direction was found in \citet[Theorem 5]{kassel2015transfer}, where the authors proved convergence of the law of so-called ``pattern fields'' for the UST. These can be seen as the random variable obtained by applying the height-one field (or a similar local observable of the UST) to a single constant test function $f_1 \equiv 1$.

However, it is possible to extend their convergence to the full setting mentioned above by following the strategy of \citet[Theorem 2]{cipriani2022properties}, where the authors proved convergence  in distribution in an appropriate Besov space.
This proof relies on bounding the joint cumulants of the field, and since we know the decay of the double gradients of the Green's function (see \eqref{eq:recall_G}) the proof can be carried through by performing exactly the same steps.
\end{remark}

\section{Proofs} \label{sec:proofs} 
Before going to the proofs of the main theorems we will give some additional definitions and auxiliary results.

\subsection{Permutations, graphs and partitions}\label{subsec:graphs-and-permutations}

In this subsection, we introduce more notation stated in Theorems~\ref{thm:main_cum2} and~\ref{thm:main_cum3_discrete}. 
\paragraph{General definitions.}\label{par:gen_perm} Let $\Lambda$ be a finite and connected (in the usual graph sense) subset of $\mathbb{Z}^d$ and $V \subseteq \Lambda$ be a good set according to Definition~\ref{def:goodset}. As $V$ is good, notice that every edge in $E(V)$ is connected to exactly one vertex in $V$.

Recal that for any finite set $A$ we denote the set of permutations of $A$ by $S(A)$.
Furthermore, we write $S_{\cycl}(A)$ to denote the set of \emph{cyclic} permutations of $A$ (without fixed points). Finally recall the set $\Pi(A)$ of partitions of $A$.

We will use the natural partial order between partitions \citep[Chapter 2]{taqqu}; that is, given two partitions $\pi$, $\tilde{\pi}$, we say that $\pi \le \tilde{\pi}$ if for every $B \in \pi$ there exists a $\tilde{B} \in \tilde{\pi}$ such that $B \subseteq \tilde{B}$.
If $\pi \le \tilde{\pi}$, we say that $\pi$ is finer than $\tilde{\pi}$ (or that $\tilde{\pi}$ is coarser than $\pi$). If $\sigma \in S(V)$, we denote as $\pi_\sigma \in \Pi(V)$ the partition given by the disjoint cycles of $\sigma$. This is the finest partition such that $\sigma(B) = B$ for all $B \in \pi_\sigma$.
\\

Fix $A\subseteq E(V)$ such that $E_v \cap A \neq \emptyset$ for all $v \in V$, i.e. we have a set of edges with at least one edge per vertex of $V$.
Let $\tau \in S(A)$ be a permutation of edges in $A$.

\paragraph{Permutations: connected and bare.}\label{par:edge_ver}

We define the multigraph $V_\tau= \left(V,\, E_\tau(V)\right)$ induced by $\tau$ in the following way. For each pair of vertices $v \neq w$ in $V$,  we add one edge between $v$ and $w$ for each $f \in E_v, f' \in E_w$ such that either $\tau (f)= f'$ or $\tau (f')= f$. If  $v=w$, we add no edge, so $\deg_{V_\tau}(v) \le |E_v|$. 

\begin{definition}[Connected and bare permutations]\label{def:connected-permutations}
	Let $\Lambda \subseteq \mathbb{Z}^d$ finite, $V$ good as in Definition~\ref{def:goodset}, $A\subseteq E(V)$ and 
	$\tau \in S(A)$ be given.
	\begin{itemize}
		\item We say that $\tau$ is \emph{connected} if the multigraph $V_\tau$ is a connected multigraph.
  
		\item We say that $\tau$ is \emph{bare} if it is connected and $\deg_{V_\tau}(v)=2$ for all $v\in V$ (it is immediate to see that the latter condition can be replaced by $|E_\tau(V)|=|V|$). 
	\end{itemize}
	We will denote by $S_{\co}(A)$ the set of connected permutations in $S(A)$, and by $S_{\bare}(A)$ the set of bare permutations. See Figures~\ref{fig:connected} and~\ref{fig:bare} for some examples, where the mapping $\tau(f)=f'$ is represented via an arrow $f\to f'.$
\end{definition}

\begin{figure}[ht!]
\begin{minipage}{0.45\textwidth}
    \centering
    \includegraphics[scale=.6]{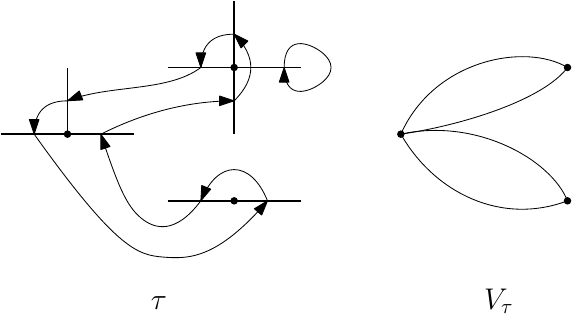}
    \caption{A connected permutation $\tau$ on edges and the multigraph $V_\tau$ associated to it, in $d=2$.
    Notice that this permutation is \textbf{not} bare.}
    \label{fig:connected}
\end{minipage}\hfill
\begin{minipage}{0.45\textwidth}
    \centering
    \includegraphics[scale=0.6]{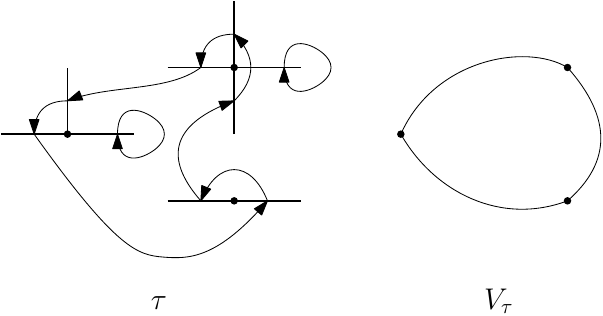}
    \caption{A bare permutation $\tau$ on edges and the multigraph $V_\tau$ associated to it, in $d=2$.}
    \label{fig:bare}
\end{minipage}
\end{figure}

For $\tau$ bare we have, by definition, that for each $v$ there are exactly two edges $f,\,f^\prime \in A$ (possibly the same) such that $\tau (f') \not \in E_v$ and $\tau^{-1} (f) \not \in E_v$. We will refer to this as $\tau$ enters $v$ through $f$ and exits $v$ through $f^\prime$. Therefore, for any bare permutation $\tau \in S_{\bare}(A)$, we can define an induced permutation on vertices $\sigma = \sigma_\tau \in S_{\cycl}(V)$ given by $\sigma(v)=w$ if there there exists (a unique) $f \in E_v$ and $f' \in E_w$ such that $\tau(f)=f'$. Figure~\ref{fig:bare_to_sigma} shows an example in $d=2$.

\begin{figure}[ht!]
    \centering
    \includegraphics[scale=.6]{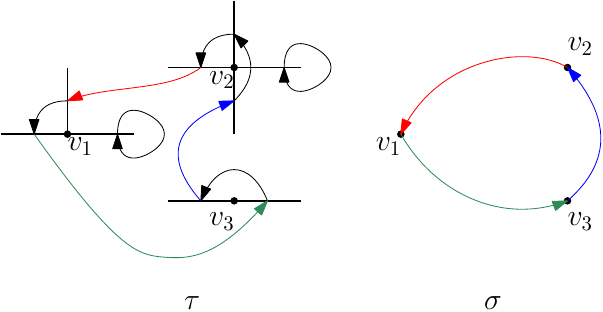}
    \caption{A bare permutation $\tau$ on edges and the induced permutation $\sigma$ on points, in $d=2$.}
    \label{fig:bare_to_sigma}
\end{figure}

Any permutation of $\tau \in S(A)$ induces a partition $\pi_\tau$ on $A$ given by the disjoint cycles in $\tau$. Likewise, the partition $\pi^V_\tau \in \Pi(V)$  given by the connected components of $V_\tau$ gives the finest partition on $V$ such that $\tau \left(\cup_{v \in B}  A_v\right) = \cup_{v \in B} A_v$ for all $B \in \pi^V_\tau$. If $\tau$ is connected, then $\pi^V_\tau = V$.

\subsection{Auxiliary results}
We start this section with a result that states $\efd{\cdot}_{\Lambda}$ and $\efp{\cdot}_{\Lambda}$ coincide over observables with support inside of $\Lambda$. 

\begin{lemma}\label{lem:equiv-fggs}
	 For all $F \in \Omega^{2\Lambda}$ we have that
	\begin{equation}\label{eq:equiv-fggs}
		\efd{F}_\Lambda 
		=
		\efp{F}_{\Lambda}.
	\end{equation}
\end{lemma}

\begin{proof}
For clarity we will distinguish between the inner product $\inpr{\cdot}{\cdot}_\Lambda$ and $\inpr{\cdot}{\cdot}_{\Lambda^g}$ with a subscript. We start by rewriting the exponent $\inpr{ \psi}{ -\Delta^g \bpsi }_{\Lambda^g}$ and isolating the terms with either $\psi_g$ or $\bpsi_g$. 
\begin{align*}
\inpr{ \psi}{ -\Delta^g \bpsi }_{\Lambda^g}
=
	\inpr{\psi}{ -\Delta_\Lambda \bpsi}_{\Lambda}
-
	\sum_{u \in \Lambda} \Delta^g(u,\,g) \left(\psi_g \bpsi_u + \psi_u \bpsi_g\right)
+   \underbrace{\left(1+\sum_{u \in \Lambda} \Delta^g(u,\,g)\right)}_{\eqqcolon c_g} \psi_g \bpsi_g.
\end{align*}
Notice that each of the elements of the sum above commutes with all other elements.
We can then compute $\exp\left(\inpr{\psi}{ -\Delta^g \bpsi}_{\Lambda^g}\right) $ with Definition~\ref{def:func-of-grasmmann} as %

\begin{multline*}
\e^{ \inpr{ \psi}{ -\Delta^g \bpsi }_{\Lambda^g}}
 = 
\sum_{k =0}^\infty
	\frac{1}{k!}
	\left(	
	\inpr{\psi}{ -\Delta_\Lambda \bpsi}_{\Lambda}
	 -
	 	\sum_{u \in \Lambda} \Delta^g(u,\,g) \left(\psi_g \bpsi_u + \psi_u \bpsi_g\right)
	+ c_g \psi_g \bpsi_g\right)^k
\\  = 
	\sum_{k =0}^\infty
	\frac{1}{k!} \sum_{\substack{k_1,\,k_2,\,k_3 \ge 0 \\ k_1 + k_2 + k_3 = k} } \frac{k!}{k_1!k_2!k_3!} \Big(\inpr{\psi}{ -\Delta_\Lambda \bpsi}_{\Lambda}\Big)^{k_1} \left(	-\sum_{u \in \Lambda} \Delta^g(u,\,g) \left(\psi_g \bpsi_u + \psi_u \bpsi_g\right) \right)^{k_2} 
	\left(c_g \psi_g \bpsi_g\right)^{k_3}.
\end{multline*}
Multiplying this sum by $\psi_g \bpsi_g$ and using \eqref{eq:squares-are-zero}, we obtain
\begin{equation*}
	\psi_g \bpsi_g	
	\e^{ \inpr{ \psi}{ -\Delta^g \bpsi }_{\Lambda^g}}
 = 
	\psi_g \bpsi_g	
	\sum_{k =0}^\infty
	\frac{1}{k!}
		\left(\inpr{\psi}{ -\Delta_\Lambda \bpsi}_{\Lambda}\right)^k
 = 
	\psi_g \bpsi_g	
	\e^{ \inpr{ \psi}{-\Delta_\Lambda \bpsi }_{\Lambda} } .
\end{equation*}
Now, provided that $F \in \Omega^{2\Lambda}$, and in particular that $\psi_g,\,\bpsi_g$ do not appear in $F$, we have that 
\begin{equation*}
	\efp{F}_{\Lambda}
	= 
	\left(\prod_{v\in\Lambda^g}\partial_{\bpsi_v}\partial_{\psi_v}\right)
	\psi_g \bpsi_g	
	\e^{ \inpr{ \psi}{ -\Delta_\Lambda \bpsi}_{\Lambda} } 
	F
    = 
	\left(\prod_{v\in\Lambda}\partial_{\bpsi_v}\partial_{\psi_v}\right)
	\e^{ \inpr{ \psi}{ -\Delta_\Lambda \bpsi }_{\Lambda} }
	F
	= \efd{F}_{\Lambda}.\qedhere
\end{equation*}
\end{proof}

The previous lemma yields a consequence which will be of great importance for our work.

\begin{corollary}\label{cor:equal_mom}
Let $V = \{v_1,\,\ldots,\,v_k\} \subseteq \Lambda^{\mathrm{in}}$ and $\eta(\cdot),\,\eta'(\cdot)$ be any two maps $V\to\{e_1,\,\ldots,\,e_{2d}\}$. Then
\[
\efp{\prod_{j=1}^k\nabla_{\eta(v_j)}\psi(v_j)\nabla_{\eta'(v_j)}\bpsi(v_j)}_\Lambda=\efd{\prod_{j=1}^k\nabla_{\eta(v_j)}\psi(v_j)\nabla_{\eta'(v_j)}\bpsi(v_j)}_\Lambda.
\]
Recalling that we denote the cumulants with respect to $\Efp{\cdot}_\Lambda$ as $\kappa^{\textbf{\textup p}}_\Lambda(\cdot)$ and similarly for $\kappa^{\textbf{\textup 0}}_\Lambda(\cdot)$, with the same notation above we have
\[
\kappa^{\textbf{\textup p}}_\Lambda\left(\nabla_{\eta(v_j)}\psi(v_j)\nabla_{\eta'(v_j)}\bpsi(v_j):\,j=1,\,\ldots,\,k\right) = \kappa^{\textbf{\textup 0}}_\Lambda\left(\nabla_{\eta(v_j)}\psi(v_j)\nabla_{\eta'(v_j)}\bpsi(v_j):\,j=1,\,\ldots,\,k\right).
\]
\end{corollary}

\noindent
The proof is a simple consequence of the previous lemma and the fact that, if $V \subseteq \Lambda^{\mathrm{in}}$, then for all $v \in V$ we have that $\nabla_{e_i} \psi(v),\, \nabla_{e_i} \bpsi(v) \in \Omega^{2\Lambda}$.

\subsection{Proof of Theorem \ref{thm:height-fgff}}

The proof of Theorem \ref{thm:height-fgff} will consist of three steps.
In the first step (Proposition~\ref{thm:swan_UST_fGFF}), we will relate the probability of having a certain set of edges in a uniform spanning tree with the fGFF state evaluated on fermionic Grassmannian observables. The statement can be found in \citet[Corollary B.3]{BCHS} and we will give its full proof to stress the normalization factor.

Then in the second step we can demonstrate that these states can be written as a determinant of a matrix containing double gradients of the Green's function (Lemma \ref{lem:joint-moments-grad}) and finally, to obtain \eqref{eq:th1} we note that the height-one field can be expressed as the average of certain fermionic observables in the state $\Efd{\cdot}_\Lambda$.

For a set of oriented edges $S$, we abbreviate $\zeta_S$ as
\[
    \zeta_S \coloneqq \prod_{f \in S} \big(\psi_{f^+}-\psi_{f^-}\big)\big(\bpsi_{f^{+}}-\bpsi_{f^-}\big) = \prod_{f\in S}\nabla_{f}\psi(f^-)\nabla_{f}\bpsi(f^-),
\]
where $\nabla_{f}$ is an abuse of notation for $\nabla_{\eta^*(f)}$, being $\eta^*$ as in~\eqref{eq:M}. Note that $\zeta_{f}=\zeta_{-f}$ for any oriented edge $f$. Therefore, we can consider $\zeta$ as defined on {\it unoriented} edges altogether. For the same reason we can also write $\nabla_{f}\psi(f^-)\nabla_{f}\bpsi(f^-)$ without pinning down the exact orientation of $f$.
\begin{proposition}\label{thm:swan_UST_fGFF}
Let $\mathcal G=(\Lambda,\, E)$ be a finite graph. For all subsets of edges $S\subseteq E$
\begin{equation}\label{eq:UST-grad-ferm}
    {\bf P}(T:\, S\subseteq T) = \Efp{\zeta_S}_{\Lambda}.
\end{equation}
\end{proposition}

\noindent
The following lemma will be useful to prove \eqref{eq:UST-grad-ferm}.

\begin{lemma}\label{lem:joint-moments-grad}
Let $S=\{f_1,\,\ldots,\,f_k\}\subseteq E$ be edges such that all their endpoints belong to $\Lambda^{\mathrm{in}}$. 
Then
\begin{equation}\label{eq:moments-into-dets} 
	\efp{\prod_{j=1}^k\nabla_{f_j}\psi(f^{-}_j)\nabla_{f_j}\bpsi(f^{-}_j)}_\Lambda = 
	\efp{ \zeta_S }_\Lambda
	=
	\det (-\Delta_\Lambda)
	\det (M)_S.
\end{equation}
A simple consequence of \eqref{eq:moments-into-dets} is
\begin{equation*}
	\Efp{\zeta_S}_\Lambda
	=
	\det (M)_S. 
\end{equation*}
\end{lemma}

\begin{proof}[Proof of Lemma~\ref{lem:joint-moments-grad}]
Firstly, notice that, due to Corollary~\ref{cor:equal_mom}, we can substitute $\efp{\cdot}_{\Lambda}$ by $\efd{\cdot}_{\Lambda}$ in the left-hand side of \eqref{eq:moments-into-dets}.
It is straightforward to show that $\zeta_S$ is commuting, so that in fact
\begin{align}
\efd{\zeta_S}_\Lambda &=
\left(\prod_{i=1}^n\partial_{\bpsi_i}\partial_{\psi_i}\right) 
\exp\left(\inpr{\bm{\psi}}{-\Delta_\Lambda\bm{\bpsi}}\right)
 \zeta_S  \nonumber\\
&=\left(\prod_{i=1}^n\partial_{\bpsi_i}\partial_{\psi_i}\right) 
\left(\prod_{i=1}^k\nabla_{f_i}\psi(f_i^{-})\nabla_{f_i}\bpsi(f_i^{-})\right)
\exp\left(\inpr{\bm{\psi}}{-\Delta_\Lambda\bm{\bpsi}}\right).\label{eq:aux_one}
\end{align}
Next, observe that
\begin{equation*}	
	\left(\nabla_{f_i}\psi(f^{-}_i):\,i=1,\,\ldots,\,k\right) =
	{\bm\psi}^T\widetilde{C},
	\quad
	\left(\nabla_{f_i}\bpsi(f^{-}_i):\,i=1,\,\ldots,\,k\right) = \widetilde{B}{\bm\bpsi},
\end{equation*}
where $\widetilde B=\widetilde C^T$ and $\widetilde C$ is a $ |\Lambda|\times k$ matrix such that the column corresponding to the $i$-th point is given by 
\[ 
\widetilde C(\cdot,\,i) = (0,\,\dots,\,0,\,-1,\,0,\,\dots,\,0,\,1,\,0,\,\dots,\,0)^T , 
\]
with the $-1$ (resp. $1$) located at the $f^{-}_i$-th position (resp. that of point $f^{+}_{i}$).
Therefore,
\begin{align}   \label{eq:aux_two}
	\eqref{eq:aux_one}=
	\left(\prod_{i=1}^n\partial_{\bpsi_i}\partial_{\psi_i}\right) 
	\left(\prod_{j=1}^k\left({\bm \psi^T \widetilde C}\right)_j \left(\widetilde B\bm\bpsi\right)_j\right)
	\exp\left(\inpr{\bm{\psi}}{-\Delta_\Lambda\bm{\bpsi}}\right).
\end{align}
The matrix $-\Delta_\Lambda$ has inverse given by the Green's function $G_\Lambda(\cdot,\cdot)$. 
The lemma now follows from item \ref{thm_Wick_two} of Theorem~\ref{thm:sportiello} and the simple computation
\[
\left(\widetilde B(G_\Lambda) \widetilde C \right) (f,\,g) = M(f,\,g)
\]
for $f,\,g \in S$.
\qedhere
\end{proof}

\begin{proof}[Proof of Proposition~\ref{thm:swan_UST_fGFF}]
	The proof follows from the previous Lemma \ref{lem:joint-moments-grad} and the Matrix-Tree Theorem:
	\begin{equation*}
    {\bf P}(T:\, S\subseteq T) =
		\det (M)_S,
	\end{equation*}
	see \citet[Equation (4.5)]{lyonsperes}. 
\end{proof}
Once we have established the previous connections between uniform spanning trees and our fermionic variables, we turn our attention to a connection between the height-one field and the uniform spanning tree.
\begin{lemma}\label{lem:height1}
Let $\mathcal{G}=(\Lambda,\, E)$ be a graph and $V\subseteq \Lambda^{\mathrm{in}}$ a good set (see Definition~\ref{def:goodset}).
Let $\eta:\, V  \to E(V)$ be a function such that $\eta(v)$ assigns to every $v\in V$ an edge incident to $v$. Let $\eta(V)$ be its image. 
Then the height-one field of the Abelian sandpile model satisfies
\begin{equation}\label{eq:height1}
\mathbb{E} \left ( \prod_{v\in V} h_{\Lambda}(v)\right ) = {\bf P}\left( e \notin T \ \forall \, e\in E(V) \setminus \eta(V)\right).
\end{equation}
\end{lemma}

\begin{proof}
Recall that the burning bijection maps every recurrent configuration of the sandpile model to a spanning tree of the graph. The height-one configuration on the good set $V$ can be considered as a special case of a minimal sandpile configuration (see \citet[Definition 1]{jarwer}).
By \citet[Lemma 4]{jarwer} there is a burning sequence in the burning algorithm that burns all vertices in $\Lambda \setminus V$ before burning any vertex in $V$. We can use this burning sequence to understand what the tree associated to a configuration $\rho$ on $\Lambda$ such that $\rho(v)=1$, $v\in V$, looks like.

Denote by $\mathcal{G}_V$ the subgraph constructed from $\mathcal{G}$ by wiring all edges incident to $V$ into a ghost site (usually called ``sink'' in the ASM language) and removing all loops.
Fix a spanning tree $t_0$ of $\mathcal{G}_V$.
Since $V$ is a good set and thus its vertices have distance at least two between each other, we know that the subgraph $\mathcal{G}_V$ consists of $|V|$ vertices, each connected to the sink by $2d$ edges.
We know that $t_0$ must contain one of those edges for each vertex. For each such $t_0$, we define the map $\eta:\,V \to E(V)$ such that $\eta(v)$ is the only edge in $E(V)\cap t_0$ which is incident to $v$.

Hence, denoting by $T_V$ the edges of $T$ with one endpoint in $V$, we obtain
\[
\mathbb{P}(\rho:\, \rho(v)=1 , \ \forall \, v\in V) = {\bf P}(T: T_V=\eta(V)).
\]
Since this probability does not depend on the choice of $t_0$ we can write
\[
\mathbb{P}(\rho:\, \rho(v)=1, \ \forall \,v\in V) = {\bf P}(T:\, T_V = \eta(V)) = {\bf P}\big(T: \, T \cap (E(V)  \setminus \eta(V))=\emptyset\big),
\]
which yields the result by the Matrix-Tree Theorem.
The fact that the above probability does not depend on $\eta$ follows from \citet[Theorem 1]{jarwer}.
\end{proof}

We will combine the two previous observations in the sequel and prove relation \eqref{eq:th1}.

\begin{proof}[Proof of Theorem~\ref{thm:height-fgff}]
Let us first prove the theorem for $V=\{o\}$, the origin (the choice of the origin is made only to avoid heavy notation as we will see), and then the general case $|V|>1$. Consider the function $\eta:\, \{o\} \to E_o$ such that $\eta(o)= o+e_1$. 
By Lemma \ref{lem:height1} we have that
\[
\mathbb{P}(h_{\Lambda}(o)=1) = {\bf P}(e_2\notin T,\,\dots,\,e_{2d} \notin T) ={\bf P}(e_2\notin T,\,\dots,\,e_{2d} \notin T,\, e_1\in T). 
\]
By the inclusion-exclusion principle we can write
\begin{multline}
\mathbb{P}(h_{\Lambda}(o)=1) = {\bf P} (e_1 \in T) - \sum_{i=2}^{2d} {\bf P}(e_i \in T,\, e_1 \in T) +\\ \sum_{2\le i_1\neq i_2\le 2d} {\bf P}(e_{i_1} \in T, e_{i_2} \in T,\, e_1 \in T) - \cdots - {\bf P}(e_1 \in T,\,\dots,\, e_{2d} \in T).
\end{multline}
We rewrite each summand above in terms of fermionic variables by using Proposition \ref{thm:swan_UST_fGFF}. Slightly abusing notation, we call $e_i$ also the point $o+e_i$. Therefore, we obtain
\begin{multline}
\mathbb{P}(h_{\Lambda}(o)=1) = \Efp{(\psi_{o}-\psi_{e_1})(\bpsi_{o}-\bpsi_{e_1})}_\Lambda - \sum_{i=2}^{2d} \Efp{(\psi_{o}-\psi_{e_1})(\bpsi_{o}-\bpsi_{e_1})(\psi_{o}-\psi_{e_i})(\bpsi_{o}-\bpsi_{e_i})}_\Lambda \\
+ \cdots - \Efp{\prod_{i=1}^{2d}(\psi_{o}-\psi_{e_i})(\bpsi_{o}-\bpsi_{e_i})}_\Lambda.\label{eq:exp_oneminus}
\end{multline}
Now we use the formula
\[
\prod_{i=1}^n\left(1-a_i\right)=\sum_{A\subseteq [n]}(-1)^{|A|}\prod_{j\in A}a_j,\quad a_i\in\R,\,n\in\N
\]
in the right-hand side of~\eqref{eq:exp_oneminus} to conclude that
\begin{align*}
    \mathbb{P}(h_{\Lambda}(o)=1)
    &= \Efp{(\psi_{o}-\psi_{e_1})(\bpsi_{o}-\bpsi_{e_1})\prod_{i=2}^{2d}\left[1-(\psi_{o}-\psi_{e_i})(\bpsi_{o}-\bpsi_{e_i})\right]}_\Lambda\\
    &=\Efp{(\psi_{o}-\psi_{e_1})(\bpsi_{o}-\bpsi_{e_1})\prod_{i=1}^{2d}\left[1-(\psi_{o}-\psi_{e_i})(\bpsi_{o}-\bpsi_{e_i})\right]}_\Lambda.
\end{align*}
 In the last equality, we are using that $\left[(\psi_{u}-\psi_{v})(\bpsi_{u}-\bpsi_{w}) \right ]^2=0$ by the anticommutation property \eqref{eq:anticommute} for all $u,\,v$.

Recall that, by Lemma \ref{lem:height1}, the above probability does not depend on the choice of $\eta(\cdot)$.
Hence, summing over all possible $2d$ functions $\eta:\, \{o\} \to E_o$,  we obtain

\[
    \mathbb{P}(h_{\Lambda}(o)=1) 
	=
	\Efp{\frac{1}{2d}
	\sum_{i=1}^{2d}\left(\psi_{o}-\psi_{e_i}\right)\left(\bpsi_{o}-\bpsi_{e_i}\right) 
	\prod_{j=1}^{2d}\left[1- (\psi_{o}-\psi_{e_j})(\bpsi_{o}-\bpsi_{e_j})\right] }_{\Lambda},
\]
which yields the claim for $V=\{o\}$. With the appropriate change of notation the proof yields the same result for any $o\neq v\in V$.

Let us pass to the general case $|V|>1$. 
The event that, for all $v\in V$,
the edges in $E_v  \setminus \{\eta(v)\}$ are not in the spanning tree $T$ of $G$ and the edge $\{\eta(v)\}$ is in $T$ can be decomposed as
\begin{align*}
  & \bigcap_{v\in V}\left( \left\{\eta(v)\in T \right\}\cap\left(\bigcup_{e\in E_v\setminus\{\eta(v)\}}\{e\in T\}\right)^c\right)=
   \bigcap_{v\in V}\left\{\eta(v)\in T \right\}\cap\left(\bigcup_{e\in E(V)\setminus\{\eta(V)\}}\{e\in T\}\right)^c.
\end{align*}
We obtain, by the above and the same inclusion--exclusion principle in the second equality,

\begin{multline}
{\bf P} \left(\bigcap_{v\in V}\left( \left\{\eta(v)\in T \right\}\cap\left(\bigcup_{e\in E_v\setminus\{\eta(v)\}}\{e\in T\}\right)^c\right)\right) \\
\begin{aligned}
&= {\bf P} \left( \bigcap_{v\in V}\{\eta(v)\in T \}\right) -  {\bf P} \left( \bigcap_{v\in V}\left\{\eta(v) \in T\right\}\cap\bigcup_{e\in E(V)\setminus\{\eta(V)\}}\,\{e\in T\}\right) \\
&= \sum_{S \subseteq E(V) \backslash \eta(V)} (-1)^{|S|}{\bf P} \left(\bigcap_{v\in V}\{\eta(v) \in T\}\cap (S \subseteq T) \right),\label{eq:fermio_height}
\end{aligned}
\end{multline}
where again we sum over the probabilities that the edges of $\eta(V)$ are in the spanning tree $T$ as well as those in $S\subseteq E(V) \backslash \eta(V)$. 
Again, by Proposition \ref{thm:swan_UST_fGFF}, Equation \eqref{eq:fermio_height} becomes

 \begin{align} \label{eq:expansion-XvYv} 
     &\sum_{S \subseteq E(V) \backslash \eta(V)} (-1)^{|S|} \Efp{\prod_{\{r,\,s\} \in \eta(V)}\left(\psi_{r}-\psi_s\right)\left(\bpsi_{r}-\bpsi_s\right)\prod_{\{u,\,w\}\in S}\left(\psi_u-\psi_w\right)\left(\bpsi_u-\bpsi_w\right)}_\Lambda.
\end{align}
We observe that the sets of edges $S$ such that $S \cap \eta(V) \neq \emptyset$ do not contribute  to~\eqref{eq:expansion-XvYv}, since again by the anticommutation relation $\big[\left(\psi_u-\psi_w\right)\left(\bpsi_u-\bpsi_w\right)\big]^2 = 0$ for all $u, \,w \in \Lambda$. Moving the sum into the bracket we obtain
\begin{align}
  \sum_{S \subseteq E(V) } & \Efp{\prod_{\{r,\,s\}\in \eta(V)}\left(\psi_{r}-\psi_s\right)\left(\bpsi_{r}-\bpsi_s\right)\prod_{\{u,\,w\}\in S}(-1)^{|S|}\left(\psi_u-\psi_w\right)\left(\bpsi_u-\bpsi_w\right)}_\Lambda\nonumber\\
  &=\Efp{\prod_{\{r,\,s\}\in \eta(V)}\left(\psi_{r}-\psi_s\right)\left(\bpsi_{r}-\bpsi_s\right)\sum_{S\subseteq E(V) }\prod_{\{u,\,w\}\in S}(-1)^{|S|}\left(\psi_u-\psi_w\right)\left(\bpsi_u-\bpsi_w\right)}_\Lambda\nonumber\\
  &=\Efp{\prod_{\{r,\,s\}\in \eta(V)}\left(\psi_{r}-\psi_s\right)\left(\bpsi_{r}-\bpsi_s\right)\prod_{\{u,\,w\}\in E(V) }\left(1-\left(\psi_u-\psi_w\right)\left(\bpsi_u-\bpsi_w\right)\right)}_\Lambda.\label{eq:same_eta}
\end{align}
Recall that, by Lemma \ref{lem:height1}, the probability above is not depending on the choice of $\eta$.
Therefore, there are $(2d)^{|V|}$ many functions $\eta(\cdot)$ giving the same expression~\eqref{eq:same_eta}. Normalizing the field by $(2d)^{-|V|}$ yields the result.
\end{proof}

\subsection{Proofs of Theorem \ref{thm:main_cum2} and Corollary~\ref{cor:degree_field}}

In the following proof, we will write $\kappa$ instead of $\kappa_\Lambda^{\bzero}$ for simplicity. 

We start with a preliminary lemma which is an immediate consequence of~\citet[Lemma 4]{cipriani2022properties}. This general lemma will be used for the proof of Theorem~\ref{thm:main_cum2} item~\ref{item2_cum2} and will also serve for proving Theorem~\ref{thm:main_cum3_cont}.

\begin{lemma}\label{lem:cip4}
For fixed $\eps>0$ let $V=\{v_\eps,\,v'_\eps\}\subseteq U_\eps$, $E\coloneqq E(U_\eps)$, $f\in E_{v_\eps},\,f'\in E_{v'_\eps}$ such that $\eta^*(f)$ and $\eta^*(f')$ do not depend on $\eps$. Then
    \begin{equation}\label{eq:contlimit_g}
   \lim_{\eps\to 0}\eps^{-d} M\left(f,\,f'\right)= \partial^{(1)}_{\eta^*(f)}\partial^{(2)}_{\eta^*(f')} g_U\left(v,\, v'\right) ,
\end{equation}
where $\eta^*$ is as in~\eqref{eq:M}.
In the right-hand side recall that $\partial_e$, with $e\in E$, denotes a directional derivative.  
\end{lemma}

\begin{proof}[Proof of Theorem \ref{thm:main_cum2}]$ $

\begin{enumerate}[wide, labelindent=0pt, label={\bf Part \arabic*.}, ref=\arabic*]
\item\label{part1} Let $B$ be a subset of $V = \{v_1,\,\dots,\,v_n\}$, where $V$ is a good set. 
From Definition \ref{eq:defX}, Lemma \ref{lem:joint-moments-grad} and Lemma~\ref{lem:equiv-fggs}, we have
\begin{equation}\label{eq:alternative}
   \Efd{ \prod_{v\in B}X_v }_\Lambda =
   \frac{1}{(2d)^{|B|}} \sum_{\substack{\eta:\,B\to E(B)\\ \eta(v)\in E_v \; \forall v}} \Efd{\zeta_{\eta(B)}}_\Lambda
   =\frac{1}{(2d)^{|B|}}  \sum_{\substack{\eta:\,B\to E(B)\\ \eta(v)\in E_v \; \forall v}} \det(M)_{\eta(B)},
\end{equation}
where $\zeta_S$ was defined as $\zeta_S\coloneqq \prod_{e \in S} (\psi_{e^+}-\psi_{e^-})(\bpsi_{e^{+}}-\bpsi_{e^-})$ for some
set of edges $S$. In the above we used that if the set $V$ is good, then any subset $B \subseteq V$ is also good. Hence, any edge in $\eta(B)$ is incident to exactly one vertex of $B$. Equation~\eqref{def:cum_Gras} and the expansion of the determinant in terms of permutations yield
\begin{multline*}
    \kappa\left(X_{v_1},\,\dots,\,X_{v_n}\right) \\
	 =
	\frac{1}{(2d)^{n}} 
	\sum_{\substack{\eta:\,V\to E(V)\\ \eta(v)\in E_v \; \forall v}} \sum_{\pi\in\Pi(V)} \left(|\pi|-1\right)\!!\, (-1)^{|\pi|-1} \prod_{B\in\pi} \sum_{\sigma\in S(B)} \sign(\sigma) \prod_{v\in B} M\left(\eta(v),\,\eta(\sigma(v))\right). 
\end{multline*}
Dropping for the moment the summation over $\eta$'s and the constant, the previous expression reads
\begin{equation}\label{eq:swap}
     \sum_{\pi\in\Pi(V)} \left(|\pi|-1\right)\!!\, (-1)^{|\pi|-1} \prod_{B\in\pi} \sum_{\sigma\in S(B)} \sign(\sigma) \prod_{v\in B} M(\eta(v),\,\eta(\sigma(v))).
\end{equation}

Let $\pi = \pi_1\cdots\pi_{|\pi|}$. We would like to swap the sum over $\sigma\in S(B)$ with the product over $B\in\pi$ in~\eqref{eq:swap}. To do this, we first note that for any function $f$ depending on $B$ and $\sigma$,
\[
    \prod_{B\in\pi} \sum_{\sigma\in S(B)} f(B,\,\sigma) = \sum_{\substack{\sigma_{\pi_i}\in S(\pi_i)\\i\in\left[|\pi|\right]}} \,\prod_{B\in\pi} f(B,\,\sigma) .
\]
In this way, we obtain from~\eqref{eq:swap}
\[
\sum_{\pi\in\Pi(V)} \left(|\pi|-1\right)\!!\, (-1)^{|\pi|-1} \sum_{\substack{\sigma_{\pi_i}\in S(\pi_i),\\i\in\left[|\pi|\right]}} \sign(\sigma) \prod_{B\in\pi} \prod_{v\in B} M(\eta(v),\,\eta(\sigma(v))) ,
\]
where $\sigma = \sigma_{\pi_1}\cdots\sigma_{\pi_{|\pi|}}$.

Before we continue, remember the partial order between partitions introduced in Subsection~\ref{subsec:graphs-and-permutations}.
We will use this to sum over the different partitions.

Let us first sum over $\sigma\in S(V)$.
Any such fixed $\sigma$ can always be uniquely decomposed into $m$ disjoint cyclic permutations without fixed points; that is, $\sigma = \sigma_1\cdots\sigma_m$ for some $m\in[n]$. Let us call $\pi_\sigma$ the partition induced by $\sigma$. Now the sum over elements of $\Pi(V)$ will be turned into a sum over $\Pi_\sigma(V)$, being $\Pi_\sigma(V)$ the set of partitions $\pi'\in\Pi(V)$ such that $\pi_\sigma \leq \pi'$.
In this way we obtain
\[
    \sum_{\sigma\in S(V)} \sign(\sigma) \prod_{v\in V} M(\eta(v),\,\eta(\sigma(v))) \sum_{\pi\in\Pi_\sigma(V)} \left(|\pi|-1\right)\!!\, (-1)^{|\pi|-1} .
\]
Notice that any partition $\pi$ such that $\pi \ge \pi_\sigma$ is given by an arbitrary union of elements of $\pi_\sigma$.
Therefore, there is a 1-to-1 correspondence between the partitions $\Pi_\sigma(V)$ and the partitions of the set $\{1,\,\dots,\,m_\sigma\}$, where $m_\sigma$ is the number of cycles in $\sigma$.   
Furthermore, such correspondence preserves the size of the partitions. So we can write our expression as
\[
    \sum_{\sigma\in S(V)} \sign(\sigma) \prod_{v\in V} M(\eta(v),\,\eta(\sigma(v))) \sum_{\pi\in\Pi\left([m_\sigma]\right)} \left(|\pi|-1\right)\!!\, (-1)^{|\pi|-1} .
\]
With this at hand, let us work with the sum over $\Pi\left([m_\sigma]\right)$.
Notice that for any given function $f$ of $|\pi|$ we know that
\begin{equation}\label{eq:sum-card} 
    \sum_{\pi\in\Pi([n])} f\left(|\pi|\right) = \sum_{k=1}^n g(k,\,n) f\left(k\right) ,
\end{equation}
with $g(k,\,n)$ a multiplicity factor.
That is, as long as $f$ depends only on $|\pi|$ and not on the complete permutation $\pi$, we can turn the sum in partitions of $V$ into a sum over number of blocks each partition has, at the expense of introducing a multiplicity factor.
This factor is given by the so-called Stirling numbers of the second kind, which are defined as the number of ways to partition a set of $n$ objects into $k$ non-empty subsets and given by 
\[
    \genfrac{\{}{\}}{0pt}{}{n}{k} \coloneqq \frac{1}{k!} \sum_{i=0}^k (-1)^i \binom{k}{i} (k-i)^n .
\]
Using this, we obtain
\begin{equation}\label{eq:from_part_to_stirl}
    \sum_{\pi\in\Pi\left([m_\sigma]\right)} \left(|\pi|-1\right)\!!\, (-1)^{|\pi|-1} = \sum_{k=1}^{m_\sigma} \genfrac{\{}{\}}{0pt}{}{m_\sigma}{k} (k-1)! \,(-1)^{k-1} .
\end{equation}
Finally, due to algebraic properties of the Stirling numbers of the second kind, it holds that
\begin{equation}\label{eq:magic-stirling} 
    \sum_{k=1}^{m} \genfrac{\{}{\}}{0pt}{}{m}{k} (k-1)!\, (-1)^{k-1} = 
    \begin{dcases}
    \hfil 1 & \text{if } m=1 , \\
    0 & \text{if } m\geq2 .
    \end{dcases}
\end{equation}
Hence, we obtain
\begin{align*}
    \sum_{\sigma\in S(V)} \sign(\sigma) \prod_{v\in V} M(\eta(v),\,\eta(\sigma(v))) &\sum_{k=1}^{m_\sigma} \genfrac{\{}{\}}{0pt}{}{m_\sigma}{k} (k-1)!\, (-1)^{k-1} \\
    &=  \sum_{\sigma\in S(V):\,m_\sigma=1} \sign(\sigma) \prod_{v\in V} M(\eta(v),\,\eta(\sigma(v))) \\
    &=  \sum_{\sigma\in S_{\mathrm{cycl}}(V)} \sign(\sigma) \prod_{v\in V} M(\eta(v),\,\eta(\sigma(v))) ,
\end{align*}
where the last equality comes from the fact that $m_\sigma=1$ if and only if $\sigma$ is a full cyclic permutation of $V$ without fixed points.
Now, for a full cycle $\sigma$ of length $n$, we have $\sign(\sigma) = (-1)^{n-1}$.
This way, we arrive to
\[
(-1)^{n-1} \sum_{\sigma\in S_{\mathrm{cycl}}(V)} \prod_{v\in V} M(\eta(v),\,\eta(\sigma(v))) = (-1)^{n-1}\sum_{\sigma\in S_{\mathrm{cycl}}(V)} \prod_{v\in V} M(\eta(v),\,\eta(\sigma(v))) .
\]
Finally, reintroducing the sum over directions of differentiation, we obtain
\[
\kappa\left(X_{v_1},\,\dots,\,X_{v_n}\right) = - \left(\frac{-1}{2d}\right)^n\sum_{\sigma\in S_{\mathrm{cycl}}(V)} \sum_{\substack{\eta:\,V\to E(V)\\ \eta(v)\in E_v \; \forall v}} \prod_{v\in V} M(\eta(v),\,\eta(\sigma(v))) .
\]
By taking the field $-X_v$, and using homogeneity of cumulants, this concludes the proof of the first statement.
\qedhere

\item\label{part2} As for the second statement of the Theorem, note that in view of~\eqref{eq:thm:maincum2} and by setting $\Lambda=\Lambda_\eps\coloneqq U_\eps$ the left-hand side of~\eqref{eq:cum_limit} can be equivalently written as
\begin{multline}\label{eq:limit_cum_X} 
	\lim_{\eps\to 0} 
	\left(- \frac{1}{(2d)^n}
	\sum_{\sigma\in S_{\text{cycl}}(V)} \sum_{\substack{\eta:\,V\to E(V)\\ \eta(v)\in E_{v}\; \forall v}} 
	\prod_{v\in V} \eps^{-d} M\left(\eta(v_\eps),\,\eta(\sigma(v_\eps))\right)\right)
 \\ =
	\sum_{\sigma\in S_{\text{cycl}}(V)} \sum_{\substack{\eta:\,V\to E(V)\\ \eta(v)\in E_{v}\; \forall v}} 
	\lim_{\eps\to 0}
	\left(- \frac{1}{(2d)^n}
	\prod_{v\in V} \eps^{-d} M\left(\eta(v_\eps),\,\eta(\sigma(v_\eps))\right)\right),
\end{multline} 
where now $M\coloneqq M_{U_\eps}$.

\begin{remark}\label{rem:different-eps}
    Notice that, in principle, we should take $\sigma^\eps \in S_{\cycl}(V_\eps)$ and $\eta^\eps:\, V_\eps \to E(V_\eps)$ in~\eqref{eq:limit_cum_X}.
    However, there is a natural bijection between $S_{\cycl}(V)$ an $S_{\cycl}(V_\eps)$: for $\sigma \in S_{\cycl}(V)$, define $\sigma^\eps(v_\eps) \coloneqq \left\lfloor \frac{\sigma(v)}{\eps} \right\rfloor$. Likewise, we have a natural bijection between $\eta$'s and $\eta^\eps$'s. 
    Notice that this bijection works for every $\eps >0$ (which is not the case in the hexagonal lattice, as it will be discussed in Subsection~\ref{subsec:general_graph}). Therefore, we will simply write $\sigma$ and $\eta$ without the dependence in $\eps$ for the remainder of this proof. This also allows us to import the notion of $\eta(v)$ when $\eps \to 0$ as direction rather than directed edge, since the graph structure disappears.
    We will use similar bijections in the proof of Theorem~\ref{thm:main_cum3_cont}.
\end{remark}

We can now continue with the proof. By~\eqref{eq:limit_cum_X}  it suffices to study
\[
\lim_{\eps\to 0}\eps^{-d} M\left(\eta(v_\eps),\,\eta(\sigma(v_\eps))\right) .
\]
Using Lemma~\ref{lem:cip4}, this expression converges to
\begin{equation}\label{eq:contlimit_g2}
     \partial^{(1)}_{\eta^*(\eta(v))}\partial^{(2)}_{\eta^*(\eta(\sigma(v)))} g_U\left(v,\, \sigma(v)\right) .
\end{equation}
Remember that, as $g_U$ is differentiable off-diagonal, for $x \neq y$ we have 
\begin{equation*}
     \partial^{(1)}_{e}\partial^{(2)}_{f} g_U(x,y) 
	 =
     -\partial^{(1)}_{-e}\partial^{(2)}_{f} g_U(x,y) 
	 =
     -\partial^{(1)}_{e}\partial^{(2)}_{-f} g_U(x,y) .
\end{equation*}
However, as $\sigma$ is cyclic, any such negative signs above will appear twice in the product \eqref{eq:limit_cum_X} and therefore will cancel, so that
\[
\eqref{eq:limit_cum_X}=- \frac{1}{(2d)^n}
\sum_{\sigma\in S_{\mathrm{cycl}}(V)} \sum_{\eta':\,V\to\{e_1,\,\ldots,\,e_{2d}\}} 
\prod_{v\in V} \partial^{(1)}_{\eta'(v)}\partial^{(2)}_{\eta'(\sigma(v))} g_U\left(v,\,\sigma(v)\right).
\]
Now note that
\begin{equation*}
	\prod_{v\in V} \partial^{(1)}_{\eta'(v)}\partial^{(2)}_{\eta'(\sigma(v))} g_U\left(v,\,\sigma(v)\right)
	=
	\prod_{v\in V} \partial^{(1)}_{\eta(v)}\partial^{(2)}_{\eta(\sigma(v))} g_U\left(v,\,\sigma(v)\right),
\end{equation*}
for any ${\eta},\, {\eta}'$ such that the scalar product $\inpr{{\eta}(v)}{ {\eta}'(v)} = \pm 1$ for all $v$.
Therefore, considering $2^n$ choices of ${\eta'}:\,V \to \{e_1,\,\dots,\,e_{2d}\}$ for a single $\eta:\, V\to \{e_1,\,\dots,\,e_{d}\}$, we arrive at
\begin{align}
\lim_{\eps\to 0}\eps^{-dn} \kappa\left(X_{v_1}^\eps,\,\dots,\,X_{v_n}^\eps\right)
& = 
- \frac{1}{d^n}
\sum_{\sigma\in S_{\mathrm{cycl}}(V)} \sum_{\eta:\,V\to\{e_1,\,\ldots,\,e_{d}\}} 
\prod_{v\in V} \partial^{(1)}_{\eta(v)}\partial^{(2)}_{\eta(\sigma(v))} g_U\left(v,\,\sigma(v)\right).\label{eq:lim_abs}
\qedhere
\end{align}

\end{enumerate}
\end{proof}
\begin{proof}[Proof of Corollary~\ref{cor:degree_field}]
 The proof of Theorem~\ref{thm:main_cum2} can be carried through as long as one has a field whose cumulants satisfy~\eqref{eq:alternative}. Therefore by Proposition~\ref{thm:swan_UST_fGFF} we can also state the results of Theorem~\ref{thm:main_cum2} for the degree field.
\end{proof}

\subsection{Proof of Theorem~\ref{thm:main_cum3_discrete}}

In the following proof, we will write $\kappa$ instead of $\kappa_\Lambda^{\bzero}$ for simplicity as before. This proof follows the same structure of Theorem~\ref{thm:main_cum2}. However, we have to account for the more intricate set of permutations of edges.
 
Call $Z_v\coloneqq X_v Y_v$. Let $v_1,\,\ldots,\,v_n$ be as in the statement of the theorem. Once again, we write the joint moments of the field of interest as a sum. Using \eqref{eq:expansion-XvYv}, we have 
\begin{equation}\label{eq:alternative2}
   \Efd{ \prod_{v\in B}Z_v }_\Lambda = 
    \frac{1}{(2d)^{|B|}}
    \sum_{\substack{\eta:\,B \to E(B)\\ \eta(v)\in E_v \; \forall v}}
    \sum_{A \subseteq E(B) \setminus \eta(B)}
    (-1)^{|A|}
    \Efd{\zeta_{\eta(B) \cup A}}_\Lambda,
\end{equation}
where, from Lemma \ref{lem:joint-moments-grad},
\[
   \Efd{\zeta_{\eta(B) \cup A}}_\Lambda = \det(M)_{\eta(B) \cup A}.
\]
This determinant can be written in terms of permutations of edges as
\[
    \det(M)_{\eta(B) \cup A} = \sum_{\tau\in S(\eta(B) \cup A)} \sign(\tau) \prod_{v\in  \eta(B) \cup A} M\left(v,\,\tau(v)\right).
\]
Note that, as there can be multiple edges attached to the same site, there is no longer a correspondence between permutation of edges in $\eta(B) \cup A$ and permutations in $B$. Using~\eqref{def:cum_Gras}, we get
\begin{multline*}
\kappa\left(Z_{v_1},\,\dots,\,Z_{v_n}\right) 
	= \\
	\left(\frac{1}{2d}\right)^{n}
	\sum_{\eta}
	\sum_{A} (-1)^{|A|}
	\sum_{\pi\in\Pi(V)} \left(|\pi|-1\right)\!!\, (-1)^{|\pi|-1} \prod_{B\in\pi} \sum_{\tau\in S(\mcE_B)} \sign(\tau) \prod_{f\in \mcE_B} M\left(f,\,\tau(f)\right),
\end{multline*}
where, once again, the sum over $\eta$'s is over all functions $\eta:\, V \to E(V)$ with $\eta(v)\in E_v $ for all $v$, the sum over $A$'s is over the subsets of $A \subseteq {E}(V) \setminus \eta(V)$, and $\mcE_B=\mcE_B(\eta,\,A)$ is the set of edges in $\eta(V) \cup A$ that intersect sites of $B$.  In the above, we are using that $V$ is a good set, and therefore $\{\mcE_B(\eta,A):\, B \in \pi\}$ provides a partition of $\mcE \coloneqq\eta(V) \cup A$. In the following, we will write $\mcE_v$ to denote $\mcE_{\{v\}}$.

Before proceeding to analyze the sum over the partitions, as we did in the proof of Theorem~\ref{thm:main_cum2}, notice that $|A|=|\mcE| - n$. Therefore, the sum above only depends on $\eta$ and $A$ through $\mcE$. 
We notice that for a fixed $\mcE$ there are $ \prod_{v \in V} |\mcE_{v}|$ choices for the pair $(\eta(V),A)$ yielding the same $\mcE$, so the sum above can be written as
\begin{multline*} 
\kappa\left(Z_{v_1},\,\dots,\,Z_{v_n}\right) 
= \\
\left(\frac{-1}{2d}\right)^{n} \sum_{\mcE:\, |\mcE_{v}| \ge 1\; \forall v} K(\mcE) \sum_{\pi\in\Pi(V)} 
\left(|\pi|-1\right)\!!\, (-1)^{|\pi|-1} \prod_{B\in\pi} \sum_{\tau\in S(\mcE_B)} \sign(\tau) 
\prod_{f\in \mcE_B} M\left(f,\,\tau(f)\right).
\end{multline*}

We now concentrate on the sum over the partitions again, dropping the sum over $\mcE$ for the moment. That is, we examine
\begin{equation}\label{eq:swap2}
	\sum_{\pi\in\Pi(V)} \left(|\pi|-1\right)\!!\, (-1)^{|\pi|-1} \prod_{B\in\pi} \sum_{\tau\in S(\mcE_B)} \sign(\tau) 
	\prod_{f\in \mcE_B} M\left(f,\,\tau(f)\right) .
\end{equation}
Once again, we write $\pi = \{\pi_1,\,\dots,\, \pi_{|\pi|}\}$ and use \eqref{eq:sum-card} to swap the sum over $\tau\in S(\mcE_B)$ with the product over $B\in\pi$ in~\eqref{eq:swap2}. This way, we obtain
\[
    \sum_{\pi\in\Pi(V)} \left(|\pi|-1\right)\!!\, (-1)^{|\pi|-1} 
    \sum_{\substack{\tau_i\in S(\mcE_{\pi_i}),\\i\in\left[|\pi|\right]}} 
    \sign(\tau) \prod_{B\in\pi} \prod_{f\in \mcE_B} M\left(f,\,\tau(f)\right) ,
\]
where $\tau = \tau_1\cdots\tau_{|\pi|}$.

Now, we wish to swap the sum over $\pi\in\Pi(V)$ with the sum over permutations $\tau$.
Let us first sum over $\tau\in S(\mcE)$. Any such $\tau$ can always be uniquely decomposed into $m(\tau)$ disjoint cyclic permutations, that is, $\tau = \tau_1\cdots\tau_{m(\tau)}$ for some $m(\tau)\in \{1,\,\dots,\,|\mcE|\}$. Recall the definitions of $\pi^V_\tau$ and the set of connected permutations $S_{\co}(\mcE)$ of $\mcE$ as defined in Subsection~\ref{subsec:graphs-and-permutations}. Denote by $\Pi_\tau(V) \subseteq \Pi(V)$ the subset of partitions of $V$ that are coarser than $\pi^V_\tau$. With this notation in place, we can rewrite the last expression as 
\begin{equation}\label{eq:stirling_2}
    \sum_{\sigma\in S(\mcE)} \sign(\sigma) \prod_{f\in \mcE} M\left(f,\,\sigma(f)\right) \sum_{\pi\in\Pi_\tau(V)} \left(|\pi|-1\right)\!!\, (-1)^{|\pi|-1} .
\end{equation}
Notice that $\pi\in\Pi_{\tau}(V)$ has at most $m(\tau)$ blocks. Using again expression \eqref{eq:from_part_to_stirl} in~\eqref{eq:stirling_2}, we get

\begin{equation*}
   \sum_{\tau\in S(\mcE)} \sign(\tau) 
   \prod_{f\in \mcE} M\left(f,\,\tau(f)\right) 
   \sum_{k=1}^{{m}(\tau)} \genfrac{\{}{\}}{0pt}{}{{m}(\tau)}{k} (k-1)!\, (-1)^{k-1} 
    =
	\sum_{\tau\in S_{\text{co}}(E)} \sign(\tau) \prod_{f\in \mcE} M\left(f,\,\tau(f)\right).
\end{equation*}
The equality is due to the fact that~\eqref{eq:magic-stirling} forces $ m(\tau)=1$, which in turn means that $\tau$ is connected. Reintroducing the sum over $\mcE$ (with multiplicity), we obtain
\[
	\kappa\left(Z_v:\, v \in V\right) 
	= 
	\left(\frac{-1}{2d}\right)^n
	\sum_{\mcE:\, |\mcE_{v}|\ge 1 \; \forall v} K(\mcE) 
	\sum_{\tau\in S_{\co}(\mcE)} \sign(\tau) \prod_{f\in \mcE} M\left(f,\,\tau(f)\right).
\]
This concludes the proof of the theorem.

\subsection{Proof of Theorem~\ref{thm:main_cum3_cont}}

The proof is divided into four steps.
In \ref{step1}, we start from the final expression obtained in Theorem~\ref{thm:main_cum3_discrete} and show that it suffices to sum over only bare permutations $\tau$, instead of the bigger set of connected permutations.
In~\ref{step2}, we simplify the expression even further, showing that only permutations that enter and exit through parallel edges on every point will give a non-zero contribution in the scaling limit as $\eps\rightarrow 0$.
In~\ref{step3}, we write the expression in terms of contributions of the permutations acting locally in the vicinity of a vertex and globally mapping an edge incident to one vertex to an edge which is incident to another vertex.
Finally, in~\ref{step4}, we identify the global multiplicative constant of the cumulants.

\begin{enumerate}[wide, labelindent=0pt, label={\bf Step \arabic*.}, ref=Step \arabic*]
\item\label{step1}
From Theorem~\ref{thm:main_cum3_discrete}, we start with the expression
\[
	\kappa\left(Z^\eps_v:\, v \in V\right) 
= 
	\left(\frac{-1}{2d}\right)^n
	\sum_{\mcE:\, |\mcE_{v_\eps}|\ge 1 \; \forall v} K(\mcE) 
	\sum_{\tau\in S_{\co}(\mcE)} \sign(\tau) \prod_{f\in \mcE} M\left(f,\,\tau(f)\right),
\]
where now $\mcE$ is a subset of edges of $U_\eps$ and $M=M_{U_\eps}$.
Throughout this proof, we will often omit the dependence of $\mcE$, $\tau$ (and later of other functions as well) on $\eps$ by using the same idea as in Remark~\ref{rem:different-eps}. This allows us not only to simplify the expressions, but also to use the notions of connected/bare permutations as $\eps \to 0$.

We will reduce the summation over $\tau \in S_{\co}$ to a summation over $\tau \in S_{\bare}$ (recall the definitions in Section \ref{subsec:graphs-and-permutations}).
Note that, for the edges $E_{\tau}(V)$ of the induced graph $V_{\tau}$, we have that $|V| \leq |E_{\tau}(V)| \leq |E(V)|$ as $\tau$ is connected.
On the other hand, for each $\tau \in S_{\co}$,

\begin{equation}\label{eq:prod_M}
	\prod_{f\in \mcE} M\left(f,\,\tau(f)\right) 
	=
	\mathcal{O} \left(\prod_{f\in \tilde E_{\tau}(V)}\varepsilon^{d} \, \partial^{(1)}_{\eta^*(f)}\partial^{(2)}_{\eta^*(\tau(f))} g_U\left(f^-,\, \tau(f)^-\right) \right), 
\end{equation}
being $\tilde E_\tau(V)$ those edges in $E_v$ such that $\tau(f)\in E_w$, with $w\neq v$, and $\eta^*$ is the direction induced by $f$ resp. $\tau(f)$ on the point $f^-$ resp. $\tau(f)^-$ as defined in Lemma~\ref{lem:cip4}. 
To show~\eqref{eq:prod_M} notice that, by Lemma~\ref{lem:cip4},
\begin{equation}\label{eq:order_M}
\overline{M}(f,\,\tau(f))
=
\lim_{\eps\to 0}\eps^{-d}M_{U_\eps}(f,\,\tau(f))
=
\partial^{(1)}_{\eta^*(f)}\partial^{(2)}_{\eta^*(\tau(f))} g_U\left(v, \, v'\right)
\end{equation}
whenever $f^-=v_\eps \neq \tau(f)^-=v'_\eps$, for some $v,\,v'\in V$. Once again, we are disregarding the dependence of $f$ in $\eps$, by using the natural bijection of Remark~\ref{rem:different-eps}. Using \citet[Lemma 5]{cipriani2022properties},
\begin{equation}
\overline{M}(f,\,\tau(f)) =
\lim_{\eps\to 0}M_{U_\eps}(f,\,\tau(f)) =
\begin{dcases}
    \nabla^{(1)}_{e_i}\nabla^{(2)}_{e_j}G_0(o,\,o)&\text{if } f^- = \tau(f)^-,\\
    \hfil 0 & \text{if } f^- \neq \tau(f)^-,
\end{dcases}\label{eq:cases_M}
\end{equation}
being $\eta^*(f)=e_i$ and $\eta^*(\tau(f)) = e_j$ for some $i,\,j\in[2d]$. Therefore, we can split the product $\prod_{f\in \mcE} M\left(f,\,\tau(f)\right)$ into one contribution where the permutation maps edges incident to one vertex $v_\eps$ to edges which are incident to another vertex $v'_\eps\neq v_\eps$, and another contribution with edges incident to vertices $v_\eps$ which are mapped by $\tau$ to vertices incident to the same vertex $v_\eps$. For the former, $M(f,\,\tau(f))$ is of order $\eps^d$ by~\eqref{eq:order_M}; for the latter, it is of order one by~\eqref{eq:cases_M}. This shows~\eqref{eq:prod_M}.

Now, remember that we rescale the cumulants by $\eps^{-d|V|}$, hence the expression in~\eqref{eq:prod_M} will be non-zero when taking the limit $\eps \to 0$ if and only if $|V|=|E_{\tau}(V)|$ (it can never diverge since $\tau$ is connected and hence $|V| \leq |E_{\tau}(V)|$). This implies that we can consider only permutations $\tau$ which are bare. Once again, following the idea of Remark~\ref{rem:different-eps}, we are ignoring the dependence of $\tau$ in $\eps$, allowing us to take the limit as $\eps \to 0$.

\item\label{step2} Now, we examine the expression
\begin{equation}\label{eq:now_bare}
	\left(\frac{-1}{2d}\right)^n
	\sum_{\mcE:\, |\mcE_{v_\eps}|\ge 1 \; \forall v_\eps}  K(\mcE) 
	\sum_{\tau\in S_{\bare}(\mcE)} \sign(\tau) \prod_{f\in \mcE} M\left(f,\,\tau(f)\right).
\end{equation}
Recall that each bare $\tau$ defines an entry and an exit edge in $v_\eps$, as stated in Subsection~\ref{subsec:graphs-and-permutations}.
Consider a permutation $\tau \in S_{\bare}(\mcE)$ which enters $v_\eps$ through the edge $f$ and exits through the edge $f'$ such that $\inpr{f}{f'} \neq \pm 1$.
We construct another permutation $\rho \in S_{\bare}(\mcE^\prime)$ for a possibly different $\mcE'$, such that $\rho (e)=\tau(e)$ for all $e \not\in \mcE_{v_\eps}$, and such that it will cancel the contribution of $\tau$ in \eqref{eq:now_bare}. 

To construct such a permutation $\rho$, we take $\mcE^\prime \coloneqq \left( \cup_{v' \neq v} \mathcal{E}_{v'_\eps} \right)\cup \mcE''_{v_\eps}$, where
\[
    \mcE''_u \coloneqq \{e\in \mcE_u:\,\inpr{e}{f}=\pm 1\}\cup \{-e\in \mcE_u:\,\inpr{e}{f}\neq \pm 1\},\quad u\in U_\eps .
\] 
Remember that for $e = (u,\,u+e_i)$, we write $-e$ to denote $(u,\,u-e_i)$. In words, $\mcE''_{v_\eps}$ is the reflection of $\mcE_{v_\eps}$ with respect to the direction induced by $f$. See Figure~\ref{fig:EtoEprime} for two examples in $d=2$.

\begin{figure}[H]
    \centering
    \includegraphics[scale=.7]{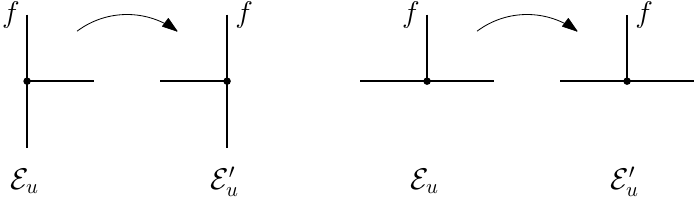}
    \caption{Two examples of $\mcE_u$ and its corresponding $\mcE'_u$ in $d=2$.}
    \label{fig:EtoEprime}
\end{figure}

\noindent
Then we set $\rho (e)=\tau(e)$ for all $e \not\in \mcE_{v_\eps}$ and $\rho(f)=-\tau(f)$ as well as $\rho(-e)=-\tau(-e)$ for all $e \in \mcE''_{v_\eps}$.
That is, we invert every edge of $\mathcal{E}_{v_\eps}$ that is not $f$. See Figure~\ref{fig:opposite} for an example of $\tau$ and $\rho$ in $d=2$. 

\begin{figure}[ht!]
    \centering
    \includegraphics[scale=.6]{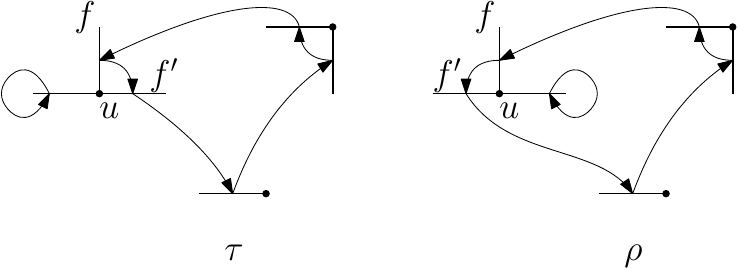}
    \caption{Example of two permutations $\tau$ (left) and $\rho$ (right) in $d=2$ whose exit edges from $u$ are opposite.
    The contributions of these two permutations will cancel each other in the limit.}
    \label{fig:opposite}
\end{figure}

\noindent
Under these conditions, by~\eqref{eq:cases_M} and the translation/rotation invariance of the discrete Green's function in $\mathbb{Z}^d$, we have
\begin{equation}\label{eq:perp1} 
\overline{M}(e,\tau(e))
=
\nabla^{(1)}_{e_2}\nabla^{(2)}_{e_1}G_0(o,\,o)
=
\nabla^{(1)}_{e_2}\nabla^{(2)}_{-e_1}G_0(o,\,o) 
\end{equation}
for all $e \in \mathcal{E}_{v_\eps}$ such that $e^{-}=\tau(e)^-$ and $\langle e, \tau(e)\rangle =0$. 

Now, remember that we assumed that $\tau$ leaves $v_\eps$ through an edge $f'$ such that $\inpr{f}{f'} \neq \pm 1$.
Therefore, $\rho$ leaves $v_\eps$ through $-f'$. Let us call $v'$ the point such that $\tau(f')\in \mcE_{v'_\eps}$. Using Lemma~\ref{lem:cip4}, we have that
\begin{align}\label{eq:perp2} 
\overline{M}(f',\,\tau(f'))
& = \nonumber
   \partial^{(1)}_{\eta^*(f')}\partial^{(2)}_{\eta^*(\tau(f'))} g_U\left(v,\, v'\right)
\\&  = \nonumber
   -\partial^{(1)}_{\eta^*(-f')}\partial^{(2)}_{\eta^*(\tau(f'))} g_U\left(v,\, v'\right)
\\&  = \nonumber
   -\partial^{(1)}_{\eta^*(-f')}\partial^{(2)}_{\eta^*(\rho(f'))} g_U\left(v,\, v'\right)
\\&  = 
\overline{M}(-f',\,\rho(f')).
\end{align}
Furthermore, notice that $K(\mathcal{E})=K(\mathcal{E}^\prime )$ and $\sign(\tau)=\sign(\rho)$.
Now, examining the product on the rightmost part of \eqref{eq:now_bare}, 
we have that $\rho$ and $\tau$ coincide for all edges outside of $\mathcal{E}_{v_\eps}$.
As for the contributions given by factors of edges in $\mathcal{E}_{v_\eps}$, \eqref{eq:perp1} and \eqref{eq:perp2} imply that their product gives the same value under $\tau$ and $\rho$, except for the opposite sign of~\eqref{eq:perp2}.
Therefore, for any permutation $\tau$ which exits $v_\eps$ through an edge that is orthogonal to the entry edge, there exists a permutation $\rho$ such that
\begin{equation}\label{eq:cancel}
K(\mcE)\sign(\tau)  \prod_{f\in \mcE} \overline{M}\left(f,\,\tau(f)\right) 
=
-K(\mcE')\sign(\rho) \prod_{f\in \mcE^\prime} \overline{M}\left(f,\,\rho(f)\right).
\end{equation}
Thus the only bare permutations which give a contribution to the limit of~\eqref{eq:now_bare} as $\eps\to 0$ are those which enter $v_\eps$ through an edge $f$ and exit $v_\eps$ through either $f$ itself or $-f$.
\end{enumerate}

For the remainder of the proof, we will use the notation
\begin{equation}\label{eq:def_bar_M}
\overline{M}(f,\,\tau(f)) = 
\begin{dcases}
    \hfil \nabla^{(1)}_{e_i}\nabla^{(2)}_{e_j}G_0(o,\,o)&\text{if } f^- = \tau(f)^-,\\ 
    \partial^{(1)}_{e_i}\partial^{(2)}_{e_j} g_U\left(v,\, v'\right)  &\text{if } f^-=v_\eps \neq v'_\eps=\tau(f)^-,\ v,\,v'\in V
\end{dcases}
\end{equation}
whenever $\eta^*(f)=e_i$ and $\eta^*(\tau(f)) = e_j$ for some $i,\,j\in[2d]$.
We now need to further expand the expression
\begin{equation}\label{eq:now_internal_ext}
	\left(\frac{-1}{2d}\right)^n
	\sum_{\mcE:\,|\mcE_v|\ge 1\,\forall v\in V}  K(\mcE) 
	\sum_{\tau\in S^*_{\bare}(\mcE)}\sign(\tau) \prod_{f\in \mcE} \overline{M}\left(f,\,\tau(f)\right)
\end{equation}
where $S^*_{\bare}(\mcE)$ indicates bare permutations that exit a point through an edge parallel to the entry one (it can be the same). Note that now $\mcE_v\subseteq E_o+v $, $v\in V$.

\begin{enumerate}[wide, labelindent=0pt, label={\bf Step \arabic*.}, ref=Step \arabic*]
\setcounter{enumi}{2}
\item\label{step3} As stated in Subsection~\ref{subsec:graphs-and-permutations}, any bare $\tau$ induces a permutation $\sigma\in S_{\cycl}(V)$ on vertices. We will extract from $\tau$ a permutation $\sigma$ among vertices and a choice of edges $\eta$ (as in Theorem~\ref{thm:main_cum2}) and we will separate it from what $\tau$ does ``locally'' in the edges corresponding to a given point. To do this, we need to introduce, for fixed $\tau$, the functions
\[
\eta:\,V\to E(V)\text{ such that }\eta(v)\in E_{v}\text{ for all }v
\]
such that $\eta(v)$ is the edge from which $\tau$ enters $v$, and
\[
\gamma:\,V\to\{-1,\,1\}
\]
which equals $1$ at $v$ if $\tau$ enters and exits $v$ through the same edge $\eta(v)$, and equals $-1$ if $\tau$ exits $v$ from $-\eta(v)$. In short, the exit edge from $v$ according to $\tau$ is $\gamma(v)\eta(v)$ (from now on written as $\gamma\eta(v)$ to relieve notation). Note that $\eta$, $\sigma$ and $\gamma$ determine $E_\tau(V)$ and are functions of $\tau$ (we will not write this to avoid heavy notation). With the above definitions we have that~\eqref{eq:now_internal_ext} becomes
\begin{multline}
\left(\frac{-1}{2d}\right)^n
	\sum_{\mcE:\, |\mcE_{v}|\ge 1 \; \forall v} 
	\sum_{\substack{\eta:\,V\to E(V)\\ \eta(v)\in E_{v}\; \forall v}}\;
	\sum_{\sigma\in S_{\cycl}(V)} \sum_{\gamma:\,V\to\{-1,\,1\}} 
	\left(\prod_{v\in V} K(\mcE_v)\overline{M}\left(\gamma\eta(v),\,\eta(\sigma(v))\right)\right)\times 
	\\
	\times\sum_{\tau\in S^*_{\bare}(\mcE;\,\eta,\,\sigma,\,\gamma)}\sign(\tau) \prod_{f\in \mcE\setminus \{\gamma\eta(V)\}} 
	\overline{M}\left(f,\,\tau(f)\right),\label{eq:first_lim}
\end{multline}
where $\gamma\eta(V) \coloneqq \{\gamma\eta(v):\,v\in V\}$, and $S^*_{\bare}(\mcE;\,\eta,\,\sigma,\,\gamma)$ is the set of bare permutations as in~\eqref{eq:now_internal_ext} which now enter and exit each point $v$ through the edges prescribed by $\eta$, $\sigma$ and $\gamma$. In this case we will say that $\tau$ is compatible with $(\mcE;\,\eta,\,\sigma,\,\gamma)$. Figures~\ref{fig:eta_sigma_gamma_2_figures_compatible} and \ref{fig:eta_sigma_gamma_incompatible_4_figures} give examples of compatible resp. non-compatible pairs of permutations in $d=2$.

\begin{figure}[ht!]
    \centering
    \includegraphics[scale=.6]{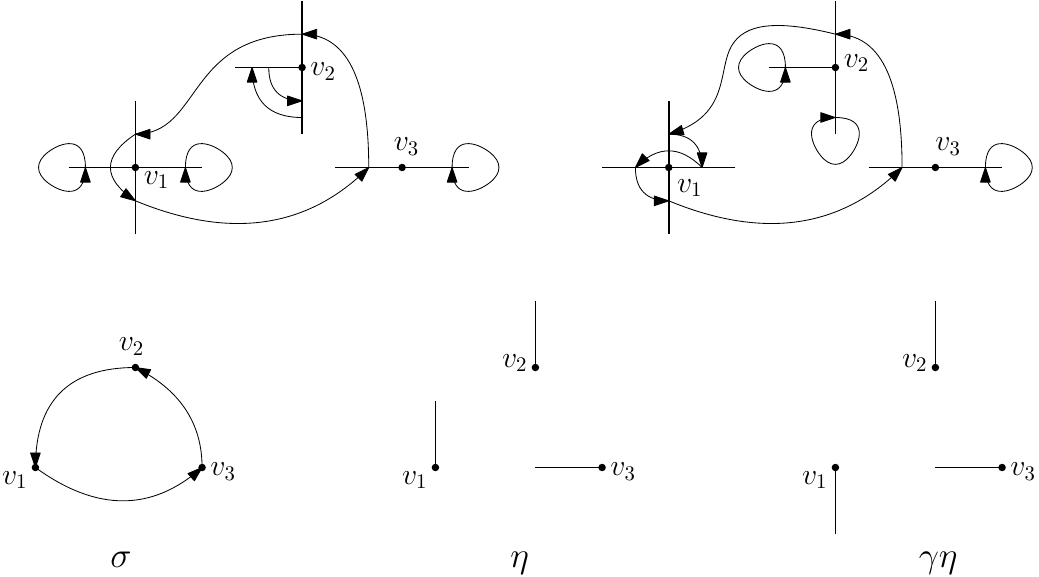}
    \caption{Top: two different compatible permutations in $d=2$. Bottom: their corresponding $\sigma$, $\eta$ and $\gamma$.}
    \label{fig:eta_sigma_gamma_2_figures_compatible}
\end{figure}

\begin{figure}[ht!]
    \centering
    \includegraphics[scale=.6]{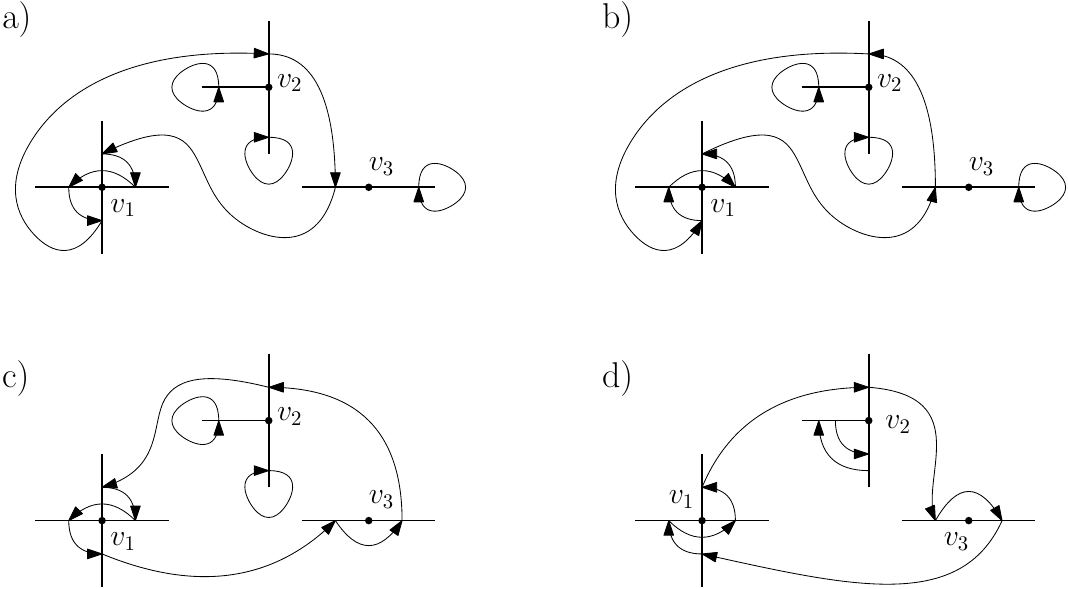}
    \caption{Four different permutations that are not compatible with those in Figure~\ref{fig:eta_sigma_gamma_2_figures_compatible}. a) Permutation that respects $\eta$ and $\gamma$ but not $\sigma$. b) Permutation that respects $\sigma$ and $\gamma$ but not $\eta(v_1)$. c) Permutation that respects $\sigma$ and $\eta$ but not $\gamma(w_3)$. d) Permutation that does not respect $\sigma$, nor $\eta(v_1)$, nor $\gamma(w_3)$.}
    \label{fig:eta_sigma_gamma_incompatible_4_figures}
\end{figure}

We have that \eqref{eq:first_lim} is equal to
\begin{align}
\left(\frac{-1}{2d}\right)^n
	&\sum_{\mcE:\, |\mcE_v|\ge 1 \; \forall v} \sum_{\substack{\eta:\,V\to E(V)\\ \eta(v)\in \mcE_v\; \forall v}}\;\sum_{\sigma\in S_{\cycl}(V)} \sum_{\gamma:\,V\to\{-1,\,1\}} \left(\prod_{v\in V} \gamma(v)K(\mcE_v)\overline M\left(\eta(v),\,\eta(\sigma(v))\right)\right)\nonumber\times \\
 &
	\times \sum_{\tau\in S^*_{\bare}(\mcE;\,\eta,\,\sigma,\,\gamma)}\sign(\tau) \prod_{f\in \mcE\setminus \{\gamma\eta(V)\}} \overline M\left(f,\,\tau(f)\right)
\nonumber\\
=- & \sum_\eta\sum_{\sigma}\left(\prod_{v\in V}\overline M\left(\eta(v),\,\eta(\sigma(v))\right)\right) \sum_{\mcE:\, \mcE_v\ni\eta(v)\,\forall v }K(\mcE) \sum_{\gamma} \sum_{\tau\in S^*_{\bare}(\mcE;\,\eta,\,\sigma,\,\gamma)}\frac{\sign(\tau)}{(-1)^{n-1}}\, \times \nonumber\\
&\times \prod_{v\in V}\left[\frac{1}{2d} \gamma(v) \prod_{f\in \mcE_v\setminus \{\gamma\eta(v)\}} \overline M\left(f,\,\tau(f)\right)\right].\label{eq:second_lim}
\end{align}
Note that $\gamma(v)$ is accounted for since, by Lemma~\ref{lem:cip4},
\[
\overline M\left(\gamma\eta(v),\,\eta(\sigma(v))\right) = \gamma(v) \, \overline M\left(\eta(v),\,\eta(\sigma(v))\right).
\]
In the next step, we will fix $\eta$ and $\sigma$, and prove that 
\begin{equation}\label{eq:at_constant}
	\sum_{\mcE:\, \mcE_v\ni \eta(v)\,\forall v } K(\mcE) \sum_{\gamma} 
	\sum_{\tau\in S^*_{\bare}(\mcE;\,\eta,\,\sigma,\,\gamma)}\frac{\sign(\tau)}{(-1)^{n-1}} \prod_{v\in V}\left[\frac{1}{2d} \gamma(v) \prod_{f\in \mcE_v\setminus \{\gamma\eta(v)\}} \overline M\left(f,\,\tau(f)\right)\right]
\end{equation}
is a constant independent of $v$, $\eta$ and $\sigma$.

\item\label{step4} Using $\sigma$, $\eta$ and $\gamma$, we have been able to recover in~\eqref{eq:second_lim} an expression that depends on permutations of vertices and directions similar to that of Theorem~\ref{thm:main_cum2} item~\ref{item2_cum2}. To complete the proof we will perform a ``surgery'' to better understand expression~\eqref{eq:at_constant}. This surgery aims at decoupling the local behavior of $\tau$ at a vertex versus the jumps of $\tau$ between different vertices. To do this, we define
\begin{equation}\label{eq:def_omega}
\omega^\tau_v(f)\coloneqq \begin{cases}
    \tau(f)&\text{if }f\neq \gamma\eta(v)\\
    \tau(\eta(v))&\text{if }f= \gamma\eta(v),\,\gamma(v)= -1
\end{cases},\quad f\in \mcE_v\setminus\{\eta(v)\}
\end{equation}
and
\begin{equation}\label{eq:def:tau_minus}
\tau\setminus\omega^\tau_v(f)\coloneqq \begin{cases}
    \tau(f)&\text{if }f\notin \mcE_v\\
    \eta(\sigma(v))&\text{if }f= \eta(v)
\end{cases},\quad f\in (\mcE\setminus \mcE_v) \cup \{\eta(v)\}.
\end{equation}
In words, $\omega^\tau_v$ is the permutation induced by $\tau$ on $\mcE_v\setminus\{\eta(v)\}$ by identifying the entry and the exit edges. On the other hand, $\tau\setminus\omega^\tau_v(f)$ follows $\tau$ globally until it reaches the edges incident to $v_\eps$, from where it departs reaching the edges of the next point. In Figure~\ref{fig:omega} we can see some examples for $\gamma(v)=-1$ in $d=2$, as the action of the surgery is trivial when $\gamma(v)=1$. 
\begin{figure}[H]
    \centering
    \includegraphics[scale=.7]{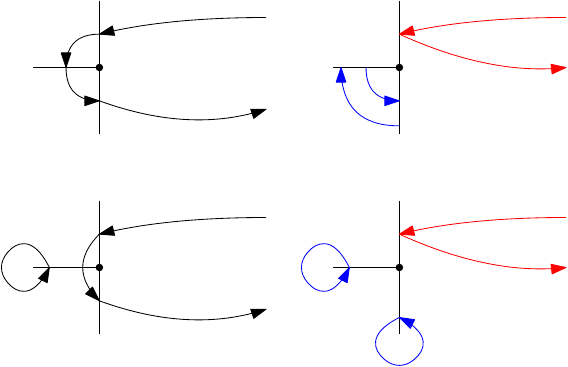}
    \caption{Example of two permutations $\tau$ on the left and their respective $\omega_v^\tau$ (in blue) and $\tau\setminus\omega_v^\tau$ (in red) on the right, both for the case $\gamma(v)=-1$, in $d=2$.}
    \label{fig:omega}
\end{figure}
In the following, we state two technical lemmas the we need to complete the proof of Theorem~\ref{thm:main_cum3_cont}. Their proofs will be given on page~\pageref{ref:pg_proof}.

\begin{lemma}\label{lem:bij_tau}
Let $\mcE\subseteq E(V)$, $ \eta:\,V\to E(V)$ such that $\eta(v)\in \mcE_v$ for all $v\in V$, $\sigma\in S_{\cycl}(V)$, $\gamma:V\to\{-1,\,1\}$ and let $\tau$ be compatible with $(\mcE;\,\eta,\,\sigma,\,\gamma)$. For every $v\in V$ there is a bijection between $S(\mcE_v\!\setminus\!\{\eta(v)\})$ and $\{\omega^\tau_v:\,\tau\text{\normalfont{ compatible with }} (\mcE;\,\eta,\,\sigma,\,\gamma)\}$.
\end{lemma}

See Figure~\ref{fig:lemma47} for an instance of Lemma~\ref{lem:bij_tau}.
\begin{figure}[htb!]
    \centering
    \includegraphics[scale=.8]{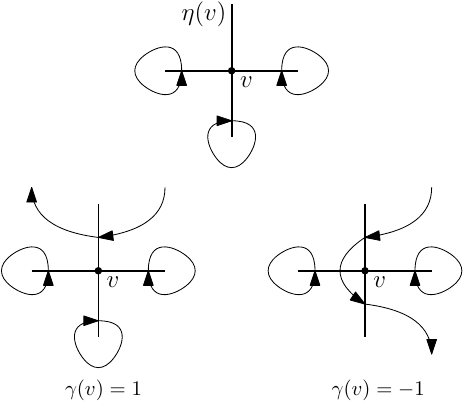}
    \caption{Top: a permutation $\omega\in S(\mcE_{v}\setminus\{\eta(v)\})$. Bottom: there is only one $\tau$ compatible with $\eta,\,\gamma$ that can induce $\omega^\tau_v=\omega$: $\tau$ is depicted on the left for $\gamma(v)=1$ and on the right for $\gamma(v)=-1$.}
    \label{fig:lemma47}
\end{figure}

\begin{lemma}[Surgery of $\tau$]\label{lem:surg_tau}
Fix $v\in V$, $\mcE$, $\eta$, $\sigma$, $\gamma$ as above. Let $\tau$ be compatible with $(\mcE;\,\eta,\,\sigma,\,\gamma)$. Then for $\mcE_v \ni \eta(v)$ we have that
\begin{equation}\label{eq:sgn_tau}
\sign(\tau)=\gamma(v)\sign(\tau\setminus\omega^\tau_v)\sign(\omega_v^\tau)
\end{equation}
and
\begin{equation}\label{eq:surg_M_original}
\prod_{f\in \mcE_v\setminus  \{\gamma\eta(v)\}} \overline M\left(f,\,\tau(f)\right) = \frac{\overline{M}\left(\eta(v),\,\omega^\tau_v(\gamma\eta(v))\right)}{\overline{M}\left(\gamma\eta(v),\,\omega^\tau_v(\gamma\eta(v))\right)} \prod_{f\in \mcE_v\setminus \{\eta(v)\}} \overline M\left(f,\,\omega^\tau_v(f)\right).
\end{equation}
Equivalently we can write
\begin{equation}\label{eq:surg_M}
\prod_{f\in \mcE_v\setminus  \{\gamma\eta(v)\}} \overline M\left(f,\,\tau(f)\right) = \prod_{f\in \mcE_v\setminus  \{\eta(v)\}} \overline M^{\gamma}\left(f,\,\omega^\tau_v(f)\right),
\end{equation}
where for any $g \in \mcE_v\setminus\{\eta(v\}$
\[
\overline M^{\gamma}(f,\,g)= \begin{cases} \overline{M}(\eta(v),\,g) & \text{ if } f= \gamma\eta(v), \\ \overline{M}(f,\,g) & \text{ if } f \neq \gamma \eta(v).
\end{cases} 
\]

\end{lemma}

\noindent
Note that $\overline{M}^{\gamma}$ is not necessarily a symmetric matrix anymore.
In Figure~\ref{fig:example_omega_complete}, we can see an example in $d=2$ of the surgery from $\tau$ to $\omega_{v_2}^\tau$, $\omega_{v_1}^\tau$, $\omega_{v_3}^\tau$ and $((\tau\setminus\omega_{v_1}^\tau)\setminus\omega_{v_2}^\tau)\setminus\omega_{v_3}^\tau$. 
\begin{figure}[ht!]
    \centering
    \includegraphics[scale=.7]{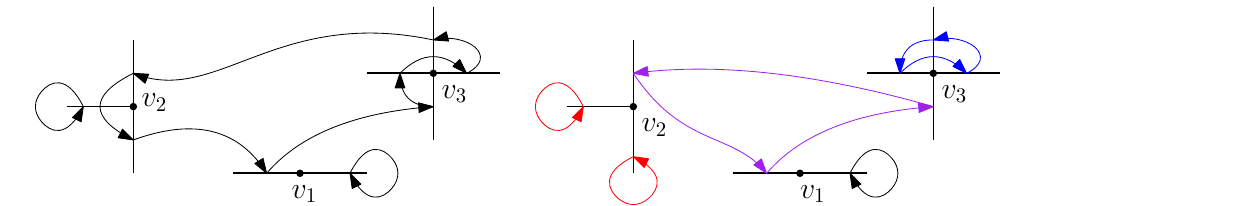}
    \caption{Example of a permutation $\tau$ on the left and its respective $\omega_{v_1}^\tau$ (in black), $\omega_{v_2}^\tau$ (in red), $\omega_{v_3}^\tau$ (in blue) and $((\tau\setminus\omega_{v_1}^\tau)\setminus\omega_{v_2}^\tau)\setminus\omega_{v_3}^\tau$ (in purple) on the right, in $d=2$.}
    \label{fig:example_omega_complete}
\end{figure}

We will now use these lemmas to rewrite~\eqref{eq:at_constant} in a more compact form.
Using~\eqref{eq:sgn_tau} recursively, we get
\[
\sign(\tau) = \left(\prod_{v\in V}\gamma(v)\sign(\omega^\tau_v)\right) \sign((((\tau\setminus\omega^\tau_{v_1})\setminus\omega^\tau_{v_2})\setminus\ldots)\setminus\omega^\tau_{v_n}).
\]
Note that the permutation $ (((\tau\setminus\omega^\tau_{v_1})\setminus\omega^\tau_{v_2})\setminus\ldots)\setminus\omega^\tau_{v_n}$ equals the permutation 
\[
(\eta(v_1),\,\eta(\sigma(v_1)),\,\eta(\sigma(\sigma(v_1))),\,\ldots,\,\eta(\sigma^{n-1}(v_1)))
\]
and, as such, it constitutes a cyclic permutation on $n$ edges in $\mcE$, so that
\[
\frac{\sign((((\tau\setminus\omega^\tau_{v_1})\setminus\omega^\tau_{v_2})\setminus\ldots)\setminus\omega^\tau_{v_n})}{(-1)^{n-1}}=1.
\]
With this in mind, applying~\eqref{eq:surg_M} at every $v$ we can rewrite~\eqref{eq:at_constant} as
\begin{align*}
\sum_{\mcE:\, \mcE_v\ni \eta(v)\,\forall v }  K(\mcE) \sum_{\gamma} \sum_{\tau\in S^*_{\bare}(\mcE;\,\eta,\,\sigma,\,\gamma)} \left(\prod_{v\in V} \frac{1}{2d} \gamma^2(v) \sign(\omega^\tau_v)\right)  \prod_{v\in V} \1_{\{\mcE_v \ni \gamma\eta(v)\}}  \prod_{f\in \mcE_v\setminus \{\eta(v)\}} \overline M^{\gamma}\left(f,\,\tau(f)\right).
\end{align*}
Furthermore, recall that, given $\gamma(v)$, $\omega_v^\tau(\gamma\eta(v))=\tau(\eta(v))$, which means that now the dependence on $\tau$ is only through $\omega_v^\tau$ and $\gamma(v)$. This, together with Lemma~\ref{lem:bij_tau}, allows us to obtain
\begin{align}\label{eq:surg_fac}
\sum_{\mcE:\, \mcE\ni \eta(V) } & K(\mcE) \sum_{\gamma} \sum_{\omega_v\in S(\mcE_{v}\setminus\{\eta(v)\}), \, v\in V}\left(\prod_{v\in V} \frac{1}{2d} \sign(\omega_v)\right)\prod_{v\in V}\1_{\{\mcE_v \ni \gamma\eta(v)\}}  \prod_{f\in \mcE_v\setminus \{\eta(v)\}} \overline M^{\gamma}\left(f,\,\tau(f)\right). 
\end{align}
At this point, we note that the expression above does not depend on $v$, $\sigma$ or $\eta$ anymore.
In fact, as $\omega_v(f)^-=f^{-}=v$, we have that $\overline{M}\left(f,\,\omega_v(f)\right)$ is a constant by definition (see~\eqref{eq:def_bar_M}). 
Therefore, without loss of generality, we can take $v=o$, $\eta(v)=e_1$ to get that \eqref{eq:surg_fac} is equal to the $n$-th power of
\begin{align}
\frac{1}{2d}\sum_{\mcE_o:\, \mcE_o\ni e_1 }K(\mcE_o) \sum_{\gamma\in\{-1,\,1\}} &\left[\1_{\{\gamma=1\}} \sum_{\omega\in S(\mcE_o\setminus\{e_1\})}\sign(\omega) \prod_{f\in \mcE_o\setminus  \{e_1\}} \overline {M}^{\gamma}\left(f,\,\omega(f)\right)\right.\nonumber \\
&\left.+\,\1_{\{\gamma=-1\}}\1_{\{\mcE_o\ni -e_1\}} \sum_{\omega\in S(\mcE_o\setminus\{e_1\})}\sign(\omega)\, \prod_{f\in \mcE_o\setminus  \{e_1\}} \overline M^{\gamma}\left(f,\,\omega(f)\right) \right]. \label{eq:not_v_2}
\end{align}
Using the definition of determinant of a matrix, after applying the sum on $\gamma\in\{-1,\,1\}$ the first term in the square brackets above is equal to $\det\left(\overline M\right)$, while for $\gamma=-1$ the second one yields $\1_{\{\mcE_o\ni -e_1\}} \det\big(\overline M'\big)$, with $\overline M'$ as in~\eqref{eq:Mprime}. Summing these contributions we obtain
\begin{equation}\label{eq:final_k}
 \tilde\kappa_2(v_1,\,\dots,\,v_n)= - (C_d)^n \sum_{\sigma\in S_{\cycl}(V)} \sum_{\eta:\,V\to \{e_1,\,\ldots,\,e_d\}} \prod_{v\in V} \partial_{\eta(v)}^{(1)}\partial_{\eta(\sigma(v))}^{(2)} g_U\left(v,\, \sigma(v)\right),
\end{equation}
with
\begin{equation}\label{eq:Cd}
    C_d = \frac{1}{d} \sum_{\mcE_o:\, \mcE_o\ni e_1} (-1)^{|\mcE_o|}|\mcE_o| \left[ \det\left(\overline M\right)_{\mcE_o\setminus{\{e_1\}}} + \1_{\{\mcE_o\ni-e_1\}} \det\big(\overline M'\big)_{\mcE_o\setminus{\{e_1\}}}\right],
\end{equation}
where the factor $1/2$ was canceled from~\eqref{eq:not_v_2} since we now take in~\eqref{eq:final_k} the sum over $d$ directions $\eta$, introducing a multiplicity of $2$ for each point $v\in V$. Using~\eqref{eq:def_bar_M} this concludes the proof of Theorem~\ref{thm:main_cum3_cont}.\qedhere
\end{enumerate}
In the final part of the section we give the proofs for Lemmas~\ref{lem:bij_tau} and \ref{lem:surg_tau}.
\label{ref:pg_proof}
\begin{proof}[Proof of Lemma~\ref{lem:bij_tau}]
The fact that every $\tau$ induces a permutation on $\mcE_{v}\setminus\{\eta(v)\}$ follows from the definition~\eqref{eq:def_omega}.
For the converse, consider a permutation $\omega\in S(\mcE_{v}\setminus\{\eta(v)\})$. Using the fact that $\gamma$ is fixed in our assumptions, we can reconstruct $\tau$ locally, and in turn $\omega^\tau_v$, according to the values of $\gamma$. If $\gamma(v)=1$ then the only $\tau$ which satisfies $\omega_v^\tau=\omega$ in $\mcE_{v}\setminus\{\eta(v)\}$ is $\tau(f)=\omega(f)$. If instead $\gamma(v)=-1$, the only $\tau$ which satisfies $\omega_v^\tau=\omega$ is $\tau(f)=\omega(f)$ for $f\in \mcE_{v}\setminus\{\eta(v)\}$ and $\tau(\eta(v))=\omega(\gamma\eta(v)).$
\end{proof}
\begin{proof}[Proof of Lemma~\ref{lem:surg_tau}]
We will assume without loss of generality that $\gamma = -1$, as for $\gamma=1$ we have trivially that $\tau=\omega_v^\tau\circ(\tau\setminus \omega_v^\tau)$. For $\eta(v) \neq \gamma\eta(v)$, we expand the left-hand side of \eqref{eq:surg_M_original} to get
\begin{align*}
\prod_{f\in \mcE_v\setminus  \{\gamma\eta(v)\}} \overline M\left(f,\,\tau(f)\right) 
& = 
\overline M\left(\eta(v),\,\tau(\eta(v))\right) 
\prod_{f\in \mcE_v\setminus  \{\gamma\eta(v),\, \eta(v)\}} \overline{M} \left(f,\,\tau(f)\right) 
\\&  = 
\overline{M}\left(\eta(v),\,\omega^\tau_v(\gamma\eta(v))\right) \prod_{f\in \mcE_v\setminus \{\eta(v),\, \gamma\eta(v)\}} \overline M\left(f,\,\omega^\tau_v(f)\right)
\\&  = 
\frac{\overline{M}\left(\eta(v),\,\omega^\tau_v(\gamma\eta(v))\right)}{\overline{M}\left(\gamma\eta(v),\,\omega^\tau_v(\gamma\eta(v))\right)} \prod_{f\in \mcE_v\setminus \{\eta(v)\}} \overline M\left(f,\,\omega^\tau_v(f)\right).
\end{align*}
In the second line we used that $\omega^\tau_v (\gamma\eta(v))=\tau(\eta(v))$ and $\tau(f)=\omega^\tau_v(f)$ for all $f \in \mathcal{E}_v \setminus \{\eta(v)\}$. 
Expression \eqref{eq:surg_M} follows from \eqref{eq:surg_M_original} and the definition of $\overline{M}^\gamma$.

As for the signs, the permutation $(\tau\setminus \omega^\tau_v)\circ \omega^\tau_v$ can be written in terms of the decomposition (in transpositions) of $\tau$ by suppressing the transposition $(\eta(v),\, \tau(-\eta(v))$.
Therefore, their parities differ by a negative sign, and we get $\sign(\tau)=-\sign(\tau\setminus\omega^\tau_v)\sign(\omega_v^\tau)$.
\qedhere
\end{proof}

\subsection{Proof of Theorem \ref{thm:scaling}}
From \citet[Lemma 5]{cipriani2022properties} we know that
\begin{equation}\label{eq:recall_G}
    \left\lvert\nabla^{(1)}_{\delta_1} \nabla^{(2)}_{\delta_2} G_{U_\eps}(v_\eps,\,w_\eps)\right\rvert \leq C\cdot
    \begin{dcases}
        |v_\eps-w_\eps|^{-d} \ & \text{if } v_\eps\neq w_\eps , \\
        \hfil 1 & \text{if } v_\eps = w_\eps ,
    \end{dcases}
\end{equation}
for some $C=C(D)>0$, any $\delta_1,\,\delta_2 \in \mathcal E$ and $v_\eps,\,w_\eps\in D_\eps$, with $D\subseteq U$ such that $\dist (D,\,\partial U) > 0$. The functions $f_1,\,\dots,\,f_n$ are bounded and have disjoint and compact supports, respectively referred to as $D_1,\,\dots,\,D_n$. Let
\[
\mathfrak D\coloneqq\min_{i\neq j\in[n]}\dist(D_i,\,D_j)>0,\qquad \mathfrak F\coloneqq\max_{i\in[n]} (\sup_{x\in D}\|f_i(x)\| )<\infty.
\]
Now, the integrand in~\eqref{eq:thm:scaling}, namely
\[
    \tilde\kappa_1\left(\floor{x_1/\eps},\,\dots,\,\floor{x_n/\eps}\right) \prod_{i\in[n]} f_i(x_i) ,
\]
can be bounded by~\eqref{eq:recall_G} by some constant multiple of

\begin{equation*}
   \left( \max_{\substack{x_1\in D_1, \,\ldots,\,x_n\in D_n}} \big\lvert\floor{x_n/\eps} - \floor{x_1/\eps}\big\rvert^{-d} \prod_{i=1}^{n-1} \big\lvert\floor{x_i/\eps} - \floor{x_{i+1}/\eps}\big\rvert^{-d}\right)\mathfrak{F}^n \le\eps^{dn}\mathfrak{D}^{-dn}\mathfrak{F}^n
\end{equation*}
whenever $\eps$ is small enough (so that $\floor{x_i/\eps} \neq \floor{x_j/\eps}$ for all $i\neq j$). Using dominated convergence we can introduce the limit inside the integral in the left-hand side of~\eqref{eq:thm:scaling}, obtaining the desired result from~\eqref{eq:cum_limit}.
\qedhere

\section{Towards universality}\label{sec:towards}

In this section, we will discuss generalizations of our results regarding the cumulants and the scaling limits for the height-one and degree fields to other lattices. This indicates universal behavior of those two fields. In particular, we will prove the analogous results of Corollary~\ref{cor:degree_field} and Theorem~\ref{thm:main_cum3_discrete} for the triangular lattice. 

We will mostly focus on the differences of the proofs and discuss the key assumptions that we believe to be sufficient to extend our results to certain general families of graphs. In particular, all these assumptions also apply to the hexagonal lattice. This will be further discussed in Section~\ref{subsec:general_graph}.

\subsection{Triangular lattice}

Let us first define the coordinate directions of the triangular lattice, that is,
\begin{equation*}
	\tilde{e}_j=\left(\cos\left(\frac{\pi (j-1)}{3}\right),\, \sin\left(\frac{\pi (j-1)}{3}\right)\right), \quad j = 1,\,\dots,\,6.
\end{equation*}
Notice that $\tilde{e}_{j+3}=-\tilde{e}_j$ for $j = 1,\,2,\,3$.
We also consider the set of directions
\begin{equation*}
	\tilde{E}_o\coloneqq \{\tilde{e}_j:\, j =1,\,\dots,\,6\}.
\end{equation*}
Similarly to the hypercubic lattice case, we will use $\tilde{E}_o$ to denote the directed edges leaving the origin as well as the undirected edges containing the origin, or simply the unit vectors in their respective directions.
The set $\tilde{E}_v$ will be analogously defined as the set of edges incident to a site $v$ and $\tilde{E}(V) \coloneqq \cup_{v \in V} \tilde{E}_v$. 

The triangular lattice in dimension $2$ (see Figure~\ref{fig:triangular_lattice}) is then given by
\begin{equation*}
	\mathbf{T} \coloneqq \left\{ a_1 \tilde{e}_1 +a_2 \tilde{e}_2 :\, a_1,\,a_2 \in \mathbb{Z}\right\}.
\end{equation*}
\begin{figure}[ht!]
    \centering
    \includegraphics[scale=.8]{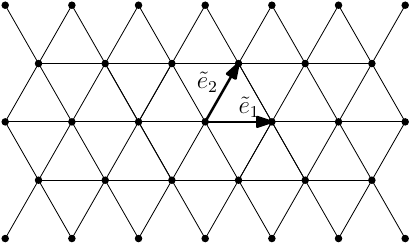}
    \caption{Triangular lattice and its spanning vectors $\tilde e_1$ and $\tilde e_2$.}
    \label{fig:triangular_lattice}
\end{figure}

For any finite connected set $\Lambda \subseteq \mathbf{T}$ we can define the discrete Laplacian $\Delta_\Lambda$ analogously to \eqref{eq:Delta_Lam}, with $\Delta_\Lambda(u, v) = -6$ if $u=v$, for $u,\,v \in \Lambda$.
Likewise, we can define the Green's function (or potential kernel) via Definition~\ref{def:Green-function-discrete}.
We can also extend the notion of good set of points Definition~\ref{def:goodset} to subsets of the triangular lattice by simply using its graph distance, and we can define both the UST and the ASM in $\Lambda \subseteq \mathbf{T}$, noticing that the burning algorithm and Lemma~\ref{lem:height1} still hold in this lattice. The fGFF with Dirichlet boundary conditions can be defined again by simply using the graph Laplacian of $\Lambda \subseteq \mathbf{T}$. Trivially, Lemma~\ref{lem:equiv-fggs} still holds in this setting.

The fermionic observables can be taken as
\begin{equation*}
	X_v:= 
    \frac{1}{\deg_{\mathbf{T}}(v) }\sum_{e:\, e^-=v}\nabla_{e}\psi(v)\nabla_{e}\bpsi(v),\quad v\in \Lambda,
\end{equation*}
and
\begin{equation*}
	Y_v\coloneqq \prod_{e:\, e^-=v}\left(1-\nabla_{e}\psi(v)\nabla_{e}\bpsi(v)\right),\quad v\in \Lambda.
\end{equation*}
It follows that if the set $V \subseteq \Lambda \subseteq \mathbf{T}$ is good, then trivially Equation~\eqref{eq:ferm-degree} and Theorem~\ref{thm:height-fgff} (with Dirichlet boundary conditions, rather than pinned) still hold. In fact, Theorem~\ref{thm:height-fgff} is valid for any finite subset of a translation invariant graph.

We will now state the results for cumulants of the degree field and the height-one field on the triangular lattice $\mathbf{T}$. Recall the definition of the average degree field of the UST in Corollary~\ref{cor:degree_field}. 

\begin{theorem}\label{thm:trig-discrete}$ $
\begin{enumerate}
\item\label{item1_cumtri_dist} Let $n \in \mathbb{N}$ and let the set of points $V\coloneqq \left\{v_1,\,\dots,\,v_n\right\}\subseteq \Lambda$ be a good set.
The  joint cumulants of the average degree field of the UST with wired boundary conditions $(\mathcal{X}_v)_{v\in V}$ are given by %
\begin{align}\label{eq:trig-UST-disc} 
	\nonumber
	\kappa_\Lambda^{\mathbf 0}\left({\mathcal{X}}_v:\,v\in V\right) 
& = 
	\kappa_\Lambda^{\mathbf 0}\left(X_v:\,v\in V\right)
\\	&=
  - \left (\frac{-1}{6} \right )^n\sum_{\sigma\in S_{\cycl}(V)} \sum_{\eta:\, V\to \tilde{E}_o} \prod_{v\in V} \nabla_{\eta(v)}^{(1)}\nabla_{\eta(\sigma(v))}^{(2)} G_\Lambda\left(v,\, \sigma(v)\right).
\end{align}

\item\label{item2_cumtri_dist} Let $n\geq 1$, $V\coloneqq \left\{v_1,\,\dots,\,v_n\right\}\subseteq \Lambda^{\mathrm{in}}$ be given. For a set of edges $\mcE\subseteq \tilde{E}(V)$ and $v\in V$ denote $\mcE_v\coloneqq \{f\in\mcE:\,f^- = v\}\subseteq \tilde{E}_v$. The $n$-th joint cumulants of the field $(X_vY_v)_{v\in V}$ are given by
\begin{align}\label{eq:trig-ASM-disc}
    \kappa_\Lambda\left(h_\Lambda(v):\,v\in V\right) 
	&= 
    \kappa_\Lambda^{\mathbf 0}\left(X_v Y_v:\,v\in V\right) 
	\\&= 
	\left(\frac{-1}{6}\right)^n
	\sum_{\mcE\subseteq \tilde{E}(V):\, |\mcE_v|\ge 1 \ \forall v} K(\mcE) 
	\sum_{\tau\in S_{\co}(\mcE)} \sign(\tau) \prod_{f\in \mcE} M\left(f,\,\tau(f)\right)
\end{align}
where $K(\mcE)\coloneqq\prod_{v\in V}K(\mcE_v)$ and $K(\mcE_v)\coloneqq (-1)^{|\mcE_v|}|\mcE_v|$.
\end{enumerate}
\end{theorem}

\noindent
In the above, we use the notation $\nabla_{\tilde{e}_j } f(v) \coloneqq f(v+\tilde{e}_j )-f(v)$ for $v \in \mathbf{T}$ and $j \in \{1,\,\dots,\,6\}$. 
The proof of Theorem~\ref{thm:trig-discrete} follows in an analogous way to the proof of Corollary~\ref{cor:degree_field}, item~\ref{item1_corDF} and the proof of Theorem~\ref{thm:main_cum3_discrete}, hence we will skip it. In the following, we will compute the scaling limits.

\begin{theorem}\label{thm:trig-cont} Let $n \ge 2$ and $V \coloneqq \{v_1,\,\dots,\,v_n\} \subseteq U$ a good set such $\dist(V,\,\partial U) >0$, where $U\subset \R^2$ is smooth, connected and bounded. 
\begin{enumerate}
\item\label{item1_limit_triang} For the average degree field $(\mathcal{X}_v)_{v\in \Lambda}$ of the UST, we have that
\begin{align}\label{eq:cum_limit_trig}
	\lim_{\eps\to 0} \eps^{-2n} \kappa_\Lambda^{\mathbf 0}\left(\mathcal{X}_v^\eps:\,v\in V\right) =  
	- \left(-\frac{1}{2}\right)^n  \sum_{\sigma\in S_{\cycl}(V)} \sum_{\eta:\,V\to \{e_1,\,e_2\}} \prod_{v\in V} \partial_{\eta(v)}^{(1)}\partial_{\eta(\sigma(v))}^{(2)} g_U\left(v,\, \sigma(v)\right).
\end{align}

\item\label{item2_limit_triang} For the ASM height-one field $(h^\varepsilon_v)_{v\in \Lambda}$, we have that 
\begin{align*}
    \lim_{\eps\to 0} \eps^{-2n} \kappa_\Lambda\left(h_v^\eps:\,v\in V\right) 
& =
    \lim_{\eps\to 0} \eps^{-2n} \kappa_\Lambda^{\mathbf 0}\left(X_v^\eps \, Y_v^\eps:\,v\in V\right) 
\\& = - (C_\mathbf{T})^n \sum_{\sigma\in S_{\cycl}(V)} \sum_{\eta:\,V\to \{e_1,\,e_2\}} \prod_{v\in V} \partial_{\eta(v)}^{(1)}\partial_{\eta(\sigma(v))}^{(2)} g_U\left(v,\, \sigma(v)\right),
\end{align*}
with
\begin{equation}\label{eq:const_triang}
    C_\mathbf{T} = -\frac{25}{36} + \frac{162}{\pi^4} - \frac{99 \sqrt{3}}{\pi^3} + \frac{99}{2 \pi^2} - \frac{5}{4\sqrt{3} \pi} \approx 0.2241 .
\end{equation}
\end{enumerate}
\end{theorem}

\noindent
Notice that the expression in \eqref{eq:cum_limit_trig} for the degree field appears with the same constant as in the square lattice, although this is not the case for the height-one field in~\eqref{eq:const_triang}. Moreover, the sum over the directions of derivation is over $e_1$ and $e_2$, and not $\tilde{e}_1$, $\tilde{e}_2$ and $\tilde{e}_3$. 

\begin{proof}[Proof of Theorem~\ref{thm:trig-cont} item~\ref{item1_limit_triang}]
Using Lemma~\ref{lem:cip4} for the triangular lattice (see \citet[Theorem 1]{kassel2015transfer}) and expression~\eqref{eq:trig-UST-disc}, we have 
\begin{multline*}
\lim_{ \varepsilon \to 0}-\eps^{2n}\left (\frac{-1}{6} \right )^n \sum_{\sigma\in S_{\cycl}(V)} \sum_{\eta:\, V\to \tilde{E}_o} \prod_{v\in V} \nabla_{\eta(v)}^{(1)}\nabla_{\eta(\sigma(v))}^{(2)} G_\Lambda\left(v,\, \sigma(v)\right)
\\
\begin{aligned}[t]
&= 
  -\left (\frac{-1}{6} \right )^n \sum_{\sigma\in S_{\cycl}(V)} \sum_{\eta:\, V\to \tilde{E}_o} \prod_{v\in V} \partial_{\eta(v)}^{(1)}\partial_{\eta(\sigma(v))}^{(2)} g_{U}\left(v,\, \sigma(v)\right)
\\
&
\begin{aligned}[t]
=
  -\left (\frac{-1}{6} \right )^n \sum_{\sigma\in S_{\cycl}(V)}
\sum_{\eta:\,V \setminus \{v_1\} \to \tilde{E}_o }
\left(\prod_{v\in V\setminus\{ v_1,\,\sigma^{-1}(v_1)\}} 
\partial_{\eta(v)}^{(1)}\partial_{\eta(\sigma(v))}^{(2)} g_{U}\left(v,\, \sigma(v)
\right)\right) \times
\\ 
\times
\sum_{\eta(v_1) \in \tilde{E}_o} 
\partial_{\eta(\sigma^{-1}(v_1))}^{(1)}\partial_{\eta(v_1)}^{(2)} g_{U}\left(v_1,\, \sigma(v_1)\right)
\partial_{\eta(v_1)}^{(1)}\partial_{\eta(\sigma(v_1))}^{(2)} g_{U}\left(v_1,\, \sigma(v_1)\right),
\end{aligned}
\end{aligned}
\end{multline*}
where in the last expression we simply isolated the factors that depend on $\eta(v_1)$.
For any two differentiable functions $f_1,\,f_2:\, U \to \mathbb{R}$, and $\eta = (\eta_1,\,\eta_2) = (\cos((j-1)\pi/3),\,\sin((j-1)\pi/3))$ with $j=1,\,\ldots,\,6$, we trivially have that
\begin{align}\label{eq:isotropic}
\sum_{\eta \in \tilde{E}_o} 
\partial_{\eta} f_{1}(x,\,y)
\partial_{\eta} f_{2}(x,\,y)
= 3\sum_{\eta' \in \{e_1,e_2\}} 
\partial_{\eta'} f_{1}(x,\,y)
\partial_{\eta'} f_{2}(x,\,y),
\end{align}
where $\eta'$ is a canonical unit vector in $\R^2$.
By iterating and combining the last two expressions, we are able to change the sum from the directions $\tilde{E}(V)$ to $\{e_1,\,e_2\}$, the usual axis directions.
It follows that
\begin{multline*}
-\left (\frac{-1}{6} \right )^n\sum_{\sigma\in S_{\cycl}(V)} \sum_{\eta:\, V\to \tilde{E}_o} \prod_{v\in V} \partial_{\eta(v)}^{(1)}\partial_{\eta(\sigma(v))}^{(2)} g_{U}\left(v,\, \sigma(v)\right)
\\ = 
-\left (\frac{-1}{2} \right )^n\sum_{\sigma\in S_{\cycl}(V)} \sum_{\eta:\, V\to  \{e_1,\,e_2\}} \prod_{v\in V} \partial_{\eta(v)}^{(1)}\partial_{\eta(\sigma(v))}^{(2)} g_{U}\left(v,\, \sigma(v)\right),
\end{multline*}
which concludes the claim of the first statement.
\end{proof}

\begin{proof}[Proof of Theorem~\ref{thm:trig-cont} item~\ref{item2_limit_triang}] $ $

\begin{enumerate}[wide, labelindent=0pt, label={\bf Step \arabic*.}, ref=Step \arabic*]
\item\label{step1_triang}
Setting $Z_v$ as in the proof of Theorem~\ref{thm:main_cum3_discrete},  Theorem~\ref{thm:trig-discrete} item~\ref{item2_cumtri_dist} yields the expression
\[
	\kappa\left(Z^\eps_v:\, v \in V\right) 
= 
	\left(\frac{-1}{6}\right)^n
	\sum_{\mcE \subseteq \tilde{E}(V):\, |\mcE_{v_\eps}|\ge 1 \; \forall v_\eps} K(\mcE) 
	\sum_{\tau\in S_{\co}(\mcE)} \sign(\tau) \prod_{f\in \mcE} M\left(f,\,\tau(f)\right).
\]
The set $S_{\co}(\mathcal{E})$ is defined analogously to the square lattice case. Once again, using the equivalent of Lemma~\ref{lem:cip4} for the triangular lattice, we get that only bare permutations make contributions to the limiting expression,
that is
\begin{equation}\label{eq:bare_trig} 
	\lim_{\varepsilon \to 0}
	\varepsilon^{-2n}
	\kappa\left(Z^\eps_v:\, v \in V\right) 
	=
	\left(\frac{-1}{6}\right)^n
	\sum_{\mcE:\, |\mcE_{v}|\ge 1 \; \forall v}  K(\mcE) 
	\sum_{\tau\in S_{\bare}(\mcE)} \sign(\tau) \prod_{f\in \mcE} \overline{M}\left(f,\,\tau(f)\right).
\end{equation}
The definitions of $S_{\bare}(\mathcal{E})$ and $\overline{M}$ are also analogous to the square lattice case. We will again abuse notation here by referring to $\mcE$ as edges instead of directions of derivation.

\item\label{step2_triang}
As we will see, this step is more delicate than its square counterpart, as in the triangular lattice  we will have fewer cancellations.

Given $\tau \in S_{\bare}(\mcE)$, fix $v\in V$, and let $\eta(v)=\eta(v,\,\tau)$ be the edge through which $\tau$ enters $v$. Let  $\alpha(v) \in \{0,\,\dots,\,5\}$. We define $\eta^\alpha(v)$ as the edge through which $\tau$ exists $v$, where $\alpha(v){\pi}/{3}$ denotes the angle between the entry and exit edges. Let $\gamma_\alpha(v) \coloneqq \cos \left(\alpha(v){\pi}/{3}\right)$, 
 so that $\gamma_{\alpha}(v) \in \{-1,\,-1/2,\,0,\,1/2,\,1\}$ and overall
 \[
 \langle\eta(v),\,\eta^\alpha(v)\rangle=\|\eta(v)\|\,\|\eta^\alpha(v)\|\,\gamma_\alpha(v).
 \]
As benchmark recall that in the square lattice the angles between entry and exit edges are multiples of $\pi/2$, hence their cosines belong to $\{-1,\,0,\,1\}$.

Define $R_{v,\,\eta}:\, \mathbb{R}^2 \to \mathbb{R}^2$ to be the reflection on the line given by $\eta(v)$. We then define
\[
    \mathcal{E}' \coloneqq R_{v,\,\eta}(\mcE)\coloneqq \left(\bigcup_{v' \neq v} \mcE_{v'}\right) \cup \left\{R_{v,\eta}(e):\, e \in \mcE_v\right\}
\]
and, for $\tau \in S_{\bare}(\mathcal{E})$, define $\rho \in S_{\bare}(\mathcal{E}')$ as
\begin{equation*}
	\rho(e) 	
	=
	\begin{cases}
	\tau(e), &	 \text{ if } e \in \cup_{v' \neq v}\mathcal{E}_{v'},\\
	\tau(\eta^\alpha (v)), &	 \text{ if } e = R_{v,\,\eta}(\eta^\alpha (v)),\\
	R_{v,\,\eta}(\tau(e')), &	 \text{ if } e=R_{v,\,\eta}(e') \text{ for some } e'\in\mcE_v\setminus\{\eta^\alpha(v)\} .
	\end{cases} 
\end{equation*}
See Figure~\ref{fig:reflection_triang} for an example of the reflected permutation  $\rho$.
\begin{figure}[ht!]
    \centering
    \includegraphics[scale=.8]{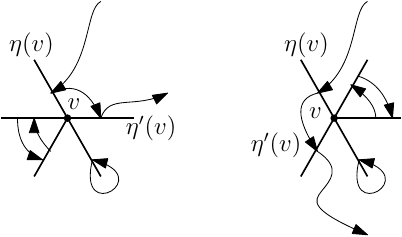}
    \caption{Left: a permutation $\tau$ on $v$. Right: its reflection $\rho$.}
    \label{fig:reflection_triang}
\end{figure}
We can then see that, once again, $K(\mathcal{E})=K(\mathcal{E}')$ and $\sign(\tau)=\sign(\rho)$.
Furthermore, with simple calculations of inner products we have
\begin{equation}\label{eq:cancel-trig-1} 
\overline{M}\left(\eta^\alpha(v),\,\eta(\sigma(v))\right)
+
\overline{M}\left(R_{v,\eta}(\eta^\alpha(v)),\,\eta(\sigma(v))\right)
=
2\cos\left(\frac{\alpha(v) \pi}{3}\right)\,\overline{M}(\eta(v),\,\eta(\sigma(v))).
\end{equation}
This equation will play the role of~\eqref{eq:cancel}, because now~\eqref{eq:bare_trig} is equal to (again we remind of our abuse of notation by using $\eta(v)$ to denote directions)
\begin{multline}\label{eq:cancel-trig-2} 
	\left(\frac{-1}{6}\right)^n
	\sum_{\mcE :\, |\mcE_v|\ge 1 \; \forall v}  K(\mcE)\sum_{\substack{\eta:\,V\to \tilde{E}(V)\\ \eta(v)\in \mcE_v\; \forall v}}\;\sum_{\sigma\in S_{\cycl}(V)} \sum_{\alpha:V\rightarrow \{0,\,\ldots,\,5\}} 
	\sum_{\tau\in S^*_{\bare}(\mcE,\eta,\sigma,\alpha)} \sign(\tau) \times
\\ \times 
\prod_{f\in \mcE \setminus \eta^{\alpha}(V)} \overline{M}\left(f,\,\tau(f)\right) \prod_{v\in V} \gamma_\alpha(v) 	\underbrace{\prod_{v\in V} 
	\partial^{(1)}_{\eta(v)}\partial^{(2)}_{\eta(\sigma(v))}
	g_{U}\left(\eta (v),\,\eta(\sigma(v))\right)}_{(\star)}.
\end{multline}
Here $S^*_{\bare}(\mcE;\,\eta,\,\sigma,\,\alpha)$ is the set of permutations compatible with $\mcE,\,\eta,\,\sigma,\,\alpha$ modulo the equivalence relation $$\tau\sim\rho\iff \rho=R_{v,\,\eta}(\tau).$$
Notice that $\gamma_\alpha(v) \neq 0$ since the corresponding term does not contribute to~\eqref{eq:cancel-trig-2}.

The factor $(\star)$, which accounts for the interactions between different points, only depends on the entry directions given by $\eta$, not on the exit directions $\eta^{\alpha}$.
This is the key cancellation to obtain expressions of the form \eqref{eq:thm:maincum3}, up to constant.

\item\label{step3_triang}

We rewrite expression \eqref{eq:cancel-trig-2} as
\begin{multline}\label{eq:first_lim_triang}
\left(\frac{-1}{6}\right)^n
\sum_{\substack{\eta:\,V\to \tilde{E}(V)\\ \eta(v)\in E_{v}\; \forall v}}\;
\sum_{\sigma\in S_{\cycl}(V)} 
\prod_{v\in V} 
\partial^{(1)}_{\eta(v)}\partial^{(2)}_{\eta(\sigma(v))}
	g_{U}\left(\eta (v),\,\eta(\sigma(v))\right)
\times \\ \times
\underbrace{
	\sum_{\mcE:\, \mcE \supseteq \eta(V)} 
	\sum_{\alpha:\,V\to \{0,\,\dots,\,5\}} 
			 \sum_{\tau\in S^*_{\bare}(\mcE;\,\eta,\,\sigma,\,\alpha)}\sign(\tau) 
			\prod_{v\in V}\left( K(\mcE_{v}) \gamma_\alpha(v) \prod_{f\in\mcE_v\setminus\{\eta^\alpha(v)\}}\overline{M}\left(f,\,\tau(f)\right)\right)}_{(\star\star)}.
\end{multline}
Remark that if $\eta^\alpha(v) \not \in {\mathcal{E}}$, the set  $S^*_{\bare}(\mcE;\,\eta,\,\sigma,\,\alpha)$ is empty, and therefore not contributing to the sum.

Notice that all entries of the type $\overline{M}(e,\,\tau(e))$ in $(\star \star)$ are discrete double gradients of the Green function of the full triangular lattice $\mathbf{T}$ (see Equation \eqref{eq:def_bar_M}).
In the following we will prove that $(\star \star)$ does not depend on the choice of $\eta$ nor $\sigma$.
The value of the term $(\star \star)$ will give the $n$-th power of the constant $C_\mathbf{T}$.

\item\label{step4_triang} 
Following the approach we used in the hypercubic lattice, we once again proceed with ``surgeries'' that will help us evaluate the local constant.

For this, given $\eta:\, V \to \tilde{E}(V)$, $\alpha: V \to \{0,\,\dots,\,5\}$, $\mathcal{E} \subseteq \tilde{E}(V)$ with $\eta(v),\, \eta^\alpha(v)  \in \mathcal{E}_v$, and $\tau \in S^*_{\bare}(\mcE;\,\eta,\,\sigma,\,\alpha)$, we define $\omega_v^\tau (\mathcal{E}_v \setminus \{\eta(v)\})$ and $\tau \setminus \omega^\tau_v ( (\mathcal{E} \setminus \mathcal{E}_v )\cup \{\eta(v)\})$  as

\begin{equation}\label{eq:def_omega_trig}
\omega^\tau_v(f)\coloneqq
\begin{cases}
    \tau(f)&\text{if }f\neq \eta^\alpha(v)\\
    \tau(\eta(v))&\text{if }f= \eta^{\alpha}(v),\,\alpha(v) \neq 0
\end{cases},\quad f\in \mcE_v\setminus\{\eta(v)\}
\end{equation}
and
\begin{equation*}\label{eq:def:tau_minus_trig}
\tau\setminus\omega^\tau_v(f)\coloneqq \begin{cases}
    \tau(f)&\text{if }f\notin \mcE_v\\
    \eta(\sigma(v))&\text{if }f= \eta(v)
\end{cases},\quad f\in (\mcE\setminus \mcE_v) \cup \{\eta(v)\}.
\end{equation*}

An example of $\omega_v^{\tau}$ can be found in Figure~\ref{fig:surgery_triangular}.
\begin{figure}[htb!]
    \centering
    \includegraphics[scale=.8]{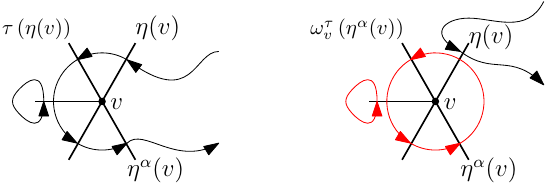}
    \caption{Left: a permutation $\tau$ at the point $v$. Right: the surgery applied to $\tau$, with $\omega^\tau_v$ denoted in red. }
    \label{fig:surgery_triangular}
\end{figure}

In this context, the analogous statement of Lemma~\ref{lem:bij_tau} still holds.
However, there is a subtle difference in the analogous statement of Lemma~\ref{lem:surg_tau}. 

\begin{lemma}[Surgery of $\tau$]\label{lem:surg_tau_trig}
Fix $v\in V$ and $\mcE$, $\eta$, $\sigma$, $\alpha$ as above.
Let $\tau$ be compatible with $\mcE$, $\eta$, $\sigma$ and $\alpha$.
Then 
\begin{equation*}
   \sign(\tau)=(-1)^{\1_{\{\alpha(v) \neq 0\}}}\sign(\tau\setminus\omega^\tau_v(f))\sign(\omega_v^\tau).
\end{equation*}
Furthermore let $\mcE_v$ be such that $\mcE \ni \eta(v)$. Then
\[
\prod_{f\in \mcE_{v}\setminus  \{\eta^\alpha(v)\}} \overline M\left(f,\,\tau(f)\right) = \frac{\overline{M}\left(\eta(v),\,\omega^\tau_v(\eta^\alpha(v))\right)}{\overline{M}\left(\eta^\alpha(v),\,\omega^\tau_v(\eta^\alpha(v))\right)}  \prod_{f\in \mcE_{v}\setminus \{\eta(v)\}} \overline M\left(f,\,\tau(f)\right).
\]
Equivalently, we can write that
\[
\prod_{f\in \mcE_v\setminus \{\eta^\alpha(v)\}} \overline M\left(f,\,\tau(f)\right) = \prod_{f\in \mcE_v\setminus  \{\eta(v)\}} \overline M^{\alpha}\left(f,\,\omega^\tau_v(f)\right),
\]
where for any $g\in \mcE_v$
\begin{equation}\label{eq:defmalp}
\overline M^{\alpha}\left(f,\,g\right) :=\begin{cases}
\overline{M}(\eta(v),\,g) & \text{ if } f=\eta^{\alpha}(v), \\
\overline{M}(f,\,g) & \text{ if } f\neq \eta^{\alpha}(v).
\end{cases}
\end{equation}
\end{lemma}

\noindent
Remark that the matrix $\overline{M}^{\alpha}$ is not symmetric anymore. Lemma~\ref{lem:surg_tau_trig} reads almost the same as its hypercubic counterpart Lemma~\ref{lem:surg_tau}, and its proof follows in the same manner, so it will be omitted. Using the translation invariance of $\overline{M}$, and setting $\eta(v)=\tilde{e}_1, \, \eta^{\alpha}(v)=\tilde{e}_{1+\alpha}, \, \omega^{\tau}_v =\omega$, it follows that $(\star\star)$ in \ref{step3_triang}
is equal to the $n$-th power of
\begin{align*}
\frac{1}{6}
\sum_{\mcE_o:\, \mcE_o\ni \tilde e_1 }K(\mcE_o) \sum_{\alpha = 0}^5 \biggl[ &\1_{\{\alpha=0\}} 
\sum_{\omega\in S(\mcE_o\setminus\{\tilde e_1\})}\sign(\omega) \prod_{f\in \mcE_o\setminus \{\tilde{e}_1\}} 
\overline M\left(f,\,\omega(f)\right) \nonumber\\ 
&- \gamma_{\alpha}\1_{\{\alpha \neq 0\}} \1_{\{\mcE_o \ni \tilde{e}_{1+\alpha}\}} \sum_{\omega\in S(\mcE_o\setminus\{\tilde e_1\})}\sign(\omega)\prod_{f\in \mcE_o \setminus \{\tilde{e}_1\}} \overline{M}^{\alpha}(f,\, \omega(f))\biggr].
\end{align*}
Using Equation~\ref{eq:isotropic}, we obtain the cumulants
\begin{equation}
- (C_\mathbf{T})^n \sum_{\sigma\in S_{\cycl}(V)} \sum_{\eta:\,V\to \{e_1,\,e_2\}} \prod_{v\in V} \partial_{\eta(v)}^{(1)}\partial_{\eta(\sigma(v))}^{(2)} g_U\left(v,\, \sigma(v)\right),
\end{equation}
with

\begin{equation}\label{eq:def-Ct}
    C_\mathbf{T} = \frac{1}{2} \sum_{\mcE_o:\, \mcE_o\ni \tilde e_1} (-1)^{|\mcE_o|}|\mcE_o| \left[\det\left(\overline M\right)_{\mcE_o\setminus{\{\tilde e_1\}}} - \sum_{\alpha=1}^5 \gamma_\alpha \1_{\{\mcE_o\ni \tilde{e}_{1+\alpha}\}} \det\big(\overline{M}^\alpha\big)_{\mcE_o\setminus\{\tilde e_1\}} \right].
\end{equation}
Plugging in the values of the potential kernel of the triangular lattice (see e.g. \cite{KenyonWilson} or \cite{RuelleTriang}), this concludes the proof.\qedhere
\end{enumerate}
\end{proof}

\begin{remark}
Note that, for the triangular lattice, we have that
\[
C_\mathbf{T} = \left( \frac{1}{18} + \frac{1}{\sqrt{3}\pi}\right)^{-1} \mathbb{P}(h(o)=1),
\]
where $\mathbb{P}(h(o)=1)$ was computed in \citet[Equation 4.3]{RuelleTriang}.
\end{remark}

\begin{remark}
As a safety check, expression~\eqref{eq:def-Ct} with $\gamma_\alpha = \cos\left(\alpha \pi/2\right)$ and $\alpha\in\{0,\,1,\,2,\,3\}$ yields the same value of the square lattice in $d=2$ as in~\eqref{eq:Cd}. 
\end{remark}

\subsection{General graphs}\label{subsec:general_graph}
\label{subsec:what-about-more-general-graphs-}

A natural question is whether our approach would work on general graphs $\mathcal{G}$ embedded in $\mathbb{R}^d$. 
In this section, we would like to  highlight the key ingredients we needed working on $\mathbb{Z}^d$ or $\mathbf{T}$ to prove our results.

\begin{enumerate}[wide, labelindent=0pt, label={\bf Ingredient \arabic* -- }, ref=Ingredient \arabic*]
\item \textbf{Matrix-Tree Theorem:}
as previously mentioned, Theorem~\ref{thm:height-fgff}, Theorem~\ref{thm:main_cum2} item~\ref{item1_cum2} and Theorem~\ref{thm:main_cum3_discrete} work on any finite graph for which the Matrix-Tree Theorem and the burning algorithm are valid, which includes subsets of any transitive, regular graph with bounded degree.

\item \textbf{Good approximation of the Green's function:}
for the scaling limits, we need that the equivalent of Lemma~\ref{lem:cip4} holds. For example, this is the case for graphs $\mathcal G$ such that the sequence $(\mathcal{G}_\eps)_\eps$ with $\mathcal G_\eps=\eps\mathcal G$ is a ``good approximation of $\mathbb{R}^d$'' in the sense of~\citet[Theorem 1]{kassel2015transfer}. They give a sufficient criterion to obtain such convergence, which in dimension $2$ includes transient isoradial graphs (and therefore the triangular and hexagonal lattices).

\item \textbf{Isotropic neighborhoods:}
we believe that the neighborhood of each vertex needs to be ``isotropic'', in the sense that for $v \in \mathcal{G}$, we need
\begin{equation*}
	\sum_{u:\, u \sim v} (u_i-v_i) (u_j-v_j)
	=
	c_{\mathcal{G}} \, \delta_{i,\,j}, \quad\forall\, i,\,j \in [d],
\end{equation*}
for some constant $c_{\mathcal{G}}$, $v=(v_1,\,\dots,\,v_d)$, $u=(u_1,\,\dots,\,u_d)$, and $\delta_{i,\,j}$ the Kronecker's delta function.
This is needed to substitute \eqref{eq:isotropic} and replace the reflection cancellations used in Lemma~\ref{lem:surg_tau_trig}.

With these ingredients in place, we believe that the scaling limit of the cumulants of the degree field of the UST should have the same form we obtained, up to constants that only depend on $c_{\mathcal{G}}$. The same applies to the height-one of the ASM field. However in this case the global constant in front of the cumulants would also depend on the values of the double discrete derivative of the Green's function on $\mathcal{G}$ in a neighborhood of the origin.
Furthermore, such a constant will be very similar to the expression for $C_\mathbf{T}$ given in \eqref{eq:def-Ct}. 

The reader might have noticed that all the conditions above are satisfied by the hexagonal lattice $\mathbf{H}$, on which the height-one of the ASM has also been studied (see \cite{RuelleTriang}).
The main difficulty for such a lattice is the lack of translation invariance, leading to the set of space ``directions'' depending on the points.
This means that the sum in $\eta$'s, say for example in \eqref{eq:limit_cum_X}, will depend on $\eps$, so that the convergence of the cumulants (either for the UST or the ASM) cannot be done as we perform in this article.
However with minor technical modifications to our proofs we should recover similar results. In fact the UST degree field should have the same global constant $c_{\mathbf H}=-1/2$ as in~\eqref{eq:cum_limit_trig}, whereas for the ASM $C_\mathbf{T}$ in~\eqref{eq:const_triang} should be replaced by $C_{\mathbf{H}}=1/8$.

Finally, we also expect that these results can be extended to other graphs that are not translation-invariant, although an extra site-dependent scaling might be necessary in this case.
\end{enumerate}

\section*{Funding and data availability statement}
LC and AR are funded by the grant OCENW.KLEIN.083 and WR by OCENW.KLEIN.083 and the Vidi grant VI.Vidi.213.112 from the Dutch Research Council. AC acknowledges the hospitality of Utrecht University and AR acknowledges the hospitality of UCL.

We do not analyze or generate any dataset since our work proceeds within a theoretical and mathematical approach.

\bibliographystyle{abbrvnat}
\bibliography{references}


\end{document}